\documentclass[11pt, twoside]{article}
\usepackage{mathrsfs}
\usepackage{amssymb}
\usepackage{amsmath}
\usepackage{mathrsfs}
\usepackage{amsthm}
\usepackage{amsfonts}
\usepackage{color}
\usepackage{latexsym}
\usepackage{txfonts}
\usepackage{indentfirst}
\usepackage{soul}
\usepackage[normalem]{ulem}
\usepackage{enumitem}

\allowdisplaybreaks

\def\red{\color{red}}

\def\rr{{\mathbb R}}
\def\rn{{\mathbb{R}^n}}

\def\nn{{\mathbb N}}
\def\zz{{\mathbb Z}}

\def\vi{\varphi}
\def\fz{\infty }

\def\lf{\left}
\def\r{\right}
\def\ls{\lesssim}

\def\wz{\widetilde}

\def\XXint#1#2#3{{\setbox0=\hbox{$#1{#2#3}{\int}$ }
		\vcenter{\hbox{$#2#3$ }}\kern-.6\wd0}}

\newcommand{\bbbr}{\mathbb R}

\newcommand{\bbbz}{\mathbb Z}

\DeclareMathOperator{\diam}{diam}

\newtheorem{theorem}{Theorem}[section]
\newtheorem{lemma}[theorem]{Lemma}

\newtheorem{proposition}[theorem]{Proposition}
\theoremstyle{definition}
\newtheorem{example}[theorem]{Example}
\newtheorem{remark}[theorem]{Remark}

\newtheorem{convention}[theorem]{Convention}
\renewcommand{\appendix}{\par
	\setcounter{section}{0}%
	\setcounter{subsection}{0}%
	\setcounter{subsubsection}{0}%
	\gdef\thesection{\@Alph\c@section}%
	\gdef\thesubsection{\@Alph\c@section.\@arabic\c@subsection}%
	\gdef\theHsection{\@Alph\c@section.}%
	\gdef\theHsubsection{\@Alph\c@section.\@arabic\c@subsection}%
	\csname appendixmore\endcsname
}

\numberwithin{equation}{section}

\def\mvint_#1{\mathchoice
	{\mathop{\vrule width 6pt height 3 pt depth -2.5pt
			\kern -9pt \intop}\limits_{\kern -3pt #1}}%
	{\mathop{\vrule width 5pt height 3 pt depth -2.6pt
			\kern -6pt \intop}\nolimits_{#1}}%
	{\mathop{\vrule width 5pt height 3 pt depth -2.6pt
			\kern -6pt \intop}\nolimits_{#1}}%
	{\mathop{\vrule width 5pt height 3 pt depth -2.6pt
			\kern -6pt \intop}\nolimits_{#1}}}

\title{\bf Optimal Embeddings for Triebel--Lizorkin and Besov Spaces on Quasi-Metric Measure Spaces
\footnotetext{\hspace{-0.35cm} 2020 {\it Mathematics Subject Classification}.
Primary 46E36, 46E35; Secondary  43A85, 42B35, 30L99.
\endgraf {\it Key words and phrases.}  lower measure bound, measure density condition, quasi-metric measure space, doubling measure, Sobolev space, Besov space, Triebel--Lizorkin space, Sobolev inequality, Sobolev-Poincar\'e inequality.
\endgraf This project is partially supported by the National
Key Research and Development Program of China (Grant No. 2020YFA0712900)
and the National Natural Science Foundation of China (Grant Nos. 11971058,
12071197, 12122102  and 11871100).}}
\author{Ryan Alvarado\footnote{Corresponding
author, E-mail: \texttt{rjalvarado@amherst.edu}/{\red February 12, 2022}/Final version.},\ \
Dachun Yang and Wen Yuan}
\date{}

\begin{document}
	
\maketitle
\vspace{-0.8cm}
	
\begin{center}
\begin{minipage}{13cm}
{\small {\bf Abstract.}\quad In this article, via certain lower bound conditions on the measures under consideration, the authors fully characterize the Sobolev  embeddings for the scales of
Haj\l asz--Triebel--Lizorkin and Haj\l asz--Besov spaces  in the general
context of quasi-metric measure spaces for an optimal range of the
smoothness parameter $s$. An interesting facet of this work is how the
range of $s$  for which the above characterizations of
these  embeddings  hold true is intimately linked (in a quantitative manner)
to the geometric makeup of the underlying space. Moreover, although stated for
Haj\l asz--Triebel--Lizorkin and Haj\l asz--Besov spaces   in the context of quasi-metric spaces,
the main results in this article  improve  known work even for Sobolev spaces in
the metric setting.}
\end{minipage}
\end{center}

\tableofcontents	
\vspace{0.5cm}

\section{Introduction}
The Sobolev embedding theorem for the classical Sobolev space $W^{1,p}(\Omega)$ on a domain $\Omega\subset \rn$  is
an indispensable tool in analysis with far-reaching applications and it highlights the fact
that functions belonging to $W^{1,p}(\Omega)$ possess certain nice properties
depending on if $p\in[1,n)$, $p=n$, or $p\in(n,\fz)$; see \cite{sob36,sob38,gag,Niren,trud,ad}.
As expected, there is significant interplay between the geometrical properties of the underlying domain $\Omega$ and the availability of these embeddings. In fact, for domains $\Omega$ with a sufficiently regular  boundary,
it is well known that the  Sobolev embedding  $W^{1,p}(\Omega)
\hookrightarrow L^{\frac{np}{n-p}}(\Omega)$  with $p\in[1,n)$
holds true if and only if $\Omega$ satisfies the so-called measure density condition, that is,
there exists a positive constant $C$ such that, for any $x\in\Omega$ and $r\in(0,1]$, one has
\begin{equation}
\label{eq22-1}
|B(x,r)\cap \Omega|\ge C r^n,
\end{equation}
where $|B(x,r)\cap \Omega|$
denotes the $n$-dimensional Lebesgue measure of $B(x,r)\cap \Omega$,
and $B(x,r)$ the ball in $\rn$ centered at $x$ with the radius $r$;
see, for instance, \cite{ad,hajlaszkt1}.  For a related
result regarding Slobodeckij--Sobolev spaces, we refer the reader to \cite{z15}. From a geometric perspective, the measure density condition implies, among other things, that $\Omega$ is, in a sense, ``thick" nearby its boundary, and it is easy to see from definitions that the category of domains satisfying \eqref{eq22-1} encompasses a wide variety of environments that appear in many branches of mathematics, including the classes of Lipschitz, John, and $(\varepsilon,\delta)$ domains, just to name a few (see \cite{EG,MS,Jon81}).

One significant development in the theory of Sobolev spaces emerged in the 1990s with the introduction of Sobolev spaces defined on environments much more general than the Euclidean ambient, namely, on metric measure spaces \cite{hajlasz2,cheeger,SMP,SMP2,shanmugalingam}. In this more general context, if one assumes a lower bound for the growth of the measure $\mu$, in the sense that
\begin{equation}
\label{eq22}
\mu(B(x,r))\geq \kappa r^Q,
\end{equation}
where $Q\in(0,\fz)$ and $\kappa$ is a positive constant independent of $x$ and $r$, then a Sobolev embedding theorem holds true for the Haj\l asz--Sobolev spaces $M^{1,p}$ (introduced by Haj\l asz in \cite{hajlasz2}), where  the nature of the embeddings depends on if $p\in(0,Q)$, $p=Q$, or $p\in(Q,\fz)$ (see \cite{hajlasz,hajlasz2,agh20}). Thus, this exponent $Q$ gives the counterpart of the dimension of the Euclidean
space $\mathbb{R}^n$. It was recently shown in \cite{agh20} the lower measure bound  \eqref{eq22} is actually equivalent
to the existence of the Sobolev embeddings for $M^{1,p}$.
We also refer the reader to \cite{gorka,hajlaszkt1,hajlaszkt2,hhhpl21,hebey,HK21,Karak1,korobenko,korobenkomr,z15}
for  partial or related results.

The main purpose of this article is to extend the work in \cite{agh20} through the consideration of both a more general geometric setting and a more general scale of function spaces that naturally contain $M^{1,p}$ spaces;
more precisely,  we aim to characterize the Sobolev embeddings for
the  Haj\l asz--Besov and  Haj\l asz--Triebel--Lizorkin spaces in the  more general setting of quasi-metric measure spaces via certain lower bound conditions on the  measures under
consideration.
A distinguishing feature of our work is that we obtain such a generalization without compromising the quantitative aspects of the theory which, in turn, permits us to improve known results even for $M^{1,p}$ spaces in the metric setting. We will expand more on this aspect below. Further applications of the optimal embeddings obtained in this work to Triebel--Lizorkin and Besov extension domains are given in
the forthcoming article \cite{AYY21}.

Besov and Triebel--Lizorkin spaces provide natural scales of spaces that include a number of function spaces used in analysis such as Lebesgue spaces, Hardy spaces, Sobolev spaces, H\"older spaces, and BMO, and they have retained their significance to this very day, playing an important role in both theoretical and applied branches of mathematics. We refer the reader to \cite{T92,T06} for a detailed introduction to these spaces. There have been several approaches to defining Besov and Triebel--Lizorkin spaces on quasi-metric spaces over the years \cite{KYZ11,HaLuYa99i, HaLuYa99ii, HaYa02, HaYa03, HaMuYa08}; however, an important virtue that the Haj\l asz--Besov and Haj\l asz--Triebel--Lizorkin spaces (defined in \cite{KYZ11}) possess over other definitions, is that there is a fruitful theory for this particular brand of spaces without needing to assume that the space is connected or that the measure is doubling.

Turning first to the geometric and measure theoretic considerations in this work, the pair $(X,\rho)$ shall
be called a \emph{quasi-metric space}
if $X$ is a set of cardinality at least 2 and $\rho$ is a quasi-metric, namely,
a nondegenerate, quasi-symmetric nonnegative function defined on
$X\times X$ satisfying the following \emph{quasi-subadditivity condition}
\begin{equation}\label{TR-ineq.2}
\rho(x,y)\leq C\max \left\{\rho(x,z),\rho(z,y)\right\},\qquad\forall\ x,\,y,\,z\in X,
\end{equation}
which is equivalent to the, perhaps, more standard \emph{quasi-triangle inequality}
$$\rho(x,y)\leq C'\left[\rho(x,z)+\rho(z,y)\right],\quad\forall\ x,\,y,\,z\in X,$$
where $C$ and $C'$ are positive constants independent of $x$, $y$, and $z$.
In this context, let $\mu$ be a nonnegative Borel measure on $(X,\rho)$ such that, for all $x\in X$ and   $r\in (0,\fz)$, the ball  $B_\rho(x,r)$ is $\mu$-measurable and
$\mu(B_\rho(x,r))\in(0,\fz)$, where
$B_\rho(x,r):=\{y\in X:\, \rho(x,y)<r\}$ (which is called a \emph{$\rho$-ball}). Under these assumptions, the triplet $(X,\rho,\mu)$ shall be called a \textit{quasi-metric measure space}.

It is important to stress that if we limited the scope of our work
to only \emph{genuine metrics} satisfying the standard triangle inequality
\begin{equation}\label{TR-ineq}
\rho(x,y)\leq\rho(x,z)+\rho(z,y),\qquad\forall\ x,\,y,\,z\in X,
\end{equation}
as opposed to the more general quasi-subadditivity condition \eqref{TR-ineq.2}, then we would not be able to establish the full strengths of our results, as the triangle inequality fails to fully capture the various subtleties within the very class of metrics. For example, ultrametrics are metrics satisfying \eqref{TR-ineq.2} with $C=1$ which is a strictly stronger inequality than \eqref{TR-ineq}. Simply put, some metrics are better behaved than others and the inability to detect these finer qualities would lead to a suboptimal function space theory, even in the metric setting. From this perspective, it is necessary for us to work in the general setting of quasi-metric spaces using \eqref{TR-ineq.2}. We shall return to this point later. More details on the underlying spaces are listed in Section \ref{s-set} below.

As concerns the analytical assumptions, instead of $M^{1,p}(X)$ considered in \cite{agh20} (on metric spaces),
in this article, we  consider the more general scale of Haj\l asz--Besov spaces $\dot{N}^s_{p,q}(X)$
and  Haj\l asz--Triebel--Lizorkin spaces $\dot{M}^s_{p,q}(X)$ on quasi-metric measure spaces $(X,\rho,\mu)$
with $s,\,p\in(0,\fz)$ and
$q\in(0,\infty]$. The definitions and some basic properties of these spaces can be
found in Section \ref{s-func} below.
In particular, it is known that $M^{1,p}(X)$ coincides with $\dot{M}^1_{p,\fz}(X) \cap L^p(X)$
for any $p\in(0,\fz)$.

The main results of this article give the equivalence between the Sobolev embeddings of the Haj\l asz--Besov spaces $\dot{N}^s_{p,q}(X)$
and the Haj\l asz--Triebel--Lizorkin spaces $\dot{M}^s_{p,q}(X)$ (and, in particular, the fractional
Haj\l asz--Sobolev spaces $\dot{M}^{s,p}(X)=\dot{M}^s_{p,\fz}(X)$),
and  certain lower bound conditions on the underlying measure.
These
results are optimal regarding the smoothness parameter $s$, where the optimality
is related to the nature of the ``best" quasi-metric on $X$ which is bi-Lipschitz equivalent to $\rho$. The latter notion is quantified via the following ``index":
\begin{equation}
\label{index-INT}
{\rm ind}\,(X,\rho):=\sup_{\varrho\approx\rho}\left(\log_2\left[\sup_{\substack{x,\,y,\,z\in X\\\mbox{\scriptsize{not all equal}}}}
\frac{\varrho(x,y)}{\max\{\varrho(x,z),\varrho(z,y)\}}\right]\right)^{-1}
\in(0,\infty],
\end{equation}
where the first supremum is taken over all
quasi-metrics $\varrho$ on $X$ which are  bi-Lipschitz equivalent to $\rho$,
and the second supremum is taken over all points $x,\,y,\,z$ in $X$ which are not all
equal.
This index was introduced in \cite{MMMM13} and its value reflects the
geometry of the underlying space, as evidenced by the following examples
(see Section~\ref{section:preliminaries} for more examples highlighting this fact):
\begin{itemize}[itemsep=1pt]
\item {${\rm ind}\,(\mathbb{R}^n,|\cdot-\cdot|)=1$ and  ${\rm ind}\,([0,1]^n,|\cdot-\cdot|)=1$,
where $|\cdot-\cdot|$ denotes the Euclidean distance;}

\item
{${\rm ind}\,(X,\rho)\geq 1$ if there exists a genuine distance on $X$ which is
pointwise equivalent to $\rho$;}

\item {$(X,\rho)$ cannot be bi-Lipschitzly embedded into some ${\mathbb{R}}^n$ with
$n\in{\mathbb{N}}$, whenever ${\rm ind}\,(X,\rho)<1$;}

\item {${\rm ind}\,(X,\rho)=\infty$ if there exists an ultrametric
on $X$ which is pointwise equivalent to $\rho$;}

\item {${\rm ind}\,(X,\rho)=1$ if  $(X,\rho)$ is a metric space that is equipped with a doubling measure and supports a weak $(1, p)$-Poincar\'e inequality with $p>1$.}
\end{itemize}

To facilitate the statement of the main theorems in this work, we make a few notational conventions: First, given a quasi-metric space $(X,\rho)$ and  fixed numbers $s\in(0,\infty)$ and $q\in(0,\infty]$, we will understand by $s\preceq_q{\rm ind}\,(X,\rho)$ that $s\leq{\rm ind}\,(X,\rho)$ and that the value $s={\rm ind}\,(X,\rho)$ is only permissible when $q=\infty$ and the supremum defining the number ${\rm ind}\,(X,\rho)$ in \eqref{index-INT} is attained. Secondly, we let $C_\rho\in[1,\infty)$ denote the least constant playing the role of $C$ in \eqref{TR-ineq.2}; see \eqref{C-RHO.111} for a more formal definition. Moreover, in what follows,  for any measurable set $E\subset X$ with $\mu(E)\in(0,\fz)$ and any measurable
function $u$,  let $\mvint_{E}:=\frac1{\mu(E)}\int_E$  and $u_E:=\mvint_E u\,d\mu$, whenever the integral is well defined. Also, recall that a quasi-metric space $(X,\rho)$ is said to be {\it uniformly} {\it perfect}
if there exists a constant $\lambda\in(0,1)$ such that, for any $x\in X$ and  $r\in(0,\infty)$,
$$
B_\rho(x,r)\setminus B_\rho(x,\lambda r)\neq\emptyset\quad
\mbox{ whenever }\quad X\setminus B_\rho(x,r)\neq\emptyset.
$$
Note that every connected quasi-metric space is uniformly perfect;
however, there are very disconnected Cantor-type sets that are also uniformly perfect.

The following first main result of this article establishes the
fact that local embeddings for $\dot{M}^s_{p,q}$ and $\dot{N}^s_{p,q}$ spaces
are \textit{equivalent} to the following \textit{lower $Q$-Ahlfors-regular condition} on  the underlying measure:
\begin{equation}\label{lowermeasure-intro}
\kappa\,r^Q\leq\mu(B_\rho(x,r))\qquad \mbox{for any }\ x\in X\
\mbox{and any finite}\ r\in(0,{\rm diam}_\rho(X)],
\end{equation}
where $Q\in(0,\fz)$ and $\kappa$ is a positive constant independent of $x$ and $r$; see also
Theorem~\ref{LBembedding} and Theorem~\ref{LMeasINT} below.
Here and thereafter, ${\rm diam}_\rho(X):=\sup_{x,\,y\in X}\rho(x,y)$, and for any ball $B$ of radius $r_B$,
and any positive number $\sigma$, $\sigma B$ is the ball with radius $\sigma r_B$ and the same center as $B$.

\begin{theorem}\label{LMeasINTCor}
Suppose that $(X,\rho,\mu)$ is a uniformly perfect quasi-metric measure space and fix parameters $q\in(0,\infty]$,  $Q\in(0,\infty)$, and $\sigma\in[C_\rho,\infty)$. Also, assume that $s\in(0,\infty)$ satisfies $s\preceq_q{\rm ind}\,(X,\rho)$. Then the following statements are equivalent.
\begin{enumerate}[label=\rm{(\alph*)}]
\item The measure $\mu$ is lower $Q$-Ahlfors-regular on $X$ $($in the sense of \eqref{lowermeasure-intro}$)$.

\item There exist a $p\in(0,Q/s)$ and a $C_S\in(0,\infty)$ such that, for any ball
$B:=B_\rho(x,r)$ with $x\in X$ and finite $r\in(0,{\rm diam}_\rho(X)]$,  one has
\begin{equation*}
\left(\, \mvint_{B} |u|^{p^*}\, d\mu\right)^{1/p^*}\leq C_Sr^{-Q/p}\left[
r^s\Vert u\Vert_{\dot{M}^s_{p,q}(\sigma B)}
+\Vert u\Vert_{L^{p}(\sigma B)}\right],
\end{equation*}
whenever $u\in \dot{M}^s_{p,q}(\sigma B,\rho,\mu)$. Here,  $p^*:=Qp/(Q-sp)$.
		
\item There exist a $p\in(0,Q/s)$ and a $C_P\in(0,\infty)$ such that, for any ball
$B:=B_\rho(x,r)$ with $x\in X$ and finite $r\in(0,{\rm diam}_\rho(X)]$, one has
\begin{equation*}
\label{HHs-175-Cor}
\inf_{\gamma\in\mathbb{R}}\left(\, \mvint_{B} |u-\gamma|^{p^*}\, d\mu\right)^{1/p^*}\leq
C_Pr^{s-Q/p}\Vert u\Vert_{\dot{M}^s_{p,q}(\sigma B)},
\end{equation*}
whenever $u\in\dot{M}^s_{p,q}(\sigma B,\rho,\mu)$.
Here,  $p^*:=Qp/(Q-sp)$.
	
\item There exist positive constants $c_1$, $c_2$, and $\omega$ such that
\begin{eqnarray*}
\mvint_{B} {\rm exp}\left(c_1\frac{|u-u_{B}|}{\Vert u\Vert_{\dot{M}^s_{Q/s,q}(\sigma B)}}\right)^{\omega}\,d\mu\leq c_2,
\end{eqnarray*}
whenever $B\subset X$ is a ball $($with radius at most ${\rm diam}_\rho(X)$$)$ and $u\in\dot{M}^s_{Q/s,q}(\sigma B,\rho,\mu)$ with $\Vert u\Vert_{\dot{M}^s_{Q/s,q}(\sigma B)}>0$.

\item There exist a $p\in(Q/s,\infty)$ and a $C_H\in(0,\infty)$
such that every $u\in \dot{M}^s_{p,q}(X,\rho,\mu)$ has a H\"older continuous representative of
order $s-Q/p$ on $X$, denoted by $u$ again,  satisfying
\begin{eqnarray*}
|u(x)-u(y)|\leq C_H\,[\rho(x,y)]^{s-Q/p}\Vert u\Vert_{\dot{M}^s_{p,q}(X)},\qquad\forall\ x,\,y\in X.
\end{eqnarray*}
\end{enumerate}
In addition, if $q\leq p$, then all of the statements above continue to be equivalent with  $\dot{M}^s_{p,q}$ replaced by $\dot{N}^s_{p,q}$.
\end{theorem}

In fact, a slightly stronger version Theorem~\ref{LMeasINTCor} follows from Theorem~\ref{LBembedding} and Theorem~\ref{LMeasINT} below. More precisely, we show that if just one of the $\dot{M}^s_{p,q}$ embeddings in (b)-(e) (or their $\dot{N}^s_{p,q}$ variants) holds true for some $p,q,s$, and $\sigma$ (in their respective ranges) then all of the embeddings in Theorem~\ref{LMeasINTCor} hold true for \textit{both} $\dot{M}^s_{p,q}$ or $\dot{N}^s_{p,q}$ spaces for \textit{all} $p,q,s$, and $\sigma$ (again, in their respective ranges). It is also instructive to remark here that the additional assumption that $(X,\rho,\mu)$ is uniformly perfect is only used in proving that each of the estimates in (c)-(e) imply the measure condition in  (a), and the restriction $\sigma\geq C_\rho$ is only used in proving that (a) implies (b)-(e). Note that the uniform perfectness of the space $X$ is a natural assumption, in general. Indeed, if $\mu$ is \textit{upper $Q$-Ahlfors-regular on $X$}, that is, if there is a positive constant $C$ such that $\mu(B_\rho(x,r))\leq Cr^Q$ for any $x\in X$ and any finite $r\in(0,{\rm diam}_\rho(X)]$, then the uniformly perfect property is necessary for the lower measure bound \eqref{lowermeasure-intro} to hold true; see, for instance \cite[Lemma~4.7]{DaSe97}.

One distinguishing feature of Theorem~\ref{LMeasINTCor} is the range of $s$ for which the conclusion of this result holds true because it is the largest range of this type to be identified and it turns out to be in the nature of best possible. More specifically, if the underlying space is $\mathbb{R}^n$, equipped with the Euclidean distance, then ${\rm ind}\,(\mathbb{R}^n, |\,\cdot-\cdot\,|)=1$, and Theorem~\ref{LMeasINTCor} is valid for $\dot{M}^1_{p,\infty}=\dot{M}^{1,p}$, and for $\dot{M}^s_{p,q}$ and $\dot{N}^s_{p,q}$   whenever $s\in(0,1)$. Therefore, we recover the expected range for $s$ in the Euclidean setting. Similar considerations also hold whenever the underlying space is a domain $\Omega\subset\rn$ or, more generally, a metric measure space.  Remarkably, there are spaces where the range of $s$ is strictly larger than what it would be in the Euclidean setting. For example, if the underlying  space $(X,\rho)$ is an ultrametric space (like a Cantor-type set), then ${\rm ind}\,(X,\rho)=\infty$ and, in this case,  Theorem~\ref{LMeasINTCor} is valid for the  spaces $\dot{M}^{s,p}$, $\dot{M}^s_{p,q}$ and $\dot{N}^s_{p,q}$ for {\it all} $s\in(0,\infty)$.

Results currently appearing in the literature that are related or similar to Theorem~\ref{LMeasINTCor} typically require that the measure is doubling or that the space is connected; see, for instance, \cite{Karak1,Karak2,HIT16}. In contrast, Theorem~\ref{LMeasINTCor} is established under minimal assumptions on the ambient space. In particular, for a \textit{geodesic metric space}  (that is,  a metric
space with the property that any two points in it can be  joined by a curve whose length equals
the distance between these two points), it was shown in \cite{Karak1,Karak2} (see also \cite{HIT16})
that the lower measure bound \eqref{lowermeasure-intro} is necessary for certain $\dot{M}^{s}_{p,q}$ and $\dot{N}^s_{p,q}$-embeddings to hold true when $s\in(0,1)$. However, this is a rather restrictive connectivity
condition that precludes many  underlying
spaces where $\dot{M}^{s}_{p,q}$ and $\dot{N}^s_{p,q}$ continue to have rich theory. For example, on $\big(\mathbb{R}^n,|\cdot-\cdot|^{1/2}\big)$, which is a metric space that is clearly not geodesic, the results in \cite{Karak1,Karak2} are not applicable, whereas Theorem~\ref{LMeasINTCor} in this work implies that the lower measure bound \eqref{lowermeasure-intro} is actually \textit{equivalent}  to the $\dot{M}^{s}_{p,q}$ and $\dot{N}^s_{p,q}$ embeddings  for all $s\in(0,2)$, because  ${\rm ind}\,(\mathbb{R}^n,|\cdot-\cdot|^{1/2})=2\cdot{\rm ind}\,(\mathbb{R}^n,|\cdot-\cdot|)=2$. As this last example illustrates, there are connected metric spaces on which $\dot{M}^s_{p,q}$ and $\dot{N}^s_{p,q}$ are nontrivial even for $s>1$, which is, perhaps, a striking consequence of our work
because the range of $s$ for these spaces is typically limited to the interval $(0,1]$ in the literature. 	
The triviality of the spaces $\dot{M}^s_{p,q}$ and $\dot{N}^s_{p,q}$ is explored further in Section~\ref{sec-triv}.

Given a quasi-metric measure space $(X,\rho,\mu)$, the measure $\mu$ is said to be {\it doubling} provided
there exists a positive  constant $C$ such that
\begin{equation}\label{doub}
\mu(2B)\leq C\mu(B)\quad\mbox{ for all $\rho$-balls\,\,$B\subset X$.}
\end{equation}
The least constant playing the role of $C$ in \eqref{doub} is denoted
by $C_{\rm doub}$. Quasi-metric measure spaces where the measure is doubling are known as \emph{spaces of homogeneous type} \cite{CoWe71,CoWe77}. It follows from \eqref{doub} that, if $X$ contains at least two elements, then
$C_{\rm doub}>1$ (see \cite[p.\,72, Proposition~3.1]{AM15}). Given an exponent $Q\in(0,\infty)$, the
measure $\mu$ is said to be $Q$-{\it doubling} provided there exists a positive constant $\kappa$ satisfying
\begin{equation}\label{Doub-2}
\kappa\left(\frac{r}{R}\right)^{Q}\leq\frac{\mu(B_\rho(x,r))}{\mu(B_\rho(y,R))},
\end{equation}
whenever $x,\,y\in X$ satisfy $B_\rho(x,r)\subset B_\rho(y,R)$ and $0<r\leq R<\infty$. It is well known that the doubling property  \eqref{doub} implies \eqref{Doub-2} for each $Q\in[\log_2C_{\rm doub},\infty)$ (see, for instance, \cite[Lemma~4.7]{hajlasz}), and it is easy to see that any measure satisfying \eqref{Doub-2} for some $Q\in(0,\infty)$ is necessarily doubling. Although these two conditions are equivalent, the advantage of \eqref{Doub-2} over  \eqref{doub} is that the exponent $Q$ provides us with a notion of dimension for the space $X$.

The  following theorem is an analogue of Theorem~\ref{LMeasINTCor} for doubling measures which illustrates that spaces of homogeneous type are, in a sense, optimal environments when working with certain local $\dot{M}^s_{p,q}$ and $\dot{N}^s_{p,q}$ embeddings. The reader is also referred to Theorems~\ref{DOUBembedding} and  \ref{DoubMeasINT} below.

\begin{theorem}\label{DoubMeasINTCor}
Suppose that $(X,\rho,\mu)$ is a uniformly perfect
quasi-metric measure space and fix parameters $q\in(0,\infty]$,
$Q\in(0,\infty)$, and $\sigma\in[C_\rho,\infty)$.
Also, assume that $s\in(0,\infty)$ satisfies $s\preceq_q{\rm ind}\,(X,\rho)$.
Then the following statements are equivalent.
\begin{enumerate}[label=\rm{(\alph*)}]
\item The measure $\mu$ is $Q$-doubling on $X$.

\item There exist a $p\in(0,Q/s)$ and a $C_S\in(0,\infty)$ such that, for any
ball $B:=B_\rho(x,r)$ with $x\in X$ and $r\in(0,\infty)$, one has
\begin{equation*}
\Vert u\Vert_{L^{p^\ast}(B)}\leq
\frac{C_S}{[\mu(\sigma B)]^{s/Q}}\left[r^s\Vert u\Vert_{\dot{M}^s_{p,q}(\sigma B)}
+\Vert u\Vert_{L^{p}(\sigma B)}\right],
\end{equation*}
whenever $u\in\dot{M}^s_{p,q}(\sigma B,\rho,\mu)$.
Here,  $p^*:=Qp/(Q-sp)$.

\item There exist a $p\in(0,Q/s)$ and a $C_P\in(0,\infty)$ such that, for any
ball $B:=B_\rho(x,r)$ with $x\in X$ and $r\in(0,\infty)$, one has
\begin{equation*}
\inf_{\gamma\in\mathbb{R}}\Vert u-\gamma\Vert_{L^{p^\ast}(B)}\leq
\frac{C_P}{[\mu(\sigma B)]^{s/Q}}\,r^s\Vert u\Vert_{\dot{M}^s_{p,q}(\sigma B)},
\end{equation*}
whenever $u\in \dot{M}^s_{p,q}(\sigma B,\rho,\mu)$.
Here,  $p^*:=Qp/(Q-sp)$.

\item There exist a $p\in(Q/s,\infty)$ and a $C_H\in(0,\infty)$
such  that, for any ball $B:=B_\rho(x,r)$ with $x\in X$ and $r\in(0,\infty)$,
every function $u\in\dot{M}^s_{p,q}(\sigma B,\rho,\mu)$ has a
H\"older continuous representative of order $s-Q/p$ on $B$, denoted by $u$ again,  satisfying
\begin{eqnarray*}
|u(x)-u(y)|\leq C_H\,[\rho(x,y)]^{s-Q/p}\frac{r^{Q/p}}{[\mu(\sigma B)]^{1/p}}\Vert u\Vert_{\dot{M}^s_{p,q}(\sigma B)},\quad\forall\,x,\,y\in B.
\end{eqnarray*}
\end{enumerate}
In addition, if $q\leq p$, then all of the statements above  continue to be equivalent with  $\dot{M}^s_{p,q}$ replaced by $\dot{N}^s_{p,q}$.
\end{theorem}

\begin{remark}
Note that Theorem~\ref{DoubMeasINTCor} does not cover the case $p=Q/s$. For that case, see Theorems~\ref{DoubMeasINT} and \ref{AS-U-BDD} below.
\end{remark}

As with Theorem~\ref{LMeasINTCor}, the conclusions of Theorem~\ref{DoubMeasINTCor} are in the nature of best possible (in the manner described above). We also wish to mention that if $q=\infty$, then Theorems~\ref{LMeasINTCor} and \ref{DoubMeasINTCor} are valid when $s=(\log_2C_\rho)^{-1}<\infty$. Since $(\log_2C_\rho)^{-1}\geq 1$ whenever $(X,\rho)$ is a metric space, it follows that Theorems~\ref{LMeasINTCor} and \ref{DoubMeasINTCor} are valid for  $M^s_{p,\infty}$ with $s\in(0,1]$. This observation, along with the fact that $M^1_{p,\infty}=M^{1,p}$ (cf. Proposition~\ref{sobequal}), implies that Theorems~\ref{LMeasINTCor} and \ref{DoubMeasINTCor} extend to full generality the related work in \cite{agh20}.

The last main result of this article concerns the equivalence between the lower measure bound condition \eqref{lowermeasure-intro} and
the  \textit{global} embeddings of the  spaces $\dot{M}^s_{p,q}$ and $\dot{N}^s_{p,q}$.

\begin{theorem}
\label{GlobalEmbeddCor-INT}
Let $(X,\rho,\mu)$ be a quasi-metric measure space and fix parameters $s,\,p,\,Q\in(0,\infty)$ and $q\in(0,\infty]$ satisfying $s\preceq_q{\rm ind}\,(X,\rho)$. Also, suppose that either $X$ is a bounded set
or $\mu$ is $Q$-doubling.	Then  the following statements are equivalent.
\begin{enumerate}[label=\rm{(\alph*)}]
\item The measure $\mu$ is lower $Q$-Ahlfors-regular on $X$ $($in the sense of \eqref{lowermeasure-intro}$)$.

\item There exist a $p\in(0,Q/s)$ and a $C_S\in(0,\fz)$ satisfying
\begin{equation*}
\|u\|_{L^{p^*}(X)}\leq C_S\|u\|_{\dot{M}^s_{p,q}(X)}+\frac{C_S}{[{\rm diam}_\rho(X)]^s}\,\|u\|_{L^p(X)},\quad\forall\ u\in \dot{M}^s_{p,q}(X,\rho,\mu).
\end{equation*}
Here,  $p^*:=Qp/(Q-sp)$.
\end{enumerate}

\noindent If the space $(X,\rho)$ is uniformly perfect, then {\rm(a)} $($hence,
also {\rm(b)}$)$ is further equivalent to the following statements:

\begin{enumerate}[label=\rm{(\alph*)}]\addtocounter{enumi}{2}	
\item There exist a $p\in(0,Q/s)$ and a $C_P\in(0,\fz)$ satisfying
\begin{equation*}
\inf_{\gamma\in\mathbb{R}}\|u-\gamma\|_{L^{p^\ast}(X)}\leq
C_P\|u\|_{\dot{M}^s_{p,q}(X)},\quad\forall\ u\in \dot{M}^s_{p,q}(X,\rho,\mu).
\end{equation*}
Here,  $p^*:=Qp/(Q-sp)$.	

\item There exist positive constants $c_1,\,c_2$, and $\omega$ such that
\begin{equation*}
\mvint_{B} {\rm exp}\lf(c_1\frac{|u-u_B|}{\Vert u\Vert_{\dot{M}^s_{Q/s,q}(X)}}\r)^{\omega}\,d\mu\leq c_2
\end{equation*}
holds true for any ball $B\subset X$  $($with radius at most ${\rm diam}_\rho(X)$$)$ and any nonconstant function $u\in\dot{M}^s_{Q/s,q}(X,\rho,\mu)$.

\item There exist a $p\in(Q/s,\infty)$ and a $C_H\in(0,\infty)$ such that every $u\in \dot{M}^s_{p,q}(X,\rho,\mu)$ has
a H\"older continuous representative of order $s-Q/p$ on $X$, denoted by $u$ again, satisfying
\begin{eqnarray*}
|u(x)-u(y)|\leq C_H\,[\rho(x,y)]^{s-Q/p}\Vert u\Vert_{\dot{M}^s_{p,q}(X)},\qquad\forall\ x,\,y\in X.
\end{eqnarray*}
\end{enumerate}
In addition, if $q\leq p$, then all of the statements above continue to be equivalent with $\dot{M}^s_{p,q}$ replaced by $\dot{N}^s_{p,q}$.
\end{theorem}

Theorem~\ref{GlobalEmbeddCor-INT} is an improvement of known work for the embeddings of
Sobolev spaces in the metric setting (see, for instance, \cite{agh20,Karak1}).
More specifically, we have the following consequence of Theorem~\ref{GlobalEmbeddCor-INT}
which is, to our knowledge, brand new for the spaces $\dot{M}^{1,p}$   in the metric setting.

\begin{theorem}
\label{GlobalEmbeddCor-INT-S}
Let $(X,\rho,\mu)$ be a metric measure space and fix parameters $p,\,Q\in(0,\infty)$.
Also, suppose that either $X$ is a bounded set or $\mu$ is $Q$-doubling.  Then
the following statements are equivalent.
\begin{enumerate}[label=\rm{(\alph*)}]
\item The measure $\mu$ is lower $Q$-Ahlfors-regular on $X$ $($in the sense of \eqref{lowermeasure-intro}$)$.

\item There exist a $p\in(0,Q)$ and a $C_S\in(0,\fz)$ satisfying
\begin{equation*}
\|u\|_{L^{p^*}(X)}\leq C_S\|u\|_{\dot{M}^{1,p}(X)}+\frac{C_S}{\big[{\rm diam}_\rho(X)\big]^s}\,\|u\|_{L^p(X)},\quad\forall\ u\in \dot{M}^{1,p}(X,\rho,\mu).
\end{equation*}
Here,  $p^*:=Qp/(Q-p)$.		
\end{enumerate}

\noindent If the space $(X,\rho)$ is uniformly perfect, then {\rm(a)} $($hence,
also {\rm(b)}$)$ is further equivalent to the following statements:

\begin{enumerate}[label=\rm{(\alph*)}]\addtocounter{enumi}{2}	
\item There exist a $p\in(0,Q/s)$ and a $C_P\in(0,\fz)$ satisfying
\begin{equation*}
\inf_{\gamma\in\mathbb{R}}\|u-\gamma\|_{L^{p^\ast}(X)}\leq
C_P\|u\|_{\dot{M}^{1,p}(X)},\quad\forall\ u\in \dot{M}^{1,p}(X,\rho,\mu).
\end{equation*}
Here,  $p^*:=Qp/(Q-p)$.

\item There exist positive constants $c_1$, $c_2$, and $\omega$ such that
\begin{equation*}
\mvint_{B} {\rm exp}\lf(c_1\frac{|u-u_{B}|}{\Vert u\Vert_{\dot{M}^{1,Q}(X)}}\r)^{\omega}\,d\mu\leq c_2
\end{equation*}
holds true for any ball $B\subset X$   $($with radius at most ${\rm diam}_\rho(X)$$)$ and any nonconstant function $u\in\dot{M}^{1,Q}(X,\rho,\mu)$.

\item There exist a $p\in(Q/s,\infty)$ and a $C_H\in(0,\infty)$ such that every $u\in \dot{M}^{1,p}(X,\rho,\mu)$
has a H\"older continuous representative of order $1-Q/p$ on $X$, denoted by $u$ again,  satisfying
\begin{eqnarray*}
|u(x)-u(y)|\leq C_H\,[\rho(x,y)]^{1-Q/p}\Vert u\Vert_{\dot{M}^{1,p}(X)},\qquad\forall\,x,\,y\in X.
\end{eqnarray*}
\end{enumerate}
\end{theorem}

The layout of this article is as follows. In Section~\ref{section:preliminaries}, we review some basic terminology and results
pertaining to quasi-metric spaces and the main classes of function spaces considered in this article,
including the fractional Haj\l asz--Sobolev spaces $\dot{M}^{s,p}$, the Haj\l asz--Triebel--Lizorkin spaces $\dot{M}^s_{p,q}$
and the Haj\l asz--Besov  spaces  $\dot{N}^s_{p,q}$.

The proofs of our main results, Theorems \ref{LMeasINTCor}, \ref{DoubMeasINTCor}, and \ref{GlobalEmbeddCor-INT},
are presented in Sections~\ref{section:embeddings} and \ref{section:measequiv}.
More precisely, Section~\ref{section:embeddings} is devoted to establishing
the Sobolev  embeddings for these Haj\l asz type spaces. Under  a local lower bound condition of
the measure (the so-called  $V(\sigma B_0,Q,b)$ condition given in \eqref{measbound}),
we first establish  certain Sobolev  embeddings of the
Haj\l asz--Sobolev spaces  $\dot{M}^{s,p}$ in Theorem \ref{embedding}
below, which, together with some inclusion properties of
the spaces $\dot{M}^{s,p}$,  $\dot{M}^s_{p,q}$, and   $\dot{N}^s_{p,q}$ (see \eqref{MN-inclusion}
and Proposition~\ref{sobequal}), implies the  desired Sobolev embeddings  of $\dot{M}^s_{p,q}$  and   $\dot{N}^s_{p,q}$
listed in the statements ({b})-({e}) of  Theorems~\ref{LMeasINTCor}   and \ref{GlobalEmbeddCor-INT}, as well as
({b})-({d}) of  Theorem  \ref{DoubMeasINTCor} (see also Theorems \ref{LBembedding}, \ref{DOUBembedding},
and \ref{mainembedding-epsilon} below).

The main focus of Section~\ref{section:measequiv} is to prove that the Sobolev embeddings   established in Section~\ref{section:embeddings} imply a lower bound of the measure, which proves that any of the items in ({b})-({e}) of  Theorems~\ref{LMeasINTCor}   and \ref{GlobalEmbeddCor-INT},  or ({b})-({d}) of  Theorem  \ref{DoubMeasINTCor},
implies the related items ({a}) of these theorems.
En route, we first establish several necessary tools in Subsection~\ref{sssec:maintools}, and use them to prove our desired results, respectively, in Subsection~\ref{non-db} for non-doubling measures and Subsection~\ref{dbc} for doubling measures.
In particular, we construct a suitable family of maximally smooth ``bump" functions that approximate arbitrarily well the characteristic function of a given ball (see Lemma \ref{HolderBump} below), which is a
key tool used in the arguments of Section~\ref{section:measequiv},  and are of independent interest.
Recall that, in the metric setting,  \cite[Lemma~13]{agh20}  constructed such bump
functions that are Lipschitz. However,
in a general quasi-metric space, there is no guarantee
that nonconstant Lipschitz
functions exist. To overcome this, we construct our ``bump" functions in the  more
general scale of H\"older continuous functions, where the
maximal amount of the smoothness for the H\"older scale
that such a bump function can possess, denoted by ${\rm ind}_H(X,\rho)$, is intimately linked to
the geometry of the underlying space. This further leads to
an optimal range of the smoothness parameter $s$
for the characterizations of these Sobolev
embeddings.

As an application, in  Section \ref{sec-triv}, we prove that, on a quasi-metric measure space
$(X,\rho,\mu)$, if the measure
satisfies the lower bound condition \eqref{lowermeasure-intro},
then
both of the spaces $\dot{M}^s_{p,q}(X)$ and
$\dot{N}^s_{p,q}(X)$ contain only  constant functions whenever $s$ is strictly
larger than ${\rm ind}_H(X,\rho)$.  Here the index
${\rm ind}_H(X,\rho)$ (see \eqref{VVV-AGBVdef}) is known to coincide with
the infimum of all $\alpha\in(0,\fz)$ for which the corresponding H\"older class of smoothness order $\alpha$
on $(X,\rho)$  is non-trivial, and it quantitatively reflects   the geometry of the underlying quasi-metric space
$(X,\rho)$.

Finally, we make  some conventions on notation.
Let $\zz$ denote all integers and $\nn:=\{1,2,\ldots\}$.
 We always denote by $C$ a \emph{positive constant}
which is independent of the main parameters, but it
may vary from line to line. We also use
$C_{(\alpha,\beta,\ldots)}$ to denote a positive
constant depending on the indicated parameters $\alpha,
\beta,\ldots.$ The \emph{symbol} $f\lesssim g$ means
that $f\le Cg$. If $f\lesssim g$ and $g\lesssim f$,
we then write $f\approx g$. If $f\le Cg$ and $g=h$ or
$g\le h$, we then write $f\ls g\approx h$ or $f\ls g\ls h$,
\emph{rather than} $f\ls g=h$ or $f\ls g\le h$. The integral average of a  measurable function $u$ on a measurable set
$E\subset X$ with $\mu(E)\in(0,\fz)$ is denoted by
$$
u_E:=\mvint_Eu\, d\mu :=\frac{1}{\mu(E)}\int_E u\, d\mu,
$$	
whenever the integral is well defined.
For sets $E\subset (X,\rho)$, let ${\rm diam}_\rho(E):=\sup\{\rho(x,y):\, x,\,y\in E\}$ and ${\bf 1}_E$ be the characteristic
function of $E$.
For any $p\in(0,\fz]$, let $L^p(X):=L^p(X,\mu)$ denote the \emph{Lebesgue space} on $(X,\mu)$ consisting of all $p$-integrable functions on $X$.

\section{Preliminaries}
\label{section:preliminaries}
In this section, we present some basic assumptions on  quasi-metric measure spaces,
as well as some definitions and basic properties of the function spaces under consideration.

\subsection{Underlying Quasi-Metric Measure Spaces}\label{s-set}

Let $X$ be a nonempty set. A function $\rho:\,X\times X\to[0,\infty)$ is called a
{\it quasi}-{\it metric} on $X$, provided there exist
two positive constants $C_0$ and $C_1$ such  that, for any $x$, $y$, $z\in X$, one has
\begin{eqnarray}\label{gabn-T.2}
\begin{array}{c}
\rho(x,y)=0\Longleftrightarrow x=y,\quad
\rho(y,x)\leq C_0\rho(x,y),
\\[6pt]
\mbox{and }\quad
\rho(x,y)\leq C_1\max\{\rho(x,z),\rho(z,y)\}.
\end{array}
\end{eqnarray}
A pair $(X,\rho)$ is called a {\it quasi-metric} {\it space} if
$X$ is a nonempty set and $\rho$ is a quasi-metric on $X$.
We will tacitly assume that $X$ is of cardinality $\geq2$.

In this context,
two quasi-metrics $\rho$ and $\varrho$ on $X$ are said to be {\it equivalent},
denoted by $\rho\approx\varrho$, if there exists a positive constant $c$ such that
$$
c^{-1}\varrho(x,y)\leq\rho(x,y)\leq c\varrho(x,y),\quad\forall\ x,\, y\in X.
$$
It was pointed out in \cite[(4.289)]{MMMM13} that $\rho\approx\varrho$ if and only if the
 identity $(X,\rho)\to (X,\varrho)$  is bi-Lipschitz (see \cite[Definition 4.32)]{MMMM13}).
So sometimes we also say $\rho\approx\varrho$ as $\rho$
 is bi-Lipschitz equivalent to $\varrho$.

Since $X$ has cardinality $\geq2$, it follows  that the
constants $C_0$ and $C_1$ appearing in \eqref{gabn-T.2} are $\geq1$. Let
$C_\rho$   be the least constant which can play the role of
$C_1$  in \eqref{gabn-T.2}, that is,
\begin{eqnarray}\label{C-RHO.111}
C_\rho:=\sup_{\substack{x,\,y,\,z\in X\\\mbox{\scriptsize{{not all equal}}}}}
\frac{\rho(x,y)}{\max\{\rho(x,z),\rho(z,y)\}}\in[1,\fz).
\end{eqnarray}
Also, let $\widetilde{C}_\rho$   be the least constant
which can play the role of $C_0$  in
\eqref{gabn-T.2}, that is,
\begin{eqnarray}\label{C-RHO.111XXX}
\widetilde{C}_\rho:=\sup_{\substack{x,\,y\in X\\x\not=y}}
\frac{\rho(y,x)}{\rho(x,y)}\in[1,\fz).
\end{eqnarray}
If $C_\rho=\widetilde{C}_\rho=1$,
then $\rho$ is a genuine metric which
is commonly referred  as an ultrametric. Note that, if the underlying metric space is
$\mathbb{R}^n$, $n\in\mathbb{N}$, equipped with the Euclidean distance,
then $C_\rho=2$. It was shown in \cite[Theorem 2.6]{AM15} that there always exist nonconstant H\"older continuous functions of order $s$, for each finite $s\in(0,(\log_2C_\rho)^{-1}]$.

We now recall the notion of an ``index" which was originally
introduced in \cite[Definition 4.26]{MMMM13}. The \textit{lower smoothness index} of a quasi-metric space $(X,\rho)$ is defined as
\begin{equation}
\label{index}
{\rm ind}\,(X,\rho):=\sup_{\varrho\approx\rho}\big(\log_2C_\varrho\big)^{-1}\in(0,\infty],
\end{equation}
where the supremum is taken over all quasi-metrics $\varrho$ on $X$ which are bi-Lipschitz equivalent to $\rho$.
Generally speaking, ${\rm ind}\,(X,\rho)$ encodes
information about the geometry of the underlying space, as evidenced by the following properties:
\begin{itemize}[itemsep=1pt]
\item
{${\rm ind}\,(X,\rho)\geq 1$ if there exists a genuine distance on $X$ which is
equivalent to $\rho$;}


\item {${\rm ind}\,(X,\|\cdot-\cdot\|)=1$ if $(X,\|\cdot\|)$ is a nontrivial
normed vector space; hence ${\rm ind}\,(\mathbb{R}^n,|\cdot-\cdot|)=1$;}

\item {${\rm ind}\,(Y,\|\cdot-\cdot\|)=1$ if $Y$ is a subset
of a normed vector space $(X,\|\cdot\|)$ containing an open line segment; hence ${\rm ind}\,([0,1]^n,|\cdot-\cdot|)=1$;}

\item {${\rm ind}\,(X,\rho)\leq 1$ whenever the interval $[0,1]$ can be
bi-Lipschitzly embedded into $(X,\rho)$;}

\item {$(X,\rho)$ cannot be bi-Lipschitzly embedded into some ${\mathbb{R}}^n$,
$n\in{\mathbb{N}}$, whenever ${\rm ind}\,(X,\rho)<1$;}

\item {${\rm ind}\,(X,\rho)\leq Q$ if $(X,\tau_\rho)$ is pathwise connected
and $(X,\rho)$ is equipped with a $Q$-Ahlfors-regular measure, where $\tau_\rho$ denotes the topology generated by $\rho$;}

\item{there are compact, totally disconnected, Ahflors regular spaces
with lower smoothness index $\infty$; for instance, the four-corner planar Cantor set equipped with $|\cdot-\cdot|$;}

\item {${\rm ind}\,(X,\rho)=\infty$ if there exists an ultrametric
on $X$ which is equivalent to $\rho$, in particular, ${\rm ind}\,(X,\rho)=\infty$ whenever the underlying set
$X$ has finite cardinality; }

\item  {${\rm ind}\,(X,\rho)=1$ if $(X,\rho)$ is a metric space that is equipped with a doubling measure  and supports a weak $(1, p)$-Poincar\'e inequality with $p>1$.}
\end{itemize}
see Remark~\ref{trivpoin} and \cite[Section 4.7]{MMMM13} or \cite[Section 2.5]{AM15} for more details.

We next recall a sharp metrization theorem from \cite[Theorem 3.46]{MMMM13}.

\begin{theorem}\label{DST1}
Let $(X,\rho)$ be a quasi-metric space and assume that 	$C_\rho$, $\widetilde{C}_\rho\in[1,\infty)$ are as in
\eqref{C-RHO.111} and \eqref{C-RHO.111XXX}. Then there exists a symmetric quasi-metric $\rho_{\#}$ on $X$ with the property that for any finite number $\alpha\in\big(0,(\log_2C_\rho)^{-1}\big]$, the function $(\rho_{\#})^\alpha\colon X\times X\to [0,\fz)$ is  a genuine metric on $X$
satisfying
\begin{eqnarray}\label{XYB-11.4}
[\rho_{\#}(x,y)]^\alpha\leq [\rho_{\#}(x,z)]^\alpha+[\rho_{\#}(z,y)]^\alpha,
\quad\ \forall\ x,\,y,\,z\in X.
\end{eqnarray}
Moreover, $\rho_{\#}\approx\rho$ and
\begin{eqnarray}\label{DEQV1}
(C_\rho)^{-2}\rho(x,y)\leq\rho_{\#}(x,y)\leq\widetilde{C}_\rho\,\rho(x,y),\quad\ \forall\ x,\,y\in X.
\end{eqnarray}
Additionally, the function $\rho_{\#}:X\times X\longrightarrow [0,\infty)$ is continuous when $X\times X$ is equipped with the natural product topology $\tau_\rho\times\tau_\rho$. In particular, all $\rho_{\#}$-balls are open in the topology $\tau_\rho$.
\end{theorem}

\begin{remark}
\label{triquasi}
In the context of Theorem~\ref{DST1}, it follows immediately from \eqref{XYB-11.4} and the equivalence in \eqref{DEQV1} that, for any collection $\{x_1,\ldots,x_n\}$, with $n\in\{2,3,\ldots\}$, of points in $X$,
it holds true that
\begin{eqnarray*}
\rho(x_1,x_n)\leq\widetilde{C}_\rho C_\rho^2
\lf\{\sum_{i=1}^{n-1}[\rho(x_i,x_{i+1})]^\alpha\r\}^{1/\alpha}.
\end{eqnarray*}
\end{remark}	

Given a $\rho$-ball  $B_\rho(x,r):=\{y\in X:\, \rho(x,y)<r\}$ with $x\in X$ and
$r\in (0,\fz)$, we define
$\overline{B}_\rho(x,r):=\{y\in X:\, \rho(x,y)\leq r\}$. As a sign of warning, note that in general $\overline{B}_\rho(x,r)$ is not necessarily equal to the closure of $B_\rho(x,r)$. If $r=0$, then $B_\rho(x,r)=\emptyset$, but $\overline{B}_\rho(x,r)=\{ x\}$. Moreover, since $\rho$ is not necessarily symmetric, one needs to pay particular attention to the order of $x$ and $y$ in the definition of $B_\rho(x,r)$.  The triplet $(X,\rho,\mu)$ is  called  a \textit{quasi-metric measure space} if $X$ is a set of cardinality $\geq2$, $\rho$ is a quasi-metric on $X$, and $\mu$ is a Borel measure such that all $\rho$-balls are $\mu$-measurable and $\mu(B_\rho(x,r))\in(0,\infty)$ for any $x\in X$ and any $r\in(0,\infty)$.
We refer the reader to \cite{MMMM13,AMM13,AM15} for more information on this setting.

\subsection{Fractional Sobolev, Triebel--Lizorkin, and Besov Spaces}\label{s-func}
Suppose $(X,\rho,\mu)$ is a quasi-metric space equipped
with a nonnegative Borel measure, and let $s\in(0,\infty)$.
Following \cite{KYZ11}, a sequence $\{g_k\}_{k\in\mathbb{Z}}$ of nonnegative $\mu$-measurable
functions on $X$ is called a \textit{fractional $s$-gradient} of
a measurable function $u\colon X\rightarrow\mathbb{R}$ if there exists
a $\mu$-measurable set $E\subset X$ with $\mu(E)=0$ such that
\begin{equation*}
\vert u(x)-u(y)\vert\leq [\rho(x,y)]^s \left[g_k(x)+g_k(y)\right]
\end{equation*}
for all $k\in\mathbb{Z}$ and for all $x,\,y\in X\setminus E$ satisfying $2^{-k-1}\leq \rho(x,y)<2^{-k}.$ The collection of all fractional $s$-gradients of $u$ is denoted by $\mathbb{D}_\rho^s(u)$.

Given $p\in(0,\infty)$, $q\in(0,\infty]$, and a sequence $\vec{g}:=\{g_k\}_{k\in\mathbb{Z}}$ of measurable functions, we define
\begin{equation*}
\Vert \vec{g}\Vert_{L^p(X,\ell^q)}:=
\lf\Vert\, \Vert \{g_k\}_{k\in\mathbb{Z}}\Vert_{\ell^q} \r\Vert_{L^p(X,\mu)}
\end{equation*}
and
\begin{equation*}
\Vert \vec{g}\Vert_{\ell^q(L^p(X))}:=\lf\Vert \lf\{\Vert (g_k)\Vert_{L^p(X,\mu)}\r\}_{k\in\mathbb{Z}} \r\Vert_{\ell^q},
\end{equation*}
where
\begin{equation*}
\Vert \{g_k\}_{k\in\mathbb{Z}}\Vert_{\ell^q}:=
\begin{cases}
 \displaystyle\lf(\sum_{k\in\mathbb{Z}}\vert g_k\vert^q\r)^{1/q}& ~\text{when}~q\in(0,\infty),\\[10pt]
\,\,\,\,\,\,\displaystyle \sup_{k\in\mathbb{Z}}\vert g_k\vert& ~\text{when}~q=\infty.
\end{cases}
\end{equation*}

The \textit{homogeneous Haj\l asz--Triebel--Lizorkin space}
$\dot{M}^s_{p,q}(X,\rho,\mu)$ is defined as the collection of
all measurable functions $u\colon X\rightarrow\mathbb{R}$ for which
\begin{equation*}
\Vert u\Vert_{\dot{M}^s_{p,q}(X)}:=\Vert u\Vert_{\dot{M}^s_{p,q}(X,\rho,\mu)}:=
\inf_{\vec{g}\in\mathbb{D}_\rho^s(u)}\Vert\vec{g}\Vert_{L^p(X,\ell^q)}<\infty.
\end{equation*}
Here and thereafter, we keep the agreement that $\inf\emptyset:=\infty$.
The corresponding \textit{inhomogeneous Haj\l asz--Triebel--Lizorkin space}
$M^s_{p,q}(X,\rho,\mu)$ is defined as
$M^s_{p,q}(X,\rho,\mu):=\dot{M}^s_{p,q}(X,\rho,\mu)\cap L^p(X),$
equipped with the following (quasi-)norm
\begin{equation*}
\Vert u\Vert_{{M}^s_{p,q}(X)}:=\Vert u\Vert_{M^s_{p,q}(X,\rho,\mu)}:=\Vert u\Vert_{\dot{M}^s_{p,q}(X,\rho,\mu)}+\Vert u\Vert_{L^p(X)},
\quad\forall\,u\in M^s_{p,q}(X,\rho,\mu).
\end{equation*}

The \textit{homogeneous Haj\l asz--Besov space} $\dot{N}^s_{p,q}(X,\rho,\mu)$ is
defined as the collection of all measurable functions $u\colon X\rightarrow\mathbb{R}$ for which
\begin{equation*}
\Vert u\Vert_{\dot{N}^s_{p,q}(X)}:=\Vert u\Vert_{\dot{N}^s_{p,q}(X,\rho,\mu)}:=\inf_{\vec{g}\in\mathbb{D}_\rho^s(u)}
\Vert\vec{g}\Vert_{\ell^q(L^p(X))}<\infty,
\end{equation*}
and the \textit{inhomogeneous Haj\l asz--Besov space} is defined
as $N^s_{p,q}(X,\rho,\mu):=\dot{N}^s_{p,q}(X,\rho,\mu)\cap L^p(X)$, and
equipped with the following (quasi-)norm
\begin{equation*}
\Vert u\Vert_{N^s_{p,q}(X)}:=\Vert u\Vert_{N^s_{p,q}(X,\rho,\mu)}:=\Vert u\Vert_{\dot{N}^s_{p,q}(X,\rho,\mu)}+\Vert u\Vert_{L^p(X)},\quad\forall\ u\in N^s_{p,q}(X,\rho,\mu).
\end{equation*}

When $p\in[1,\infty)$ and $q\in[1,\infty]$, it follows that  $\Vert\cdot\Vert_{M^s_{p,q}(X,\rho,\mu)}$ and $\Vert\cdot\Vert_{N^s_{p,q}(X,\rho,\mu)}$ are genuine norms and the corresponding spaces   are Banach spaces. Otherwise, these spaces are quasi-Banach.
We will simply use $\dot{M}^s_{p,q}(X)$, $\dot{N}^s_{p,q}(X)$, $M^s_{p,q}(X)$, and $N^s_{p,q}(X)$,
respectively, in place of $\dot{M}^s_{p,q}(X,\rho,\mu)$, $\dot{N}^s_{p,q}(X,\rho,\mu)$, $M^s_{p,q}(X,\rho,\mu)$, and $N^s_{p,q}(X,\rho,\mu)$,  whenever the quasi-metric and the measure are well understood from the context. It was shown in \cite{KYZ11} that $M^s_{p,q}\big(\mathbb{R}^n\big)$ coincides with the classical Triebel--Lizorkin space $F^s_{p,q}(\mathbb{R}^n)$ for any $s\in(0,1)$, $p\in(\frac{n}{n+s},\infty)$, and $q\in(\frac{n}{n+s},\infty]$,  and $N^s_{p,q}(\mathbb{R}^n)$ coincides with the classical Besov space $B^s_{p,q}\big(\mathbb{R}^n\big)$ for any $s\in(0,1)$, $p\in(\frac{n}{n+s},\infty)$, and $q\in(0,\infty]$. For more studies on these spaces, we refer the reader to, for instance,  \cite{GKZ13,HIT16,HKT17,y17,Karak1,Karak2,agh20,AWYY21}.

The following result appears in \cite{AYY21}.

\begin{proposition}
\label{constant}
Let $(X,\rho,\mu)$ be a quasi-metric space equipped with a nonnegative Borel measure, and suppose that $s,\,p\in(0,\infty)$ and $q\in(0,\infty]$. For any fixed $u\in\dot{M}^s_{p,q}(X,\rho,\mu)$, one has that $\Vert u\Vert_{\dot{M}^s_{p,q}(X,\rho,\mu)}=0$   if and only if $u$ is constant $\mu$-almost everywhere in $X$. The above statement also holds true with $\dot{M}^s_{p,q}$ replaced by $\dot{N}^s_{p,q}$.
\end{proposition}

The above Triebel--Lizorkin and Besov spaces are monotone in the parameter $q$. Indeed, recall the following very basic inequality: for
any $\theta\in(0,1]$ and any sequence $\{a_k\}_{k\in\mathbb{Z}}$ of complex numbers, it holds true that
\begin{equation}
\label{thetaless1}
\left(\sum_{k\in\mathbb{Z}}|a_k|\right)^{\theta}
\leq\sum_{k\in\mathbb{Z}}|a_k|^{\theta}.
\end{equation}
Using \eqref{thetaless1} with $\theta=q_1/q_2$, it follows that
\begin{align}\label{MN-inclusion}
&\dot{M}^s_{p,q_1}(X,\rho,\mu)\hookrightarrow\dot{M}^s_{p,q_2}(X,\rho,\mu)\quad\mbox{and}\quad
\dot{N}^s_{p,q_1}(X,\rho,\mu)\hookrightarrow\dot{N}^s_{p,q_2}(X,\rho,\mu) \nonumber\\
&\qquad \mbox{(with appropriate control of the quasi-norms) for any } \nonumber\\
&\qquad \mbox{$s,\,p\in(0,\infty)$ and every $q_1,\,q_2\in(0,\infty]$ satisfying $q_1\leq q_2$.}
\end{align}
Here and thereafter, for any two quasi-Banach spaces $(Y,\|\cdot\|_{Y})$ and $(Z,\|\cdot\|_{Z})$,
$Y\hookrightarrow Z$ means that $(Y,\|\cdot\|_Y)$ is continuously embedded into $(Z,\|\cdot\|_{Z})$.

Next, we recall the definition of (fractional) Haj\l asz--Sobolev spaces. Suppose that
$(X,\rho)$ is a quasi-metric space equipped with a nonnegative Borel measure $\mu$, and let $s\in(0,\infty)$. Following \cite{Y03}, a nonnegative $\mu$-measurable function $g$, defined on $X$, is called an \textit{$s$-gradient} of a measurable function $u\colon X\rightarrow\mathbb{R}$ if there exists a set $E\subset X$ with $\mu(E)=0$ such that
\begin{equation}\label{fracHajlasz}
\vert u(x)-u(y)\vert\leq [\rho(x,y)]^s\left[g(x)+g(y)\right],\quad\forall\,x,\,y\in X\setminus E.
\end{equation}
The collection of all $s$-gradients of $u$ is denoted by $\mathcal{D}_\rho^s(u)$.

Let $p\in(0,\infty)$. The \textit{homogeneous Haj\l asz--Sobolev space} $\dot{M}^{s,p}(X,\rho,\mu)$ is defined as the collection of all measurable functions $u\colon X\rightarrow\mathbb{R}$ for which
\begin{equation*}
\Vert u\Vert_{\dot{M}^{s,p}(X)}:=\Vert u\Vert_{\dot{M}^{s,p}(X,\rho,\mu)}:=\inf_{g\in\mathcal{D}_\rho^s(u)}\Vert g\Vert_{L^p(X)}<\infty,
\end{equation*}
again, with the understanding that $\inf\emptyset:=\infty$. The corresponding \textit{inhomogeneous Haj\l asz--Sobolev space} $M^{s,p}(X,\rho,\mu)$ is defined by setting
$M^{s,p}(X,\rho,\mu):=\dot{M}^{s,p}(X,\rho,\mu)\cap L^p(X),$   equipped with the following (quasi-)norm
\begin{equation*}
\Vert u\Vert_{M^{s,p}(X)}:=\Vert u\Vert_{M^{s,p}(X,\rho,\mu)}:=\Vert u\Vert_{\dot{M}^{s,p}(X,\rho,\mu)}+\Vert u\Vert_{L^p(X)},
\quad\forall\ u\in M^{s,p}(X,\rho,\mu).
\end{equation*}
Also, let
$$\mathcal{D}_\rho^{s,+}(u):=\lf\{g\in \mathcal{D}_\rho^s(u):\,\|g\|_{L^p(X)}>0\r\}.$$

Recall that the Haj\l asz--Sobolev spaces $M^{1,p}$ were originally introduced by Haj\l asz \cite{hajlasz2}, while
their fractional variants were introduced in \cite{Hu03,Y03}. The Haj\l asz--Sobolev spaces are naturally related to Haj\l asz--Triebel--Lizorkin and Haj\l asz--Besov spaces as follows.

\begin{proposition}\label{bas}
\label{sobequal}
Let $(X,\rho,\mu)$ be a quasi-metric space
equipped with a nonnegative Borel measure $\mu$, and
suppose that $s,\,p\in(0,\infty)$. Then
\begin{enumerate}
\item[${\rm (i)}$] $\dot{M}^s_{p,\infty}(X,\rho,\mu)=\dot{M}^{s,p}(X,\rho,\mu)$ as sets, with equal semi-norms\footnote{Given a vector space
$\mathscr{X}$ over $\rr$, recall that a function $\|\cdot\|:\mathscr{X}\to\mathbb[0,\infty)$
is called a {\it semi}-{\it norm} provided that, for any $x,\,y\in\mathscr{X}$, the following three conditions hold true:\
${\rm(i)}$ $x=0$ implies $\|x\|=0$, ${\rm(ii)}$
$\|\lambda x\|=|\lambda|\!\cdot\!\|x\|$ for any  $\lambda\in\mathbb{R}$,
and ${\rm(iii)}$ $\|x+y\|\leq \|x\|+\|y\|$.};

\item[${\rm (ii)}$] if $q\in(0,p]$, then $\dot{N}^s_{p,q}(X,\rho,\mu)\hookrightarrow \dot{M}^{s,p}(X,\rho,\mu)\hookrightarrow\dot{N}^s_{p,\infty}(X,\rho,\mu)$;

\item[${\rm (iii)}$] for any $\varepsilon\in(0,s)$, there exists a  positive constant $C$ with the property that, if $B\subset X$ is a $\rho$-ball with radius $r\in(0,\infty)$, then, for any $q\in(0,\infty]$, $\dot{N}^s_{p,q}(B,\rho,\mu)\hookrightarrow\dot{M}^{\varepsilon,p}(B,\rho,\mu)$   and $\Vert u\Vert_{\dot{M}^{\varepsilon,p}(B,\rho,\mu)}\leq Cr^{s-\varepsilon}\Vert u\Vert_{\dot{N}^s_{p,q}(B,\rho,\mu)}$ for all $u\in \dot{N}^s_{p,q}(B,\rho,\mu)$.
\end{enumerate}
\end{proposition}

\begin{proof}
The equality $\dot{M}^s_{p,\infty}(X,\rho,\mu)=\dot{M}^{s,p}(X,\rho,\mu)$ was proven in \cite[Proposition~2.1]{KYZ11} when $(X,\rho,\mu)$ is a \textit{metric}-measure space.
We include a short proof for the sake of completeness. Let $u\in\dot{M}^{s,p}(X,\rho,\mu)$ and  $g\in\mathcal{D}_\rho^s(u)$. Defining $g_k:=g$ for all $k\in\mathbb{Z}$, it is straightforward to check that $\vec{g}:=\{g_k\}_{k\in\mathbb{Z}}\in\mathbb{D}^s_\rho(u)$ and $\Vert\vec{g}\Vert_{L^p(X,\ell^\infty)}=\Vert g\Vert_{L^p(X,\mu)}$. Therefore, $u\in \dot{M}^s_{p,\infty}(X,\rho,\mu)$ and $\Vert u\Vert_{\dot{M}^{s,p}(X,\rho,\mu)}\leq\Vert u\Vert_{\dot{M}^s_{p,\infty}(X,\rho,\mu)}$. Conversely, if $u\in\dot{M}^s_{p,\infty}(X,\rho,\mu)$ and $\vec{g}:=\{g_k\}_{k\in\mathbb{Z}}\in\mathbb{D}^s_\rho(u)$, then, taking $g:=\sup_{k\in\mathbb{Z}}g_k$, we have $\Vert g\Vert_{L^p(X,\mu)}=\Vert\vec{g}\Vert_{L^p(X,\ell^\infty)}$. It follows that $u\in \dot{M}^s_{p,\infty}(X,\rho,\mu)$ and $\Vert u\Vert_{\dot{M}^s_{p,\infty}(X,\rho,\mu)}\leq\Vert u\Vert_{\dot{M}^{s,p}(X,\rho,\mu)}$.
This finishes the proof of (i).

Turning our attention to  (ii), we  first show that $\dot{M}^{s,p}(X,\rho,\mu)\subset\dot{N}^s_{p,\infty}(X,\rho,\mu)$. Suppose that $u\in \dot{M}^{s,p}(X,\rho,\mu)$ and choose $g\in\mathcal{D}_\rho^s(u)$. Taking $g_k:=g$ for each $k\in\mathbb{Z}$, it is easy to check that $\vec{g}:=\{g_k\}_{k\in\mathbb{Z}}\in\mathbb{D}_\rho^s(u)$ and $\Vert\vec{g}\Vert_{\ell^\infty(L^p(X))}=\Vert g\Vert_{L^p(X,\mu)}$, which implies that $u\in\dot{N}^s_{p,\infty}(X,\rho,\mu)$ and $\Vert u\Vert_{\dot{N}^s_{p,\infty}(X,\rho,\mu)}\leq\Vert u\Vert_{\dot{M}^{s,p}(X,\rho,\mu)}$.

Suppose $q\in(0,p]$. From definitions, we  immediately deduce that $\dot{N}^s_{p,p}(X,\rho,\mu)=\dot{M}^s_{p,p}(X,\rho,\mu)$ as sets, with equal semi-norms. Combining this
observation with \eqref{MN-inclusion} and part  (i)  of this proposition gives
$$\dot{N}^s_{p,q}(X,\rho,\mu)\hookrightarrow\dot{N}^s_{p,p}(X,\rho,\mu)=\dot{M}^s_{p,p}(X,\rho,\mu)
\hookrightarrow\dot{M}^s_{p,\infty}(X,\rho,\mu)=\dot{M}^{s,p}(X,\rho,\mu),$$
with appropriate control of the associated semi-norms. This finishes the proof of (ii).

To address the claims in  (iii), fix $\varepsilon\in(0,s)$ and a $\rho$-ball
$B:=B_\rho(x,r)\subset X$ with $x\in X$ and $r\in(0,\infty)$.
Since $\dot{N}^s_{p,q}(B,\rho,\mu)\hookrightarrow\dot{N}^s_{p,\infty}(B,\rho,\mu)$
(with appropriate control of the norms) for any given $q\in(0,\infty]$, to complete the proof of (ii),
it suffices to show $\dot{N}^s_{p,\infty}(B,\rho,\mu)\hookrightarrow\dot{M}^{\varepsilon,p}(B,\rho,\mu)$
and $\Vert u\Vert_{\dot{M}^{\varepsilon,p}(B,\rho,\mu)}
\lesssim r^{s-\varepsilon}\Vert u\Vert_{\dot{N}^s_{p,\infty}(B,\rho,\mu)}$
for all $u\in \dot{N}^s_{p,\infty}(B,\rho,\mu)$. To this end, fix
$u\in \dot{N}^s_{p,\infty}(B,\rho,\mu)$ and take
$\vec{g}:=\{g_k\}_{k\in\mathbb{Z}}\in\mathbb{D}_\rho^s(u)$. With $C_\rho$ and $\widetilde{C}_\rho\in[1,\infty)$
as in \eqref{C-RHO.111} and \eqref{C-RHO.111XXX}, let $k_0$ be the unique integer satisfying
$2^{-k_0-1}\leq C_\rho\widetilde{C}_\rho r<2^{-k_0}$ and define
$$g:=\left[\sum_{k=k_0}^\infty2^{-k(s-\varepsilon)p}g_k^p\right]^{\frac{1}{p}}.$$

We claim that $g\in\mathcal{D}_\rho^s(u)$. Indeed, observe that for any $y,\, z\in B$, we have
$$
\rho(y,z)\leq C_\rho\max\{\rho(y,x),\rho(x,z)\}<C_\rho\widetilde{C}_\rho r<2^{-k_0}.
$$
As such, for each $y,\, z\in B$, there exists a unique integer $j \geq k_0$ such that
$2^{-j-1}\leq \rho(y,z)<2^{-j}.$ Since $\{g_k\}_{k\in\zz}\in\mathbb{D}_\rho^s(u)$,
it follows that there exists an $E\subset X$ with $\mu(E)=0$ such that
\begin{align*}
|u(y)-u(z)|&\leq[\rho(y,z)]^s\lf[g_{j}(y)+g_{j}(z)\r]\leq
[\rho(y,z)]^\varepsilon 2^{-j(s-\varepsilon)}\lf[g_{j}(y)+g_{j}(z)\r]\\
&\leq [\rho(y,z)]^\varepsilon \lf[g(y)+g(z)\r]
\end{align*}
whenever $y,\, z\in B\setminus E$,
which implies that $g\in\mathcal{D}_\rho^\varepsilon(u)$. On the other hand, we know that
\begin{align*}
\|g\|_{L^p(B)}
&=\left[\sum_{k=k_0}^\infty 2^{-k(s-\varepsilon)p}\int_{B}
g_k^p\,d\mu\right]^{\frac{1}{p}}\leq\left[\sum_{k=k_0}^\infty 2^{-k(s-\varepsilon)p}
\right]^{\frac{1}{p}}\sup_{j\in\mathbb{Z}}\Vert g_j\Vert_{L^p(B)}\\
&\lesssim 2^{-k_0(s-\varepsilon)}\Vert\vec{g}\Vert_{L^p(B,\ell^\infty)}
\approx r^{s-\varepsilon}\Vert\vec{g}\Vert_{L^p(B,\ell^\infty)}.
\end{align*}
From this, we  deduce that $u\in\dot{M}^{\varepsilon,p}(B,\rho,\mu)$ and $\Vert u\Vert_{\dot{M}^{\varepsilon,p}(B,\rho,\mu)}\lesssim r^{s-\varepsilon}\Vert u\Vert_{\dot{N}^s_{p,\infty}(B,\rho,\mu)}$ whenever $u\in \dot{N}^s_{p,\infty}(B,\rho,\mu)$.	This finishes the proof of Proposition \ref{bas}.
\end{proof}

\begin{remark}  By  \cite[Theorem 1.2 and Proposition 2.1]{KYZ11} (see also \cite{Y03}),
\cite[Remark 2.2]{y17}
 and \cite[Theorem 3.1.1]{ST95}, we know that, for any $s\in(0,1)$,
the inhomogeneous counterpart on $\rn$ of the embeddings in
Proposition \ref{bas}(ii)
$${N}^s_{p,q}(\rn)={B}^s_{p,q}(\rn)\hookrightarrow  F^s_{p,\infty}(\rn)=M^s_{p,\infty}(\rn)=\dot{M}^{s,p}(\rn)$$
holds true if and only if $q\in(0,p]$.
\end{remark}

The next result  shows that
the homogeneous and the
inhomogeneous versions for each of the Sobolev, Triebel--Lizorkin, and Besov spaces agree on balls.

\begin{proposition}
\label{equalBDD}
Let $(X,\rho,\mu)$ be a quasi-metric measure space and suppose that $B\subset X$ is a $\rho$-ball with radius $r\in(0,\infty)$. Then $\dot{M}^{s,p}(B,\rho,\mu)={M}^{s,p}(B,\rho,\mu)$, as sets, for any $s,\,p\in(0,\infty)$. As a corollary, one has $\dot{M}^s_{p,q}(B,\rho,\mu)=M^s_{p,q}(B,\rho,\mu)$ and $\dot{N}^s_{p,q}(B,\rho,\mu)=N^s_{p,q}(B,\rho,\mu)$, as sets, for any $s,\,p\in(0,\infty)$ and $q\in(0,\infty]$.
\end{proposition}	

\begin{proof}
Fix $s,\,p\in(0,\infty)$. Clearly, we have  $M^{s,p}(B,\rho,\mu)\subset\dot{M}^{s,p}(B,\rho,\mu)$.
To show the inverse inclusion,  we need to prove that every   $u\in \dot{M}^{s,p}(B,\rho,\mu)$
 also belongs to $L^p(B)$. To see this, choose $g\in\mathcal{D}_\rho^s(u)$
with $\Vert g\Vert_{L^p(B)}<\infty$, and let $E\subset B$ be a $\mu$-measurable set such that $\mu(E)=0$ and
\begin{equation}\label{av-ui2X}
\vert u(x)-u(y)\vert\leq [\rho(x,y)]^s\left[g(x)+g(y)\right],\quad\forall\ x,\,y\in B\setminus E.
\end{equation}
Fix a point $y_0\in B\setminus E$ where $g(y_0)<\infty$. Note that such a choice is possible,
because  $\Vert g\Vert_{L^p(B)}<\infty$. Then, by \eqref{av-ui2X}, for any $x\in B\setminus E$, one has
\begin{equation}\label{av-ui2XX}
\vert u(x)\vert\leq[\rho(x,y)]^s\left[g(x)+g(y_0)\right]+\vert u(y_0)\vert
\leq[{\rm diam}_\rho(B)]^s\left[g(x)+g(y_0)\right]+\vert u(y_0)\vert.
\end{equation}
Raising the extreme most sides of \eqref{av-ui2XX} to the power $p$ and integrating in the $x$ variable over $B$,
we find that
\begin{align*}
\int_B\vert u\vert^p\,d\mu
&\lesssim \int_B \lf\{g(x)+g(y_0)+\vert u(y_0)\vert\r\}^p\,d\mu(x)\nonumber\\
&\lesssim \int_B g^p\,d\mu+\mu(B)[g(y_0)]^p+\mu(B)\vert u(y_0)\vert^p<\infty.
\end{align*}
Thus, $u\in L^p(B)$, as desired.
This finishes the proof of $\dot{M}^{s,p}(B,\rho,\mu)=M^{s,p}(B,\rho,\mu)$, as sets.

The equalities $\dot{M}^s_{p,q}(B,\rho,\mu)=M^s_{p,q}(B,\rho,\mu)$ and $\dot{N}^s_{p,q}(B,\rho,\mu)=N^s_{p,q}(B,\rho,\mu)$ now easily follow from Proposition~\ref {sobequal} and what we have just shown. This finishes the proof of Proposition \ref{equalBDD}.
\end{proof}



\section{Embedding Theorems for Triebel--Lizorkin and Besov Spaces on Quasi-Metric Measure Spaces}
\label{section:embeddings}

In this section, we establish the Sobolev-type embedding properties
for the Haj\l asz--Triebel--Lizorkin and Haj\l asz--Besov spaces on quasi-metric measure spaces,
which deduce the implications
from the item (a) to
all other items in Theorems \ref{LMeasINTCor}, \ref{DoubMeasINTCor},   and \ref{GlobalEmbeddCor-INT}.

We begin by establishing Sobolev-type embeddings of the
Haj\l asz--Sobolev spaces $\dot{M}^{s,p}$, which, together with the inclusions in
\eqref{MN-inclusion} and Proposition~\ref{sobequal}, further implies  the corresponding embeddings for
Haj\l asz--Triebel--Lizorkin spaces $\dot{M}^s_{p,q}$ and  Haj\l asz--Besov spaces $\dot{N}^s_{p,q}$.

To facilitate the formulation of the result, we introduce the
following piece of notation:\  Given $Q,\,b\in(0,\infty)$, $\sigma\in[1,\infty)$
and a ball $B_0\subset X$ of radius $R_0\in(0,\infty)$, the measure $\mu$ is
said to satisfy the \emph{$V(\sigma B_0,Q,b)$ condition}
\footnote{This condition is a slight variation of the one in \cite[p.\,197]{hajlasz}.} provided
\begin{equation}
\label{measbound}
\mu\big(B_\rho(x,r)\big)\geq br^Q
\quad
\text{for any}\,\,
\text{$x\in X$ and $r\in(0,\sigma R_0]$ satisfying $B_\rho(x,r)\subset\sigma B_0$.}
\end{equation}
Recall that the constants $C_\rho,\,\widetilde{C}_\rho\in[1,\infty)$ are defined in \eqref{C-RHO.111} and \eqref{C-RHO.111XXX}.
We have the following conclusion on Sobolev-type embeddings of
Haj\l asz--Sobolev spaces, whose proof is similar to that of \cite[Theorem 6]{agh20}.
However, because we are working with quasi-metrics instead of metrics, the proof is a bit
more delicate and so we give some details here for the sake of completeness.

\begin{theorem}
\label{embedding}
Let $(X,\rho,\mu)$ be a quasi-metric measure space.	
Let $s,\,p\in(0,\infty)$, $\sigma\in[C_\rho,\infty)$, and $B_0$ be  a
$\rho$-ball of radius $R_0\in(0,\infty)$.
 Assume that the measure $\mu$
satisfies the  $V(\sigma B_0,Q,b)$ condition for some $Q,\,b\in(0,\infty)$.
Then there exist positive constants $C$, $C_1$ and $C_2$, depending only on $\rho$,
$s$, $p$, $Q$, and $\sigma$, such that, for any $u\in\dot{M}^{s,p}(\sigma B_0,\rho,\mu)$
and $g\in \mathcal{D}_\rho^s(u)$, the following statements hold true.
\begin{enumerate}
\item[{\rm (a)}] If $p\in(0,Q/s)$, then $u\in L^{p^*}(B_0)$, where $p^*:=Qp/(Q-sp)$, and $u$ satisfies
\begin{equation}
\label{eq18}
\left(\, \mvint_{B_0} |u|^{p^*}\, d\mu\right)^{1/p^*}\leq
C\left[\frac{\mu(\sigma B_0)}{bR_0^Q}\right]^{1/p}
R_0^s\left(\,\mvint_{\sigma B_0}g^{p}\, d\mu\right)^{1/p}
+C\left(\,\mvint_{\sigma B_0}|u|^p\, d\mu\right)^{1/p}
\end{equation}
and
\begin{equation}
\label{eq19}
\inf_{\gamma\in\mathbb{R}}\left(\, \mvint_{B_0} |u-\gamma|^{p^*}\, d\mu\right)^{1/p^*}\leq
C\left[\frac{\mu(\sigma B_0)}{bR_0^Q}\right]^{1/p}R_0^s\left(\,\mvint_{\sigma B_0}g^{p}\, d\mu\right)^{1/p}.
\end{equation}

\item[{\rm (b)}] If $p=Q/s$ and $g\in \mathcal{D}_\rho^{s,+}(u)$, then
\begin{equation}
\label{eq20-X}
\mvint_{B_0} {\rm exp}\left(C_1b^{s/Q}\frac{|u-u_{B_0}|}{\|g\|_{L^{Q/s}(\sigma B_0)}}\right)\,d\mu\leq C_2.
\end{equation}

\item[{\rm (c)}] If $p>Q/s$, then
\begin{equation}
\label{eq29}
\Vert u-u_{B_0}\Vert_{L^\infty(B_0)}\leq
C\left[\frac{\mu(\sigma B_0)}{bR_0^Q}\right]^{1/p}R_0^s\left(\,\mvint_{\sigma B_0}g^{p}\, d\mu\right)^{1/p}.
\end{equation}
In particular, $u$ has a H\"older continuous representative of order $s-Q/p$ on $B_0$, denoted by $u$ again, satisfying
\begin{equation}
\label{eq30}
|u(x)-u(y)|\leq C b^{-1/p}[\rho(x,y)]^{s-Q/p}\left(\,\int_{\sigma B_0}g^{p}\, d\mu\right)^{1/p},
\quad \forall\ x,\,y\in B_0.
\end{equation}
\end{enumerate}
\end{theorem}

\begin{remark}
Let all notation be as in Theorem \ref{embedding}(a).
If $p^*\in[1,\fz)$, then, by the H\"older inequality, we have, for any ball $B_0$, $u\in L^1(B_0)$ and
\begin{equation}
\label{eq20}
\left(\, \mvint_{B_0} |u-u_{B_0}|^{p^*}\, d\mu\right)^{1/p^*}\leq
2\inf_{\gamma\in\mathbb{R}}\left(\, \mvint_{B_0} |u-\gamma|^{p^*}\, d\mu\right)^{1/p^*}.
\end{equation}
Hence we can replace the expression on the left hand side of \eqref{eq19} with the one on
the left hand side of \eqref{eq20}. Moreover, if $s<Q$ and  $p\in [1,Q/s)$,
we also find that \eqref{eq18} easily follows from this new version of \eqref{eq19}.
However, we do not know how to
conclude \eqref{eq18} from \eqref{eq19} when $p<1$.
\end{remark}

\begin{proof}[Proof of Theorem~\ref{embedding}]
Let $u\in\dot{M}^{s,p}(\sigma B_0,\rho,\mu)$
and $g\in \mathcal{D}_\rho^s(u)$.
Without loss of generality, we may assume $\int_{\sigma B_0}g^p\, d\mu>0$.
Indeed, if this integral equals zero,
then $g=0$ $\mu$-almost everywhere in $X$ which, in turn, implies that
$u$ is a constant function $\mu$-almost everywhere and the result is obvious
in this scenario.

By replacing, if necessary, $g$ with
$\widetilde{g}:=g+ (\mvint_{\sigma B_0} g^p\, d\mu )^{1/p}$, we may further assume that
\begin{equation}
\label{eq4}
g(x)\geq 2^{-(1+1/p)}\left(\, \mvint_{\sigma B_0} g^p\,  d\mu\right)^{1/p}>0
\quad
\text{for $\mu$-almost every $x\in\sigma B_0$.}
\end{equation}
Let $E\subset \sigma B_0$ be a set of measure zero such that the pointwise inequality
\eqref{fracHajlasz} holds true for every $x,\,y\in \sigma B_0\setminus E$. For any $k\in\mathbb{Z}$, let
$$
E_k:=\lf\{ x\in\sigma B_0\setminus E:\, g(x)\leq 2^k\r\}.
$$
Clearly  $E_{k-1}\subset E_k$ for any $k\in\zz$. By \eqref{eq4},  we find that
\begin{equation}
\label{eq42}
\mu\left(\sigma B_0\setminus\bigcup_{k\in\mathbb{Z}} (E_k\setminus E_{k-1})\right)=0.
\end{equation}
It follows from the pointwise inequality \eqref{fracHajlasz} that $u$ restricted to $E_k$ is $2^{k+1}$-H\"older continuous of order $s$, that is,
\begin{equation}
\label{eq8}
|u(x)-u(y)|\leq 2^{k+1}[\rho(x,y)]^s,
\quad
\forall\ x,\,y\in E_k.
\end{equation}
Also, applying the Chebyshev inequality, we have
\begin{equation}
\label{eq7}
\mu(\sigma B_0\setminus E_k) = \mu\Big(\Big\{x\in\sigma B_0:\, g(x)>2^k\Big\}\Big)\leq 2^{-kp}\int_{\sigma B_0} g^p\, d\mu.
\end{equation}
In light of \eqref{eq42},   we  write
\begin{equation}
\label{eq5}
\int_{\sigma B_0} g^p\, d\mu\approx
\sum_{k\in\mathbb{Z}}2^{kp}\mu(E_k\setminus E_{k-1}).
\end{equation}
Moreover, if $p\in(0,Q/s)$,  for any  $\gamma\in\bbbr$, one has
\begin{equation}
\label{eq6}
\int_{B_0} |u-\gamma|^{p^*}\, d\mu \leq\sum_{k\in\mathbb{Z}} a_k^{p^*}\mu(B_0\cap(E_k\setminus E_{k-1})),
\end{equation}
where,  for any $k\in\mathbb{Z}$,
\begin{eqnarray}
\label{eq43}
a_k:=\left\{
\begin{array}{ll}
\sup\limits_{E_k\cap B_0} |u-\gamma|
\quad&
\text{if }\, E_k\cap B_0\neq\emptyset,\\[15pt]
\qquad0
&\text{otherwise.}
\end{array}
\right.
\end{eqnarray}
Here and thereafter, $\sup$ denotes the essential supremum.
Clearly, $a_k\leq a_{k+1}$ for any $k\in\zz$.

The idea of the proof in the case $p\in(0,Q/s)$ is
to estimate the series in \eqref{eq6}
by the series in \eqref{eq5}. A similar approach
is also used in the cases $p=Q/s$ and $p>Q/s$. In the sequel,
we need the following elementary result, whose proof  is  similar to that of \cite[Lemma 8]{agh20},
and so we omit the details.

\begin{lemma}
\label{joasia}
For any $x\in X$ and $r\in(0,\fz)$, if $B_\rho(x,r)\subset\sigma B_0$
and $\mu\big(B_\rho(x,r)\big)\geq 2\mu(\sigma B_0\setminus E_k)$ for some $k\in\bbbz$, then
$$
\mu(B_\rho(x,r)\cap E_k)\geq\frac{1}{2}\mu(B_\rho(x,r))>0.
$$
\end{lemma}
%

Moving on, fix any $\beta\in\big(0,(\log_2C_\rho)^{-1}\big)$ and let $k_0$ be the least integer such that
\begin{equation*}
2^{k_0}\geq \left[\frac{2^{1/Q}\widetilde{C}_\rho C_\rho^3}{\sigma(1-2^{-p\beta/Q})^{1/\beta}}\right]^{Q/p}\Big(bR_0^Q\Big)^{-1/p}\left(\, \int_{\sigma B_0} g^p\, d\mu\right)^{1/p}
\end{equation*}
or, equivalently,
\begin{equation}
\label{eq2}
2^{-k_0p/Q}\,\frac{2^{1/Q}\widetilde{C}_\rho C_\rho^3\,b^{-1/Q}}{\big(1-2^{-p\beta/Q}\big)^{1/\beta}}\,\left(\,\int_{\sigma B_0} g^p\,d\mu\right)^{1/Q}\leq \sigma R_0.
\end{equation}
Clearly,
\begin{equation}
\label{eq16}
2^{k_0}\approx \Big(bR_0^Q\Big)^{-1/p}\left(\, \int_{\sigma B_0} g^p\, d\mu\right)^{1/p},
\end{equation}
where the positive equivalence constants depend on $Q$, $p$, $\sigma$, $\rho$, and $\beta$.
Moreover, we have the following variant of \cite[Lemma 9]{agh20}.

\begin{lemma}
\label{piotr}
Under the above assumptions, one has  $\mu(E_{k_0})\geq\mu(\sigma B_0)/2$.
\end{lemma}
\begin{proof}
Suppose to the contrary that $\mu(E_{k_0})<\mu(\sigma B_0)/2$. Then
\begin{equation}
\label{eq41}
\mu(\sigma B_0\setminus E_{k_0})>\mu(\sigma B_0)/2.
\end{equation}
By \eqref{eq7} and \eqref{eq2}, we find that
\begin{equation*}
\begin{split}
r:=& \, 2^{1/Q}\widetilde{C}_\rho C_\rho^3\,b^{-1/Q}[\mu(\sigma B_0\setminus E_{k_0})]^{1/Q}\\
\leq& \,2^{1/Q}\widetilde{C}_\rho C_\rho^3\, b^{-1/Q} 2^{-k_0p/Q}\left(\,\int_{\sigma B_0} g^p\, d\mu\right)^{1/Q}\\
\leq& \,\sigma\lf(1-2^{-p\beta/Q}\r)^{1/\beta} R_0<\sigma R_0.
\end{split}
\end{equation*}
Therefore, if $z_0\in X$ is the center of the ball $B_0$, then $B_\rho(z_0,r)\subset\sigma B_0$, so the $V(\sigma B_0,Q,b)$ condition and \eqref{eq41} give (keeping in mind $C_\rho,\widetilde{C}_\rho\geq1$)
$$
\mu(\sigma B_0)\geq \mu(B_\rho(z_0,r))\geq br^Q=2\Big(\widetilde{C}_\rho C_\rho^3\Big)^Q\,\mu(\sigma B_0\setminus E_{k_0})
\geq 2\mu(\sigma B_0\setminus E_{k_0})>\mu(\sigma B_0),
$$
which is an obvious contradiction. This finishes the proof of Lemma \ref{piotr}.
\end{proof}

Now we prove that, for any integer  $k>k_0$,
\begin{equation}
\label{eq44}
a_k\lesssim b^{-s/Q}\left(\,\int_{\sigma B_0} g^p\, d\mu\right)^{s/Q}\,
\sum_{j=k_0}^{k-1} 2^{j(1-sp/Q)} + \sup_{E_{k_0}}|u-\gamma|.
\end{equation}
To see this, let $k\in\mathbb{Z}$ satisfy $k>k_0$. Observe that, if $E_k\cap B_0=\emptyset$, then
 $a_k=0$ due to its definition, and \eqref{eq44} is trivially true.
So we only need to consider the case  $E_k\cap B_0\neq\emptyset$ below. In this case, $a_k=\sup_{E_k\cap B_0} |u-\gamma|$. Now, for any $i\in \{0,1,\ldots,k-k_0-1\}$,  define
$$
r_{k-i}:=2^{1/Q}b^{-1/Q}2^{-[k-(i+1)]p/Q}\left(\,\int_{\sigma B_0}g^p\, d\mu\right)^{1/Q}.
$$
Then,  by \eqref{eq2}, we conclude that
\begin{align}
\label{michal}
&\lf[(r_k)^\beta+(r_{k-1})^\beta+\ldots+(r_{k_0+1})^\beta\r]^{1/\beta}\nonumber\\
&\quad=
2^{1/Q}b^{-1/Q}\left(\,\int_{\sigma B_0} g^p\, d\mu\right)^{1/Q}\left(\sum_{i=0}^{k-k_0-1} 2^{-[k-(i+1)]p\beta/Q}\right)^{1/\beta}\nonumber\\
&\quad<
2^{-k_0 p/Q}\frac{2^{1/Q}b^{-1/Q}}{\lf(1-2^{-p\beta/Q}\r)^{1/\beta}}\left(\,\int_{\sigma B_0} g^p\, d\mu\right)^{1/Q}
\leq \frac{\sigma R_0}{\widetilde{C}_\rho C_\rho^3}.
\end{align}
The importance of \eqref{michal} will reveal itself shortly. Since $E_k\cap B_0\neq\emptyset$, we can choose a point $x_k\in E_k\cap B_0$ arbitrarily. We   now use induction with respect to $i$ to define a sequence $x_{k-i}\in\sigma B_0$, $i\in\{1,\ldots,k-k_0\}$, such that
$$
x_{k-1}\in E_{k-1}\cap B_\rho(x_k,r_k),\quad
x_{k-2}\in E_{k-2}\cap B_\rho(x_{k-1},r_{k-1}),\quad\ldots,\quad
x_{k_0}\in E_{k_0}\cap B_\rho(x_{k_0+1},r_{k_0+1}).
$$
To be precise, for $i=1$,  we choose $x_{k-1}$ as follows. First,
by \eqref{michal}, we find that, for any point $y\in B_\rho(x_k,r_k)$,
\begin{align*}
\rho(z_0,y)
&\leq C_\rho\max\{\rho(z_0,x_k),\rho(x_k,y)\}<C_\rho\max\{R_0,r_k\}\\
&\leq \max\{C_\rho R_0,\sigma R_0(\widetilde{C}_\rho C_\rho^2)^{-1}\}
\leq \max\{C_\rho R_0,\sigma R_0\}=\sigma R_0,
\end{align*}
where we have also used the fact that $\sigma\geq C_\rho$ and $C_\rho,\widetilde{C}_\rho\geq1$.
This   implies $B_\rho(x_k,r_k)\subset \sigma B_0$.
 As such, applying the $V(\sigma B_0,Q,b)$ condition  and  \eqref{eq7}, we conclude that
$$
\mu(B_\rho(x_k,r_k))\geq br_k^Q=2\cdot 2^{-(k-1)p}\int_{\sigma B_0} g^p\, d\mu\geq2\mu(\sigma B_0\setminus E_{k-1}),
$$
which, together with  Lemma~\ref{joasia},
further  implies that $\mu (E_{k-1}\cap B_\rho(x_k,r_k))>0$ and, hence,
we can find a point $x_{k-1}\in E_{k-1}\cap B_\rho(x_k,r_k)$.
Clearly $x_{k-1}\in\sigma B_0$.
If $k-k_0=1$, then we are done, so suppose that $k-k_0>1$
and assume that we already selected point
$x_{k-1},\ldots, x_{k-i}$ for some $1\leq i<k-k_0$ satisfying
\begin{equation*}
x_{k-j}\in \sigma B_0\cap E_{k-j}\cap B_\rho(x_{k-j+1},r_{k-j+1})
\quad
\mbox{for\ any\ $j\in\{1,\ldots,i\}$}.
\end{equation*}
It remains to  select
$$
x_{k-(i+1)}\in \sigma B_0\cap E_{k-(i+1)}\cap B_\rho(x_{k-i},r_{k-i}).
$$
By Remark~\ref{triquasi} and \eqref{michal}, we find that, for any $y\in B_\rho(x_{k-i},r_{k-i})$,
\begin{align*}
\rho(y,x_k)&
\leq \widetilde{C}_\rho C_\rho^2\lf\{[\rho(y,x_{k-i})]^\beta+[\rho(x_{k-i},x_{k-i+1})]^\beta+\cdots+ [\rho(x_{k-1},x_k)]^\beta\r\}^{1/\beta}\\
&<
\widetilde{C}_\rho C_\rho^2\lf[(r_{k-i})^\beta+(r_{k-i+1})^\beta+\ldots+(r_k)^\beta\r]^{1/\beta}\leq \frac{\sigma R_0}{C_\rho},
\end{align*}
which, together with $x_k\in B_0$, implies that
\begin{equation*}
\rho(z_0,y)\leq C_\rho\max\{\rho(z_0,x_k),\rho(x_k,y)\}<C_\rho\max\{R_0,\sigma R_0/C_\rho\}
=\max\{C_\rho R_0,\sigma R_0\}=\sigma R_0.
\end{equation*}
Thus, $B_\rho(x_{k-i},r_{k-i})\subset\sigma B_0$, and then the $V(\sigma B_0,Q,b)$ condition  and  \eqref{eq7} imply
$$
\mu(B_\rho(x_{k-i},r_{k-i}))\geq br_{k-i}^Q=2\cdot 2^{-[k-(i+1)]p}\int_{\sigma B_0} g^p\, d\mu\geq2\mu(\sigma B_0\setminus E_{k-(i+1)}).
$$
Applying Lemma~\ref{joasia}, we have $\mu(E_{k-(i+1)}\cap B_\rho(x_{k-i},r_{k-i}))>0$ and  can find a point
$$
x_{k-(i+1)}\in E_{k-(i+1)}\cap B_\rho(x_{k-i},r_{k-i}).
$$
Clearly $x_{k-(i+1)}\in\sigma B_0$. That finishes the inductive argument on the choose
of $\{x_{k_0},\ldots,x_{k-1}\}$.

Notice that, for any $i\in\{0,1,\ldots,k-k_0-1\}$,
\begin{equation*}
\begin{split}
\rho(x_{k-i},x_{k-(i+1)})<r_{k-i}
&=2^{1/Q}b^{-1/Q}2^{-[k-(i+1)]p/Q}\left(\,\int_{\sigma B_0} g^p\, d\mu\right)^{1/Q}.
\end{split}
\end{equation*}

Also, recall that, by \eqref{eq8}, $u$ restricted to $E_{k-i}$ is $2^{k-i+1}$-H\"older continuous of
order $s$. From these and the fact that $x_{k-i},\,x_{k-(i+1)}\in E_{k-i}$
and $x_{k_0}\in E_{k_0}$, it follows that
\begin{align}
\label{eq45}
|u(x_k)-\gamma|
&\leq
\sum_{i=0}^{k-k_0-1}|u(x_{k-i})-u(x_{k-(i+1)})| + |u(x_{k_{0}})-\gamma|\nonumber\\
&\leq
\sum_{i=0}^{k-k_0-1} 2^{k-i+1}[\rho(x_{k-i},x_{k-(i+1)})]^s + \sup_{E_{k_0}}|u-\gamma|\nonumber\\
&\lesssim
2^{s/Q}b^{-s/Q}\left(\,\int_{\sigma B_0} g^p\, d\mu\right)^{s/Q}\,
\sum_{i=0}^{k-k_0-1} 2^{[k-(i+1)](1-sp/Q)}+\sup_{E_{k_0}}|u-\gamma|\nonumber\\
&\lesssim b^{-s/Q}\left(\,\int_{\sigma B_0} g^p\, d\mu\right)^{s/Q}\,
\sum_{j=k_0}^{k-1} 2^{j(1-sp/Q)} + \sup_{E_{k_0}}|u-\gamma|.
\end{align}
Since $x_k\in E_k\cap B_0$ was selected arbitrarily, taking the supremum in \eqref{eq45} over all $x_k\in E_k\cap B_0$, we  obtain the desired estimate in \eqref{eq44}.

Observe that, on the one hand,  since
$$
\rho(x,y)\leq C_\rho\max\{\rho(x,z_0),\rho(z_0,y)\}
\leq C_\rho\max\{\widetilde{C}_\rho \sigma R_0,\sigma R_0\}= C_\rho\widetilde{C}_\rho \sigma R_0,\quad \forall\ x,\,y\in\sigma B_0,
$$
one has
\begin{equation}\label{estdia}
{\rm diam}_\rho(\sigma B_0)\leq \widetilde{C}_\rho C_\rho\sigma R_0.
\end{equation}
On the other hand, since $\mu(E_{k_0})>0$ (see Lemma~\ref{piotr}),
there exists a $y\in E_{k_0}\subset \sigma B_0$. Now,
taking $\gamma=u(y)$ and applying the H\"older continuity \eqref{eq8}, \eqref{gabn-T.2},
and \eqref{eq16}, we conclude that
\begin{align}\label{eq17}
\sup_{E_{k_0}}|u-\gamma|
&\leq 2^{k_0+1}\big[{\rm diam}_\rho(\sigma B_0)\big]^s \leq
2^{k_0+1}\big[\widetilde{C}_\rho C_\rho\sigma R_0\big]^s
\lesssim R_0^s\Big(bR_0^Q\Big)^{-1/p}\left(\, \int_{\sigma B_0} g^p\, d\mu\right)^{1/p},
\end{align}
where the implicit positive
constant depends only on $Q$, $p$, $\sigma$, $s$, $\rho$, and $\beta$ (which ultimately depends on $\rho$).

Now we are ready to prove Theorem~\ref{embedding}(a).

\vskip.1in

\noindent
{\em Proof of Theorem~\ref{embedding}{\rm(a)}.} In this case, $p\in(0,Q/s)$.
First we   prove   \eqref{eq19}.
Let $\{a_k\}_{k\in\mathbb{Z}}$ be as in \eqref{eq43} with $\gamma\in\rr$.
Since $p\in(0,Q/s)$,
\eqref{eq44} gives
$$
a_k\leq Cb^{-s/Q}\left(\, \int_{\sigma B_0} g^p\, d\mu\right)^{s/Q} 2^{k(1-sp/Q)}
+\sup_{E_{k_0}} |u-\gamma|
\quad
\text{for any integer $k>k_0$,}
$$
where $C$ is a positive constant depending only on $s$, $p$, and $Q$.
However, since
$$
a_k=\sup_{E_k\cap B_0}|u-\gamma|\leq \sup_{E_{k_0}\cap\sigma B_0}|u-\gamma|=\sup_{E_{k_0}}|u-\gamma|
\quad
\text{whenever\ $k\leq k_0$}
$$
we actually have
\begin{equation}\label{ess}
a_k\leq Cb^{-s/Q}\left(\, \int_{\sigma B_0} g^p\, d\mu\right)^{s/Q} 2^{k(1-ps/Q)}+\sup_{E_{k_0}} |u-\gamma|
\quad  \forall\  k\in\bbbz.
\end{equation}

Now, we take $\gamma=u(y)$   as in \eqref{eq17}.
Then, by \eqref{eq6},
\eqref{ess}, \eqref{eq5},    and \eqref{eq17}, we find that
\begin{align}\label{repe}
\int_{B_0} |u-\gamma|^{p^*}\, d\mu
&\leq\sum_{k\in\mathbb{Z}} a_k^{p^*}\mu(B_0\cap(E_k\setminus E_{k-1}))\nonumber\\
&\ls
 b^{-sp^*/Q}\left(\, \int_{\sigma B_0} g^p\, d\mu\right)^{sp^*/Q}\sum_{k\in\mathbb{Z}} 2^{kp}\mu(E_{k}\setminus E_{k-1})
+ \left(\sup_{E_{k_0}}|u-\gamma|\right)^{p^*}\mu(B_0)\nonumber\\
&\ls
b^{-sp^*/Q}\left(\, \int_{\sigma B_0} g^p\, d\mu\right)^{p^*/p}
+ \left(\sup_{E_{k_0}}|u-\gamma|\right)^{p^*}\mu(B_0)\nonumber\\
&\ls
b^{-sp^*/Q}\left[1+\frac{\mu(B_0)}{bR_0^Q}\right]\left(\, \int_{\sigma B_0} g^p\, d\mu\right)^{p^*/p}\nonumber\\
&\ls
b^{-sp^*/Q}\frac{\mu(B_0)}{bR_0^Q}\left(\, \int_{\sigma B_0} g^p\, d\mu\right)^{p^*/p},
\end{align}
where, in the last inequality, we used the $V(\sigma B_0,Q,b)$ condition to estimate
$$1+\mu(B_0)/(bR_0^Q)\leq 2\mu(B_0)/(bR_0^Q).$$ It is now straightforward to
check that the above estimate of $\int_{B_0} |u-\gamma|^{p^*}\, d\mu$ implies \eqref{eq19}.

To finish the proof of the statement in (a), it remains to establish inequality \eqref{eq18}.
Now we take $\gamma =0$ and let $b_{k_0}:=\inf_{E_{k_0}}|u|$. Since $b_{k_0}^p{\bf 1}_{E_{k_0}}\leq |u|^p{\bf 1}_{\sigma B_0},$
Lemma~\ref{piotr} gives
$$
\frac{\mu(\sigma B_0)}{2} b_{k_0}^p\leq
b_{k_0}^p\mu(E_{k_0}) \leq
\int_{\sigma B_0} |u|^p\, d\mu
$$
and so
$$
b_{k_0}\leq 2^{1/p}\left(\, \mvint_{\sigma B_0}|u|^p\, d\mu\right)^{1/p}\, .
$$
From this, the fact $E_{k_0}\subset \sigma B_0$,
the H\"older continuity \eqref{eq8}, \eqref{estdia}, and \eqref{eq16}, we deduce that
\begin{align}\label{sim-e}
\sup_{E_{k_0}}|u|
&\leq 2^{k_0+1}\lf[{\rm diam}_\rho(\sigma B_0)\r]^s + b_{k_0}\leq
2^{k_0+1}\big[\widetilde{C}_\rho C_\rho\sigma R_0\big]^s+ 2^{1/p}\left(\, \mvint_{\sigma B_0} |u|^p\, d\mu\right)^{1/p}\nonumber\\
&\ls  R_0^s\lf(bR_0^Q\r)^{-1/p}\left(\, \int_{\sigma B_0} g^p\, d\mu\right)^{1/p} + \left(\, \mvint_{\sigma B_0} |u|^p\, d\mu\right)^{1/p}.
\end{align}
Repeating the proof of \eqref{repe} with $\gamma=0$, and using \eqref{sim-e}
instead of \eqref{eq17},   we find that
\begin{align*}
\int_{B_0} |u|^{p^*}\, d\mu
&\ls
b^{-sp^*/Q}\left(\, \int_{\sigma B_0} g^p\, d\mu\right)^{p^*/p}
+ \left(\sup_{E_{k_0}}|u|\right)^{p^*}\mu(B_0)\nonumber\\
&\ls b^{-sp^*/Q}\left[1+\frac{\mu(B_0)}{bR_0^Q}\right]\left(\, \int_{\sigma B_0} g^p\, d\mu\right)^{p^*/p} +
\left(\, \mvint_{\sigma B_0} |u|^p\, d\mu\right)^{p^*/p}\mu(B_0)\nonumber\\
&\ls
b^{-sp^*/Q}\frac{\mu(B_0)}{bR_0^Q}\left(\, \int_{\sigma B_0} g^p\, d\mu\right)^{p^*/p}+
\left(\, \mvint_{\sigma B_0} |u|^p\, d\mu\right)^{p^*/p}\mu(B_0).
\end{align*}
This estimate easily implies  $u\in L^{p^*}(B_0)$ and \eqref{eq18}.
The item (a) of Theorem~\ref{embedding} is then proved.

\vskip.1in	
Next we turn to prove the items (b) and (c) of Theorem~\ref{embedding}. Observe that, in both cases,
$u \in\dot{M}^{s,p}(\sigma B_0,\rho,\mu)$ for some $p\in [Q/s,\fz)$, which implies $u\in\dot{M}^{s,q}(\sigma B_0,\rho,\mu)$ with $q=Q/(Q+s)$, due to the H\"older inequality. Thus, by the already proved Theorem~\ref{embedding}({\rm a}), one has $u\in L^{q^*}(B_0)=L^1(B_0)$, which implies that $u_{B_0}$ is well defined and finite.
\vskip.1in
	
\noindent {\em Proof of Theorem~\ref{embedding}{\rm(b)}.} In this case, $p=Q/s$. Let $\gamma$   be as in \eqref{eq17} and $\{a_k\}_{k\in\mathbb{Z}}$
as in \eqref{eq43}. For any $a\in(0,\fz)$, by the Jensen inequality and the convexity of $e^t$, we have
\begin{align*}
\mvint_{B_0}e^{a|u(x)-u_{B_0}|}\, d\mu(x)
&\leq
\mvint_{B_0}\exp\left(\, \mvint_{B_0}a|u(x)-u(z)|\, d\mu(z)\right)\, d\mu(x)\\
&\leq
\mvint_{B_0}\mvint_{B_0}e^{a|u(x)-u(z)|}d\mu(z)\, d\mu(x)
\leq
\mvint_{B_0}e^{a|u(x)-\gamma|}\, d\mu(x)\,
\mvint_{B_0}e^{a|u(z)-\gamma|}\, d\mu(z) \\
&=
\left\{\, \mvint_{B_0}e^{a|u(x)-\gamma|}\, d\mu(x)\right\}^{2}.
\end{align*}
Hence
\begin{equation}
\label{eq23}
\mvint_{B_0}\exp\left(C_1b^{s/Q}\frac{|u-u_{B_0}|}{\|g\|_{L^{Q/s}(\sigma B_0)}}\right)\, d\mu
\leq
\left\{\, \mvint_{B_0}\exp\left(C_1b^{s/Q}\frac{|u-\gamma|}{\|g\|_{L^{Q/s}(\sigma B_0)}}\right)\, d\mu\right\}^{2},
\end{equation}
where $C_1\in(0,\fz)$ is arbitrary and will be determined later. Thus, to prove \eqref{eq20-X},
it suffices to estimate the right hand side of \eqref{eq23}.
Since $p=Q/s$, inequality \eqref{eq17} reads as
\begin{equation}
\label{eq24}
\sup_{E_{k_0}} |u-\gamma|\le C b^{-s/Q}\,\|g\|_{L^{Q/s}(\sigma B_0)}
\end{equation}
for some positive constant $C$,
which, together with  \eqref{eq44}, implies that
\begin{equation}
\label{eq25}
a_k\le  \wz C b^{-s/Q}\,\|g\|_{L^{Q/s}(\sigma B_0)}\,(k-k_0)
\quad\text{for any integer \ $k>k_0$,}
\end{equation}
where $\wz C$ is a positive constant independent of $k$.
From  \eqref{eq43} and \eqref{eq25}, it follows that
\begin{equation}
\label{eq27}
\frac{C_1b^{s/Q}|u(x)-\gamma|}{\|g\|_{L^{Q/s}(\sigma B_0)}}\leq
 C_1 \wz C(k-k_0)
\quad
\text{for any integer  $k>k_0$ and any $x\in E_k\cap B_0$}.
\end{equation}
Next we take the positive constant  $C_1$ satisfying   that $\exp(\widetilde{C}C_1)=2^{Q/s}$.
	
Let us split the integral in the right-hand side of \eqref{eq23}   into two integrals as follows
\begin{align*}
\mvint_{B_0}\exp\left(\frac{C_1 b^{s/Q}|u-\gamma|}{\|g\|_{L^{Q/s}(\sigma B_0)}}\right)\, d\mu
&=\frac{1}{\mu(B_0)}\int_{B_0\cap E_{k_0}}\exp\left(\frac{C_1 b^{s/Q}|u-\gamma|}{\|g\|_{L^{Q/s}(\sigma B_0)}}\right)\, d\mu +\,\, \frac{1}{\mu(B_0)}\int_{B_0\setminus E_{k_0}}\cdots\\
&=: {\rm I}_1+{\rm I}_2.
\end{align*}
Estimate \eqref{eq24} gives
$$
{\rm I}_1\leq\frac{\mu(B_0\cap E_{k_0})}{\mu(B_0)}\exp(CC_1)\leq\exp(CC_1),
$$
while   \eqref{eq27}, \eqref{eq42}, \eqref{eq5}, and the fact  $\exp(\widetilde{C}C_1)=2^{Q/s}=2^p$ imply that
\begin{equation*}
\begin{split}
{\rm I}_2
&\leq
\frac{1}{\mu(B_0)}\sum_{k=k_0+1}^\infty
e^{\widetilde{C}C_1(k-k_0)}\mu\lf(B_0\cap(E_k\setminus E_{k-1})\r)\\
&\leq \frac{2^{-k_0Q/s}}{\mu(B_0)}\sum_{k=-\infty}^\infty 2^{kQ/s}\mu(E_k\setminus E_{k-1})
\ls \frac{2^{-k_0Q/s}}{\mu(B_0)}\|g\|_{L^{Q/s}(\sigma B_0)}^{s/Q}
\ls\frac{bR_0^Q}{\mu(B_0)}\ls 1,
\end{split}
\end{equation*}
where the last two estimates follow from \eqref{eq16} and the $V(\sigma B_0,Q,b)$ condition, respectively.
This finishes the proof of Theorem~\ref{embedding}(b).

\vskip.1in

\noindent
{\em Proof of Theorem~\ref{embedding}{\rm(c)}.}
In this case, $p\in(Q/s,\fz)$. Let $\gamma$   be as in \eqref{eq17} and $\{a_k\}_{k\in\mathbb{Z}}$
as in \eqref{eq43} again.
Since $2^{1-sp/Q}<1$, by \eqref{eq44},
we find that,  for any integer  $k> k_0$,
\begin{equation}
\label{eq28}
a_k\ls b^{-s/Q}\left(\, \int_{\sigma B_0} g^p\, d\mu\right)^{s/Q} 2^{k_0(1-sp/Q)}
+\sup_{E_{k_0} }|u-\gamma|.
\end{equation}
From \eqref{eq28}, \eqref{eq16}, and \eqref{eq17}, we deduce that,  for any integer  $k> k_0$,
\begin{equation}\label{kysr}
a_k\ls \lf(bR_0^Q\r)^{-1/p}R_0^s\left(\, \int_{\sigma B_0}g^p\, d\mu\right)^{1/p}
\approx \left[\frac{\mu(\sigma B_0)}{bR_0^Q}\right]^{1/p}
R_0^s\left(\, \mvint_{\sigma B_0} g^p\, d\mu\right)^{1/p}.
\end{equation}
Noticing that  the right hand side of \eqref{kysr}  is a positive
constant independent of $k$, by \eqref{eq42} and
the definition of $a_k$ in \eqref{eq43}, we conclude that
$|u-\gamma|$ equals almost everywhere to a  function
on $B_0$ which is bounded by  the right hand side of \eqref{kysr}
modulo a positive constant. Therefore,
we obtain
$$
\Vert u-u_{B_0}\Vert_{L^\infty(B_0)}\leq 2\Vert u-\gamma\Vert_{L^\infty(B_0)}
\ls \left[\frac{\mu(\sigma B_0)}{bR_0^Q}\right]^{1/p}
R_0^s\left(\, \mvint_{\sigma B_0} g^p\, d\mu\right)^{1/p}.
$$
This proves \eqref{eq29}.

Next we show  the  H\"older continuity of $u$ along with the estimate \eqref{eq30}.
To this end, let $A_0:=\{x\in B_0:\ |u(x)|\le \|u\|_{L^\fz(B_0)}\}$. It follows from
\eqref{eq29} that $\mu(B_0\setminus A_0)=0$.

Fix $x,\,y\in A_0$. If $2\rho(x,y)\leq (C_\rho)^{-1}R_0$ with $C_\rho$ being as in \eqref{C-RHO.111},   let
$R_1:=2\rho(x,y)$. Clearly,   $x,\,y\in B_1:=B_\rho(x,R_1)$. Moreover,
if $z_0$ denotes the center of $B_0$, then, for any $z\in \sigma B_1$, we have (keeping in mind that $\sigma\geq C_\rho$)
\begin{align*}
\rho(z_0,z)&\leq C_\rho\max\{\rho(z_0,x),\rho(x,z)\}< C_\rho\max\{R_0,\sigma R_1\}\\
&\leq C_\rho\max\{R_0,\sigma (C_\rho)^{-1}R_0\}=\sigma R_0,
\end{align*}
which implies that $\sigma B_1\subset\sigma B_0$.
From this and the assumption that $\mu$ satisfies the $V(\sigma B_0,Q,b)$ condition,
it follows that $\mu$ also satisfies the $V(\sigma B_1,Q,b)$ condition,
and the estimate \eqref{eq29} applied to $B_1$ in place of $B_0$, together with the definition of $A_0$, gives
\begin{equation*}
\begin{split}
|u(x)-u(y)|
&\leq
2\Vert u-u_{B_1}\Vert_{L^\infty(B_1)}\ls \left[\frac{\mu(\sigma B_1)}{bR_1^Q}\right]^{1/p}R_1^s\, \left(\, \mvint_{\sigma B_1}g^p\, d\mu\right)^{1/p}\\
&\approx b^{-1/p}[\rho(x,y)]^{s-Q/p}\left(\, \int_{\sigma B_1} g^p\, d\mu\right)^{1/p}
\ls b^{-1/p}[\rho(x,y)]^{s-Q/p}\left(\, \int_{\sigma B_0} g^p\, d\mu\right)^{1/p}.
\end{split}
\end{equation*}
If $2\rho(x,y)>(C_\rho)^{-1}R_0$, then \eqref{eq29} immediately gives
\begin{equation*}
\begin{split}
|u(x)-u(y)|
&\leq
2\Vert u-u_{B_0}\Vert_{L^\infty(B_0)}
\ls\left[\frac{\mu(\sigma B_0)}{bR_0^Q}\right]^{1/p}R_0^s\, \left(\, \mvint_{\sigma B_0}g^p\, d\mu\right)^{1/p}\\
&\ls b^{-1/p}[\rho(x,y)]^{s-Q/p}\left(\, \int_{\sigma B_0} g^p\, d\mu\right)^{1/p}.
\end{split}
\end{equation*}
These estimates imply \eqref{eq30}, and  the proof of Theorem~\ref{embedding} is finished.
\end{proof}

As an immediate consequence of Theorem~\ref{embedding} and the inclusions in Proposition~\ref{sobequal},
we   obtain the following local embeddings for $\dot{M}^s_{p,q}$ and $\dot{N}^s_{p,q}$,
under the assumption that the measure is lower Ahlfors-regular.

\begin{theorem}
\label{LBembedding}
Let $(X,\rho,\mu)$ be a quasi-metric measure space and suppose that there exist
positive constants $\kappa$ and $Q$ satisfying
\begin{equation}\label{lowermeasure-thm}
\kappa\,r^Q\leq\mu(B_\rho(x,r))  \qquad \mbox{for any}\ x\in X\
\mbox{and any finite}\ r\in(0,{\rm diam}_\rho(X)].
\end{equation}
Let $s,\,p\in(0,\infty)$, $q\in(0,\infty]$, and $\sigma\in[C_\rho,\infty)$
with $C_\rho$  as in \eqref{C-RHO.111}. Then there exist positive constants $C$, $C_1$, and $C_2$,
depending  on $\rho$, $\kappa$, $Q$, $s$, $p$, $q$, and $\sigma$, such that the following
statements hold true for any $\rho$-ball $B_0:=B_\rho(x_0,R_0)$, where
$x_0\in X$ and $R_0\in(0,{\rm diam}_\rho(X)]$ is finite.
\begin{enumerate}
\item[{\rm(a)}] If $p\in(0,Q/s)$, then, for any $u\in\dot{M}^s_{p,q}(\sigma B_0,\rho,\mu)$, one has
$u\in L^{p^*}(B_0)$, where $p^*:=Qp/(Q-sp)$, and the following inequalities are satisfied:
\begin{equation}
\label{eq18-LB}
\left(\, \mvint_{B_0} |u|^{p^*}\, d\mu\right)^{1/p^*}\leq
CR_0^{-Q/p}\left[
R_0^s\Vert u\Vert_{\dot{M}^s_{p,q}(\sigma B_0)}
+\Vert u\Vert_{L^{p}(\sigma B_0)}\right],
\end{equation}
and
\begin{equation}
\label{eq19-LB}
\inf_{\gamma\in\mathbb{R}}\left(\, \mvint_{B_0} |u-\gamma|^{p^*}\, d\mu\right)^{1/p^*}\leq
CR_0^{s-Q/p}\Vert u\Vert_{\dot{M}^s_{p,q}(\sigma B_0)}.
\end{equation}

\item[{\rm(b)}] If $p=Q/s$, then, for any $u\in\dot{M}^s_{p,q}(\sigma B_0,\rho,\mu)$ with
$\Vert u\Vert_{\dot{M}^s_{p,q}(\sigma B_0)}>0$, one has
\begin{equation}
\label{eq20-LB}
\mvint_{B_0} {\rm exp}\left(\frac{C_1|u-u_{B_0}|}{\Vert u\Vert_{\dot{M}^s_{p,q}(\sigma B_0)}}\right)\,d\mu\leq C_2.
\end{equation}

\item[{\rm(c)}] If $p>Q/s$, then each function $u\in\dot{M}^s_{p,q}(X,\rho,\mu)$
has a H\"older continuous representative of order $s-Q/p$ on $X$, denoted by $u$ again, satisfying
\begin{equation}
\label{eq30-LB}
|u(x)-u(y)|\leq C[\rho(x,y)]^{s-Q/p}\Vert u\Vert_{\dot{M}^s_{p,q}(X)},
\quad \forall\ x,\,y\in X.
\end{equation}
\end{enumerate}
In addition, if $q\leq p$, then all of the statements above are valid with  $\dot{M}^s_{p,q}$ replaced by $\dot{N}^s_{p,q}$.
\end{theorem}

\begin{remark}
\label{embedlocalLB}
In the context of Theorem~\ref{LBembedding}, if, instead of the lower measure bound \eqref{lowermeasure-thm}, one assumes that there exists some $r_\ast\in(0,\infty)$ such that $r^Q\lesssim\mu(B_\rho(x,r))$ for any $x\in X$ and any $r\in(0,r_\ast]$, then the items (a) and (b) of Theorem~\ref{LBembedding} still hold true for
any $\rho$-ball $B_0:=B_\rho(x_0,R_0)$, with $x_0\in X$ and $R_0\in(0,r_\ast]$.
In such a scenario, when $p>Q/s$, the global pointwise inequality in \eqref{eq30-LB} would be replaced by a local condition, where \eqref{eq30-LB} only holds true for functions $u$ in $\dot{M}^s_{p,q}(\sigma B_0,\rho,\mu)$ and points in $B_0$.
\end{remark}	

\begin{proof}[Proof of Theorem \ref{LBembedding}]
Fix a ball $B_0$ having finite radius $R_0\in\big(0,{\rm diam}_\rho(X)\big]$. If $B_\rho(x,r)\subset\sigma B_0$ and $r\in (0,\sigma R_0]$, then $\sigma^{-1}r\leq\min\{r,{\rm diam}_\rho(X)\}$,
and   \eqref{lowermeasure-thm} gives
\begin{equation*}
\mu(B_\rho(x,r))\geq\mu(B_\rho(x,\sigma^{-1}r))\geq\kappa(\sigma^{-1}r)^Q=br^Q,
\end{equation*}
where $b:=\kappa\sigma^{-Q}\in(0,\infty)$. Thus, $\mu$ satisfies the $V(\sigma B_0,Q,b)$ condition in \eqref{measbound}.
On the other hand, \eqref{MN-inclusion} and Proposition~\ref{sobequal}(i) imply that
$$\dot{M}^s_{p,q}(\sigma B_0,\rho,\mu)\hookrightarrow\dot{M}^s_{p,\infty}(\sigma B_0,\rho,\mu)
=\dot{M}^{s,p}(\sigma B_0,\rho,\mu).$$ Combining this with Proposition~\ref{constant}
permits us to deduce that,
\begin{align}
\label{qpp-509}
&\mbox{for any $u\in\dot{M}^s_{p,q}(\sigma B_0,\rho,\mu)$, there exists a}\nonumber \\
&\qquad \mbox{$g\in\mathcal{D}^s_\rho(u)$ such that }\,\, \Vert g\Vert_{L^p(\sigma B_0)}\lesssim \Vert u\Vert_{\dot{M}^s_{p,q}(\sigma B_0)}.
\end{align}
As such, \eqref{eq19-LB} follows immediately from combining \eqref{qpp-509} with \eqref{eq19}  in Theorem~\ref{embedding}.

Now we prove \eqref{eq18-LB}. Observe that, by \eqref{lowermeasure-thm},
$$
\left[\frac{\mu(\sigma B_0)}{bR_0^Q}\right]^{1/p}\geq
\left[\frac{\kappa[\sigma R_0]^Q}{bR_0^Q}\right]^{1/p}=\sigma^{2Q/p}>1,
$$
where $b=\kappa\sigma^{-Q}$, which, together with  \eqref{eq18}  and \eqref{qpp-509}, implies that
\begin{equation*}
\begin{split}
\left(\, \mvint_{B_0} |u|^{p^*}\, d\mu\right)^{1/p^*}
&\ls\left[\frac{\mu(\sigma B_0)}{bR_0^Q}\right]^{1/p}
R_0^s\left(\,\mvint_{\sigma B_0}g^p\, d\mu\right)^{1/p}
+\left(\,\mvint_{\sigma B_0}|u|^p\, d\mu\right)^{1/p}\\
&\ls \left[\frac{\mu(\sigma B_0)}{R_0^Q}\right]^{1/p}
\left[R_0^s\left(\,\mvint_{\sigma B_0}g^p\, d\mu\right)^{1/p}
+\left(\,\mvint_{\sigma B_0}|u|^p\, d\mu\right)^{1/p}\right]\\
&\approx R_0^{-Q/p}
\left[R_0^s\left(\,\int_{\sigma B_0}g^p\, d\mu\right)^{1/p}
+\left(\,\int_{\sigma B_0}|u|^p\, d\mu\right)^{1/p}\right]\\
&\ls R_0^{-Q/p}\left[
R_0^s\Vert u\Vert_{\dot{M}^s_{p,q}(\sigma B_0)}
+\Vert u\Vert_{L^{p}(\sigma B_0)}\right].
\end{split}
\end{equation*}
Hence, the inequality in \eqref{eq18-LB} is valid.

For \eqref{eq20-LB}, observe that, if $u\in\dot{M}^s_{p,q}(\sigma B_0,\rho,\mu)$ with $\Vert u\Vert_{\dot{M}^s_{p,q}(\sigma B_0)}>0$, then  $g$   in \eqref{qpp-509} satisfies $\Vert g\Vert_{L^p(\sigma B_0)}>0$. Otherwise, if $\Vert g\Vert_{L^p(\sigma B_0)}=0$, then $g=0$ $\mu$-almost everywhere in $\sigma B_0$ and it would follow that $u$ is a constant in $\sigma B_0$, and hence $\Vert u\Vert_{\dot{M}^s_{p,q}(\sigma B_0)}=0$. Thus, $g\in\mathcal{D}^{s,+}_\rho(u)$. Combining \eqref{qpp-509} with \eqref{eq20-X} in Theorem~\ref{embedding}, we obtain \eqref{eq20-LB}.

We now turn our attention to verifying  \eqref{eq30-LB}.
Since $\mu$ satisfies the $V(\sigma B_0,Q,b)$ condition with $b=\kappa\sigma^{-Q}$,
from \eqref{qpp-509} and   Theorem~\ref{embedding}(c) (specifically \eqref{eq30}), we deduce that there exists a
positive constant $C$ such that, for any $u\in\dot{M}^s_{p,q}(\sigma B_0,\rho,\mu)$,
\begin{equation}\label{Jnb.1}
|u(x)-u(y)|\leq C\,b^{-1/p}[\rho(x,y)]^{s-Q/p}\Vert u\Vert_{\dot{M}^s_{p,q}(\sigma B_0)},
\quad\ \forall\ x,\ y\in B_0.
\end{equation}
Notice that, if $u\in\dot{M}^s_{p,q}(X,\rho,\mu)$,
then its pointwise restriction to $\sigma B_0$ belongs to $\dot{M}^s_{p,q}(\sigma B_0,\rho,\mu)$
with $\|u\|_{\dot{M}^s_{p,q}(\sigma B_0)}\le \|u\|_{\dot{M}^s_{p,q}(X)}.$
Then, by \eqref{Jnb.1},  we have
\begin{equation}\label{Jnb.10}
|u(x)-u(y)|\leq C\,b^{-1/p}[\rho(x,y)]^{s-Q/p}\Vert u\Vert_{\dot{M}^s_{p,q}(X)},
 \quad\ \forall\ x,\ y\in B_0.
\end{equation}
As the positive constants $C$ and $b$ are independent of $B_0$, it follows
that \eqref{Jnb.10} implies \eqref{eq30-LB}, as desired.

The justification of the embeddings for $\dot{N}^s_{p,q}$ when $q\leq p$ follows a similar reasoning using  Proposition~\ref{sobequal}(ii) in place of Proposition~\ref{sobequal}(i). This finishes the proof of Theorem \ref{LBembedding}.
\end{proof}

The following theorem is the counterpart of Theorem~\ref{LBembedding} for doubling measures.

\begin{theorem}
\label{DOUBembedding}
Let $(X,\rho,\mu)$ be a quasi-metric measure space and suppose that there exist positive
constants $\kappa$ and $Q$ satisfying
\begin{equation}\label{Doub-thm}
\kappa\lf(\frac{r}{R}\r)^{Q}\leq\frac{\mu(B_\rho(x,r))}{\mu(B_\rho(y,R))},
\end{equation}
whenever $x,\,y\in X$ with $B_\rho(x,r)\subset B_\rho(y,R)$ and $0<r\leq R<\infty$.
Let $s,\,p\in(0,\infty)$, $q\in(0,\infty]$, and $\sigma\in[C_\rho,\infty)$
with $C_\rho$   as in \eqref{C-RHO.111}.
Then there exist positive constants $C$, $C_1$, and $C_2$, depending only on $\rho$, $\kappa$,
$Q$, $s$, $p$, and $\sigma$, such that the following statements hold true for any
$\rho$-ball $B_0:=B_\rho(x_0,R_0)$, where $x_0\in X$ and $R_0\in(0,\infty)$.
\begin{enumerate}
\item[{\rm(a)}] If $p\in(0,Q/s)$,
then, for any $u\in\dot{M}^s_{p,q}(\sigma B_0,\rho,\mu)$,
one has $u\in L^{p^*}(B_0)$, where $p^*:=Qp/(Q-sp)$,
and the following inequalities are satisfied:
\begin{equation}
\label{eq18-DOUB}
\Vert u\Vert_{L^{p^\ast}(B_0)}\leq
\frac{C}{[\mu(\sigma B_0)]^{s/Q}}\left[
R_0^s\Vert u\Vert_{\dot{M}^s_{p,q}(\sigma B_0)}
+\Vert u\Vert_{L^{p}(\sigma B_0)}\right],
\end{equation}
and
\begin{equation}
\label{eq19-DOUB}
\inf_{\gamma\in\mathbb{R}}\Vert u-\gamma\Vert_{L^{p^\ast}(B_0)}\leq
\frac{C}{[\mu(\sigma B_0)]^{s/Q}}\,R_0^{s}\Vert u\Vert_{\dot{M}^s_{p,q}(\sigma B_0)}.
\end{equation}

\item[{\rm(b)}] If $p=Q/s$, then, for any $u\in\dot{M}^s_{p,q}(\sigma B_0,\rho,\mu)$
with $\Vert u\Vert_{\dot{M}^s_{p,q}(\sigma B_0)}>0$, one has
\begin{equation*}
\mvint_{B_0} {\rm exp}\left(C_1\frac{[\mu(\sigma B_0)]^{s/Q}|u-u_{B_0}|}{R_0^s\Vert u\Vert_{\dot{M}^s_{p,q}(\sigma B_0)}}\right)\,d\mu\leq C_2.
\end{equation*}

\item[{\rm(c)}] If $p>Q/s$, then  each function $u\in\dot{M}^s_{p,q}(\sigma B_0,\rho,\mu)$ has a H\"older continuous representative of order $s-Q/p$ on $B_0$, denoted by $u$ again, satisfying
\begin{equation*}
|u(x)-u(y)|\leq C[\rho(x,y)]^{s-Q/p}\frac{R_0^{Q/p}}{[\mu(\sigma B_0)]^{1/p}}\,\Vert u\Vert_{\dot{M}^s_{p,q}(\sigma B_0)},
\quad \forall \ x,\, y\in B_0.
\end{equation*}
\end{enumerate}
In addition, if $q\leq p$, then all of the statements above are valid with  $\dot{M}^s_{p,q}$ replaced by $\dot{N}^s_{p,q}$.
\end{theorem}

\begin{proof}
Consider a $\rho$-ball  $B_0:= B_\rho(x_0,R_0)$ with $x_0\in X$ and $R_0\in(0,\infty)$, and let $B_\rho(y,R):=\sigma B_0$. Then \eqref{Doub-thm} implies that the  measure $\mu$ satisfies the $V(\sigma B_0, Q, b)$ condition
with $b:=\kappa\mu(\sigma B_0)(\sigma R_0)^{-Q}$. The proof now proceeds just as the proof of Theorem~\ref{LBembedding}, via combining the embeddings in Theorem~\ref{embedding} (for $\dot{M}^{s,p}$ spaces) with the inclusions in Proposition~\ref{sobequal}. This finishes the proof of Theorem \ref{Doub-thm}.
\end{proof}

\begin{remark}
\label{embedlocaldoub}
In the context of Theorem~\ref{DOUBembedding}, if, instead of
the doubling condition \eqref{Doub-thm}, one assumes that the measure $\mu$ is
\textit{locally $Q$-doubling up to scale $r_0\in(0,\infty)$}, namely, if one assumes
\begin{equation*}
\lf(\frac{r}{R}\r)^{Q}\lesssim\frac{\mu(B_\rho(x,r))}{\mu(B_\rho(y,R))},
\end{equation*}
whenever $x,\,y\in X$ with $B_\rho(x,r)\subset B_\rho(y,R)$ and $0<r\leq R\leq r_0$, then, for any fixed $\sigma\in[C_\rho,\infty)$,   the embeddings in  (a)-(c) of Theorem~\ref{DOUBembedding}
hold true for any $\rho$-ball $B_0:=B_\rho(x_0,R_0)$ with $x_0\in X$ and $R_0\in(0,r_0/\sigma]$.
\end{remark}

The embeddings for the Haj\l asz-Besov spaces  $\dot{N}^s_{p,q}$  in Theorems~\ref{LBembedding} and \ref{DOUBembedding} are restricted to the case when $q\leq p$; however, an upper bound on the exponent $q$ is to be expected (see Remark~\ref{BesovUB} below).  In the following theorem, we prove that one can relax the restriction on $q$ and still obtain Sobolev-type embeddings with the critical exponent $Q/s$ replaced by $Q/\varepsilon$, where $\varepsilon\in(0,s)$ is any fixed number.

\begin{theorem}
\label{mainembedding-epsilon}
Let $(X,\rho,\mu)$ be a quasi-metric measure space and suppose that
$u\in\dot{N}^s_{p,q}(\sigma B_0,\rho,\mu)$, where $s,\,p\in(0,\infty)$,
$q\in(0,\infty]$, $\sigma\in[C_\rho,\infty)$ with $C_\rho$ as in \eqref{C-RHO.111},
and $B_0$ is a $\rho$-ball of radius $R_0\in(0,\infty)$. Assume that the
measure $\mu$ satisfies the $V(\sigma B_0,Q,b)$ condition for some $Q,\,b\in(0,\infty)$.
Then, for any fixed $\varepsilon\in(0,s)$, there exist positive constants $C$, $C_1$, and $C_2$,
depending  only on $\rho$, $s$, $\varepsilon$, $p$, and $\sigma$,  such that
the following statements hold true.
\begin{enumerate}
\item[{\rm(a)}] If $p\in(0,Q/\varepsilon)$, then  $u\in L^{p^*}(B_0,\mu)$, where $p^*:=Qp/(Q-\varepsilon p)$,
and the following inequalities are satisfied:
\begin{equation*}
\left(\, \mvint_{B_0} |u|^{p^*}\, d\mu\right)^{1/p^*}\leq
Cb^{-1/p}R_0^{-Q/p}\left[
R_0^s\Vert u\Vert_{\dot{N}^s_{p,q}(\sigma B_0)}
+\Vert u\Vert_{L^{p}(\sigma B_0)}\right],
\end{equation*}
and
\begin{equation*}
\inf_{\gamma\in\mathbb{R}}\left(\, \mvint_{B_0} |u-\gamma|^{p^*}\, d\mu\right)^{1/p^*}\leq
Cb^{-1/p}R_0^{s-Q/p}\Vert u\Vert_{\dot{N}^s_{p,q}(\sigma B_0)}.
\end{equation*}

\item[{\rm(b)}] If $p=Q/\varepsilon$ and $\Vert u\Vert_{\dot{N}^s_{p,q}(\sigma B_0,\rho,\mu)}>0$, then
\begin{equation*}
\mvint_{B_0} {\rm exp}\left(C_1b^{1/p}\frac{|u-u_{B_0}|}{R_0^{s-\varepsilon}\Vert u\Vert_{\dot{N}^s_{p,q}(\sigma B_0)}}\right)\,d\mu\leq C_2.
\end{equation*}

\item[{\rm(c)}] If $p>Q/\varepsilon$, then
\begin{equation*}
\Vert u-u_{B_0}\Vert_{L^\infty(B_0)}\leq
Cb^{-1/p}R_0^{s-Q/p}\Vert u\Vert_{\dot{N}^s_{p,q}(\sigma B_0)}.
\end{equation*}
In particular, $u$ has a H\"older continuous representative of order $s-Q/p$ on $B_0$, denoted by $u$ again,  satisfying
\begin{equation*}
|u(x)-u(y)|\leq C b^{-1/p}[\rho(x,y)]^{s-Q/p}\Vert u\Vert_{\dot{N}^s_{p,q}(\sigma B_0)},
\quad \forall\ x,\,y\in B_0.
\end{equation*}
\end{enumerate}
\end{theorem}

\begin{proof}
The conclusions of this theorem are
immediate consequences of Theorem~\ref{embedding}
 and Proposition~\ref{sobequal}(iii), which gives
$\dot{N}^s_{p,q}(\sigma B_0,\rho,\mu)\hookrightarrow\dot{M}^{\varepsilon,p}(\sigma B_0,\rho,\mu)$
 with
 $$\Vert u\Vert_{\dot{M}^{\varepsilon,p}(\sigma B_0,\rho,\mu)}\lesssim R_0^{s-\varepsilon}\Vert u\Vert_{\dot{N}^s_{p,q}(\sigma B_0,\rho,\mu)},\quad  \forall \  u\in \dot{N}^s_{p,q}(\sigma B_0,\rho,\mu).$$
 This finishes the proof of  Theorem \ref{mainembedding-epsilon}.
\end{proof}

\begin{remark}
Using Theorem~\ref{mainembedding-epsilon}, one can then establish the analogues of Theorems~\ref{LBembedding} (for lower Ahlfors-regular measures) and \ref{DOUBembedding} (for doubling measures) for the Haj\l asz--Besov spaces $N^s_{p,q}$ with the full range of $q\in(0,\infty]$,
but with the critical exponent $Q/s$ replaced by $Q/\varepsilon$,
where $\varepsilon\in(0,s)$ is any fixed number.
\end{remark}

\section{From Embeddings to The Lower Bound of   Measures}
\label{section:measequiv}

In this section, we show that  the
 Sobolev-type embedding properties in any of (b)-(e)
of   Theorems \ref{LMeasINTCor}  and \ref{GlobalEmbeddCor-INT},  or (b)-(d)
of  Theorem  \ref{DoubMeasINTCor},
imply the lower bound condition of the measure in the items (a) of these theorems.
To this end, we first give some necessary tools in Subsection \ref{sssec:maintools},
 and  use them to prove our desired results, respectively,
 in Subsection~\ref{non-db} for non-doubling measures and Subsection~\ref{dbc} for doubling measures.

\subsection{Main Tools}
\label{sssec:maintools}

We begin with collecting a few technical results that were established in \cite{agh20} and \cite{AM15}.
Given the important role that these results play in the proofs of the main theorems in this article, we include their statements for the convenience of the reader. For now, the reader may skip to the next section and return to it as needed.

Let $\Omega\subset\mathbb{R}^n$ be an open connected set. For any $x\in\Omega$ and $r\in(0,\infty)$, it is always possible to find a radius $\tilde{r}<r$ such that
$\mathcal{L}^n(B(x,\tilde{r})\cap\Omega)=\frac{1}{2}\mathcal{L}^n(B(x,r)\cap\Omega)$, where $\mathcal{L}^n$ denotes the $n$-dimensional Lebesgue measure on $\mathbb{R}^n$. Such a radius $\tilde{r}$ still exists in \textit{geodesic} metric measure spaces $(X,d,\mu)$ (or, more generally, spaces where the boundary of balls have zero measure), where the natural condition that one now seeks is $\mu(B(x,\tilde{r}))=\frac{1}{2}\mu(B(x,r))$.
However, in a general quasi-metric measure space, there is
no guarantee that there exists a concentric ball with half of
the measure of the original ball. To overcome this, we consider the following suitable
replacement for $\tilde{r}$: given a quasi-metric measure space  $(X,\rho,\mu)$,
for any $x\in X$ and $r\in[0,\infty)$,  define
\begin{equation}
\label{radphi}
\mbox{$\vi^x_{\rho}(r):=\sup\lf\{s\in [0,r]:\, \mu(B_\rho(x,s))\leq\frac{1}{2}\mu(B_\rho(x,r))\r\}.$}
\end{equation}
Notice that, when $s=0$, $B_\rho(x,s)=\emptyset$ and so $\varphi^x_{\rho}(r)\geq 0$.


We now take a moment to list some basic properties of $\vi^x_{\rho}(r)$ in the lemma below. This result was established in \cite[Lemma 10]{agh20} in the context of metric measure spaces, but the proof in the quasi-metric setting is the same and,
therefore, is omitted.

\begin{lemma}
\label{Gds.24}
Let $(X,\rho,\mu)$ be a quasi-metric measure
space and fix $x\in X$ and $r\in[0,\infty)$. Then the following statements hold true.
\begin{enumerate}[label=(\roman*)]
\item[${\rm(i)}$] It holds true that
\begin{equation}
\label{LLk-3}
\mu\lf(B_\rho(x,\varphi^x_{\rho}(r))\r)\leq\frac{1}{2}\,\mu\lf(B_\rho(x,r)\r)
\leq\mu\lf(\overline{B}_\rho(x,\varphi^x_{\rho}(r))\r).
\end{equation}
\item[${\rm(ii)}$] It holds true that $\varphi^x_{\rho}(r)\in[0,r]$, and $\varphi^x_{\rho}(r)=r$ if and only $r=0$.
\end{enumerate}
\end{lemma}

Recall that a quasi-metric space $(X,\rho)$ is said to be {\it uniformly} {\it perfect}
 if there exists a constant $\lambda\in(0,1)$ with the property that, for any $x\in X$ and  $r\in(0,\infty)$,
\begin{equation}
\label{U-perf}
B_\rho(x,r)\setminus B_\rho(x,\lambda r)\neq\emptyset\quad
\mbox{ whenever }\quad X\setminus B_\rho(x,r)\neq\emptyset.
\end{equation}
It follows immediately from \eqref{U-perf} that, if \eqref{U-perf} holds true for some $\lambda\in(0,1)$, then it also holds true for any $\lambda'\in(0,\lambda]$.
As such, in what follows, we will always assume that $0<\lambda<(C_\rho\widetilde{C}_\rho)^{-2}$, where $C_\rho,\,\widetilde{C}_\rho\in[1,\infty)$ are as in \eqref{C-RHO.111} and \eqref{C-RHO.111XXX}, respectively.

Moreover, if $(X,\rho)$ is uniformly perfect, then, for any quasi-metric $\varrho$ on $X$ with $\varrho\approx\rho$,
$(X,\varrho)$ is also uniformly perfect. More specifically, if $\lambda$ is as in \eqref{U-perf}
and there exists a constant
$c\in(1,\infty)$ such that $c^{-1}\varrho\leq\rho\leq c\varrho$ pointwise on $X\times X$, then $(X,\varrho)$ is uniformly perfect with constant $\widetilde{\lambda}:=c^{-2}\lambda\in(0,1)$.

\begin{lemma}\label{tx}
Let $(X,\rho,\mu)$ be a uniformly perfect quasi-metric measure space and suppose that $\lambda\in(0,(C_\rho\widetilde{C}_\rho)^{-2})\,$ is as in \eqref{U-perf}, where $C_\rho,\,\widetilde{C}_\rho\in[1,\infty)$ are as in \eqref{C-RHO.111} and \eqref{C-RHO.111XXX}, respectively. Then, for any  $x\in X$ and any finite $r\in (0,{\rm \diam}_\rho(X)]$ satisfying $r>C_\rho\varphi^x_{\rho}(r)/\lambda^2$,  there exists a point $\tilde{x}\in X$ and a radius $\tilde{r}\in(0,\infty)$ satisfying $\lambda r\leq\tilde{r}\leq\min\{r,C_\rho\varphi^{\tilde{x}}_{\rho}(\tilde{r})/\lambda^2\}$ such that $B_\rho(\tilde{x},\tilde{r})\subset B_\rho(x,r)$.
\end{lemma}

\begin{proof}
Fix a point $x\in X$ and a finite number $r\in (0,{\rm \diam}_\rho(X)]$
satisfying $r>C_\rho\varphi^x_{\rho}(r)/\lambda^2$.

To find the point $\tilde{x}$, we first claim that $X\setminus B_\rho(x,\lambda r)\neq\emptyset$.
Indeed, if we had $B_\rho(x,\lambda r)=X$,
then, for any $z,\,y\in X$,
$$
\rho(z,y)\leq C_\rho\max\big\{\widetilde{C}_\rho\rho(x,z),\rho(x,y)\big\}<C_\rho\widetilde{C}_\rho\lambda r,
$$
from which it follows that ${\rm diam}_\rho(X)\leq C_\rho\widetilde{C}_\rho\lambda r$. Combining this,
$\lambda\in(0,(C_\rho\widetilde{C}_\rho)^{-2})$,  the fact that $r\leq{\rm \diam}_\rho(X)$, and
$C_\rho,\,\widetilde{C}_\rho\in[1,\infty)$, we find that
$$
{\rm diam}_\rho(X)\leq C_\rho\widetilde{C}_\rho\lambda r<\big(C_\rho\widetilde{C}_\rho\big)^{-1}{\rm \diam}_\rho(X)\leq{\rm \diam}_\rho(X),
$$
which is a contradiction. This proves the above claim that
$X\setminus B_\rho(x,\lambda r)\neq\emptyset$.

From the above claim  and that  $X$ is uniformly perfect, we
deduce that $ B_\rho(x,\lambda r)\setminus B_\rho(x,\lambda^2 r)\neq\emptyset.$
Now we choose a point $\tilde{x}\in B_\rho(x,\lambda r)\setminus B_\rho(x,\lambda^2 r)$ and let
$\tilde{r}:=C_\rho\widetilde{C}_\rho\lambda r$. Since $C_\rho\widetilde{C}_\rho\geq1$
and $\lambda<(C_\rho\widetilde{C}_\rho)^{-2}$, it is straightforward to show that
\begin{equation*}
\lambda r\leq\tilde{r}=C_\rho\widetilde{C}_\rho\lambda r<(C_\rho\widetilde{C}_\rho)^{-1}r\leq r.
\end{equation*}

Next we prove
\begin{equation}
\label{dx.1}
\overline{B}_\rho\lf(x,\varphi^x_{\rho}(r)\r)\subset B_\rho(\tilde{x},\tilde{r})\subset B_\rho(x,r)
\quad
\text{and}
\quad
B_\rho\lf(\tilde{x},(C_\rho\widetilde{C}_\rho)^{-2}\lambda \tilde{r}\r)\subset B_\rho(\tilde{x},\tilde{r})\setminus
\overline{B}_\rho\lf(x,\varphi^x_{\rho}(r)\r).
\end{equation}
To this end, we first observe that,
if $z\in \overline{B}_\rho(x,\varphi^x_{\rho}(r))$, then, by the choice of $\tilde{x}$ and
the facts that $r>C_\rho\varphi^x_{\rho}(r)/\lambda^2$ and $\lambda/C_\rho<1$,  one has
\begin{equation*}
\begin{split}
\rho(\tilde{x},z)&\leq C_\rho\max\lf\{\rho(\tilde{x},x),\rho(x,z)\r\}
\le C_\rho\max\lf\{\widetilde{C}_\rho \rho(x,\tilde{x}) ,\varphi^x_{\rho}(r)\r\}
<C_\rho\max\lf\{\widetilde{C}_\rho\lambda r,\varphi^x_{\rho}(r)\r\}\\
&<C_\rho\widetilde{C}_\rho\max\lf\{\lambda r,\lambda^2r/C_\rho\r\}
=C_\rho\widetilde{C}_\rho\lambda r=\tilde{r},
\end{split}
\end{equation*}
which implies the first inclusion $\overline{B}_\rho(x,\varphi^x_{\rho}(r))
\subset B_\rho(\tilde{x},\tilde{r})$ in \eqref{dx.1}.
The second inclusion $B_\rho(\tilde{x},\tilde{r})\subset B_\rho(x,r)$ in \eqref{dx.1}
is an immediate consequence of the observation that, for any $z\in B_\rho(\tilde{x},\tilde{r})$,
\begin{equation*}
\begin{split}
\rho(x,z)&\leq C_\rho\max\big\{\rho(x,\tilde{x}),\rho(\tilde{x},z)\big\}
< C_\rho\max\big\{\lambda r,\tilde{r}\big\}\\
&=C_\rho\tilde{r}=C_\rho^2\widetilde{C}_\rho\lambda r<r.
\end{split}
\end{equation*}
To finish the proof of \eqref{dx.1},
we need to prove
$B_\rho(\tilde{x},(C_\rho\widetilde{C}_\rho)^{-2}\lambda \tilde{r})
\subset B_\rho(\tilde{x},\tilde{r})\setminus\overline{B}_\rho(x,\varphi^x_{\rho}(r))$.
Since $(C_\rho\widetilde{C}_\rho)^{-2}\lambda<1$, we have
$B_\rho(\tilde{x},(C_\rho\widetilde{C}_\rho)^{-2}\lambda \tilde{r})\subset B_\rho(\tilde{x},\tilde{r})$.
Fix $z\in B_\rho(\tilde{x},(C_\rho\widetilde{C}_\rho)^{-2}\lambda \tilde{r})$.
Observe that, by the choices of $\tilde{x}$  and $\tilde{r}$, one has
\begin{equation*}
\begin{split}
\lambda^2 r\leq \rho(x,\tilde{x})
&\leq C_\rho\max\lf\{\rho(x,z),\rho(z,\tilde{x})\r\}\leq C_\rho\max\lf\{\rho(x,z), \widetilde{C}_\rho\rho(\tilde{x},z)\r\}\\
&< C_\rho\max\lf\{\rho(x,z),\widetilde{C}_\rho(C_\rho\widetilde{C}_\rho)^{-2}\lambda \tilde{r}\r\}= \max\lf\{C_\rho\rho(x,z),\lambda^2 r\r\},
\end{split}
\end{equation*}
which implies that $\lambda^2 r<C_\rho\rho(x,z).$
Hence, $\rho(x,z)>\lambda^2 r/C_\rho> \varphi^x_{\rho}(r)$, which   implies the desired inclusion
$B_\rho(\tilde{x},(C_\rho\widetilde{C}_\rho)^{-2}\lambda \tilde{r})\subset B_\rho(\tilde{x},\tilde{r})\setminus \overline{B}_\rho(x,\varphi^x_{\rho}(r)).
$
This finishes the proof of \eqref{dx.1}.

To complete the proof of the present lemma, it suffices to show $\tilde{r}\leq  C_\rho\varphi^{\tilde{x}}_{\rho}(\tilde{r})/\lambda^2$. 	
It follows from \eqref{dx.1} that $\overline{B}_\rho(x,\varphi^x_{\rho}(r))$ and $B_\rho(\tilde{x},(C_\rho\widetilde{C}_\rho)^{-2}\lambda \tilde{r})$ are disjoint subsets of $B_\rho(\tilde{x},\tilde{r})$ and
\begin{align}
\label{rn-2}
\mu\lf(B_\rho\lf(\tilde{x},(C_\rho\widetilde{C}_\rho)^{-2}\lambda \tilde{r}\r)\r)
&=\frac{1}{2}\lf[\mu\lf(B_\rho\lf(\tilde{x},(C_\rho\widetilde{C}_\rho)^{-2}\lambda \tilde{r}\r)\r)+\mu\lf(B_\rho(\tilde{x},\lf(C_\rho\widetilde{C}_\rho)^{-2}\lambda \tilde{r}\r)\r)\r]\nonumber\\
&\leq\frac{1}{2}\lf[\mu\lf(B_\rho(x,r)\setminus\overline{B}_\rho\lf(x,\varphi^x_{\rho}(r)\r)\r)
+\mu\lf(B_\rho\lf(\tilde{x},(C_\rho\widetilde{C}_\rho)^{-2}\lambda \tilde{r}\r)\r)\r]\nonumber\\
&\leq\frac{1}{2}\lf[\mu\lf(\overline{B}_\rho\lf(x,\varphi^x_{\rho}(r)\r)\r)
+\mu\lf(B_\rho\lf(\tilde{x},(C_\rho\widetilde{C}_\rho)^{-2}\lambda \tilde{r}\r)\r)\r]\nonumber\\
&
\leq\frac{1}{2}\,\mu\lf(B_\rho(\tilde{x},\tilde{r})\r),
\end{align}
where, in obtaining the second inequality of \eqref{rn-2}, we have used the	second inequality of
\eqref{LLk-3}.  From  \eqref{rn-2} and \eqref{radphi}, it follows that $\varphi^{\tilde{x}}_{\rho}(\tilde{r})\geq (C_\rho\widetilde{C}_\rho)^{-2}\lambda \tilde{r}$, and hence
$$\tilde{r}\leq (C_\rho\widetilde{C}_\rho)^{2}\varphi^{\tilde{x}}_{\rho}(\tilde{r})/\lambda\leq C_\rho\varphi^{\tilde{x}}_{\rho}(\tilde{r})/\lambda^2,$$
due to  $(C_\rho\widetilde{C}_\rho)^{2}<1/\lambda\leq C_\rho/\lambda$. This finishes the proof   of Lemma \ref{tx}.
\end{proof}

The following lemma contains, among other things, a generalization of \cite[Lemma~12]{agh20} to the setting of quasi-metric measure spaces. Using Lemma~\ref{tx} instead of \cite[Lemma 11]{agh20}, its proof is a
straightforward modification of \cite[Lemma~12]{agh20},
as the argument therein is independent of the fact that $\rho$ is a genuine metric. We omit the details.

\begin{lemma}
\label{en2-4}
Let $(X,\rho,\mu)$ be a uniformly perfect quasi-metric measure space and  $Q\in (0,\infty)$.
Let $\lambda\in(0,(C_\rho\widetilde{C}_\rho)^{-2})$ be as in \eqref{U-perf}. Then the following statements are valid.
\begin{enumerate}
\item[${\rm(i)}$] Assume that there exists a positive constant $C$ such that $\mu(B_\rho(x,r))\geq Cr^Q$
whenever $x\in X$ and $r\in(0,{\rm diam}_\rho(X)]$ satisfy $r\leq C_\rho\varphi^x_{\rho}(r)/\lambda^2$. Then,
for any $x\in X$ and any finite $r\in(0,{\rm diam}_\rho(X)]$,
$\mu(B_\rho(x,r))\geq \widetilde{C}r^Q$, where $\widetilde{C}:=C\lambda^{Q}$.
		
\item[${\rm(ii)}$]  Assume that there exists a positive constant $C$ such that,
for any $x,\,y\in X$ and $0<r\leq R<\infty$ satisfying $B_\rho(x,r)\subset B_\rho(y,R)$
and  $r\leq C_\rho\varphi^x_{\rho}(r)/\lambda^2$,  it holds true that
\begin{equation}
\label{eq48}
\frac{\mu(B_\rho(x,r))}{\mu(B_\rho(y,R))}\geq C\left(\frac{r}{R}\right)^Q.
\end{equation}
Then, for any $x,\,y\in X$ and $0<r\leq R<\infty$ satisfying $B_\rho(x,r)\subset B_\rho(y,R)$,
\begin{equation}
\label{eq49}
\frac{\mu(B_\rho(x,r))}{\mu(B_\rho(y,R))}\geq \widetilde{C}\left(\frac{r}{R}\right)^Q,
\end{equation}
where $\widetilde{C}:=C\lambda^{Q}$.
\end{enumerate}
\end{lemma}

Before proceeding, we take a moment to recall the smoothness class
of H\"older functions  in the context of quasi-metric spaces.
Let $(X,\rho)$ be a quasi-metric space and fix $\alpha\in(0,\infty)$.
The {\it homogeneous} {\it H\"older} {\it space} (of order $\alpha$) $\dot{\mathscr{C}}^\alpha(X,\rho)$ is defined
to be the collection of all functions $f\colon X\to\mathbb{R}$ such that the following semi-norm
\begin{eqnarray}\label{Hol.T2}
\|f\|_{\dot{\mathscr{C}}^\alpha(X,\rho)}:=
\sup_{x,\,y\in X,\ x\not=y}\frac{|f(x)-f(y)|}{\rho(x,y)^\alpha}<\fz.
\end{eqnarray}
Define an equivalence relation, $\sim$, on $\dot{\mathscr{C}}^\alpha(X,\rho)$,
by $f\sim g$ if and only if $f-g$ is a constant function on $X$.
Then $\dot{\mathscr{C}}^\alpha(X,\rho)/\sim$
is a Banach space   equipped with
the norm $\|\cdot\|_{\dot{\mathscr{C}}^\alpha(X,\rho)}$.
It is also  easy to prove that,   when $\rho$ and $\varrho$ are two quasi-metrics on $X$ satisfying $\rho\approx\varrho$,
then $\dot{\mathscr{C}}^\alpha(X,\rho)=\dot{\mathscr{C}}^\alpha(X,\varrho)$ with equivalent semi-norms.
As a notational convention,   we  write
$${\rm Lip}(X,\rho):=\dot{\mathscr{C}}^1(X,\rho).$$

In order to prove that the lower bound condition of measures follows from
the Sobolev-type embedding properties of the spaces $\dot{M}^s_{p,q}$ and $\dot{N}^s_{p,q}$,
we need to construct a suitable family of bump functions that belong to $\dot{M}^s_{p,q}$ and $\dot{N}^s_{p,q}$,
and approximate arbitrarily well the characteristic function of a given ball.
Recall that \cite[Lemma~13]{agh20}   constructed such bump
functions that are Lipschitz and belong to the  spaces $M^{1,p}$. However,
in a general quasi-metric space, there is no guarantee that nonconstant Lipschitz
functions exist. From this perspective, we need to work with the more
general scale of H\"older continuous and, as it turns out, the
maximal amount of smoothness  (measured here on the H\"older scale)
that such a bump function can possess is intimately linked to
the geometry of the underlying space. More specifically, we have the following lemma.

\begin{lemma}
\label{GVa2}
Suppose $(X,\rho,\mu)$ is a quasi-metric measure space
and $\varrho$ is any quasi-metric on $X$ such that $\varrho\approx\rho$.
Let $C_\varrho\in[1,\infty)$ be as in \eqref{C-RHO.111} and
fix a finite number $\alpha\in(0,(\log_{2}C_\varrho)^{-1}]$, along with the
parameters $p\in(0,\infty)$, $q\in(0,\infty]$, and  $s\in(0,\alpha]$,
where the value  $s=\alpha$ is only permissible when $q=\infty$.
Finally, let $\varrho_{\#}$ be the regularized quasi-metric given by Theorem~\ref{DST1}, and suppose $x\in X$ and  $0\leq r<R<\infty$.
Then there exists a function $\Phi_{r,R}\in\dot{\mathscr{C}}^\alpha(X,\varrho_{\#})$ such that
\begin{equation}\label{PMab5}
0\leq\Phi_{r,R}\leq 1\,\,\mbox{ on }\,\,X,\quad
\Phi_{r,R}\equiv 0\,\,\mbox{ on }\,\,X\setminus B_{\varrho_{\#}}(x,R),\quad
\Phi_{r,R}\equiv 1\,\,\mbox{ on }\,\,\overline{B}_{\varrho_{\#}}(x,r).
\end{equation}
Moreover, $\Phi_{r,R}$ belongs to both $M^s_{p,q}(X,\rho,\mu)$ and $N^s_{p,q}(X,\rho,\mu)$,
and there holds true %
\begin{equation}
\label{gradest}
0<\Vert\Phi_{r,R}\Vert_{\dot{M}^s_{p,q}(X)}\leq C\lf(R^\alpha-r^\alpha\r)^{-s/\alpha}\lf[\mu(B_{\varrho_{\#}}(x,R))\r]^{1/p},
\end{equation}
and
\begin{equation}
\label{gradest2}
0<\Vert\Phi_{r,R}\Vert_{\dot{N}^s_{p,q}(X)}\leq C\lf(R^\alpha-r^\alpha\r)^{-s/\alpha}\lf[\mu(B_{\varrho_{\#}}(x,R))\r]^{1/p},
\end{equation}
where  $C$ is a positive constant depending only on $s$, $q$, $\alpha$, $\varrho$, and $\rho$.
\end{lemma}

\begin{remark}
Lemma~\ref{GVa2} improves \cite[Lemma~13]{agh20} even in the metric setting. Indeed, for the metric space $(\mathbb{R},|\cdot-\cdot|^{1/2})$, Lemma~\ref{GVa2} guarantees the existence of a family of H\"older continuous bump functions of any order $\alpha\in(0,2)$, where as \cite[Lemma~13]{agh20} can only yield Lipschitz functions.
\end{remark}

\begin{remark}
\label{trivpoin}
It follows from Lemma~\ref{GVa2} that $\dot{M}^s_{p,q}(X)$ and
$\dot{N}^s_{p,q}(X)$ contain plenty of nonconstant
functions whenever $p\in(0,\infty)$, $q\in(0,\infty]$, and
$s\in(0,\infty)$ satisfies $s\preceq_q{\rm ind}\,(X,\rho)$. In particular,  we have that ${\rm ind}\,(X,\rho)=1$ if $(X,\rho,\mu)$ is a metric measure space that supports a weak $(1, p)$-Poincar\'e inequality with $p>1$ and $\mu$ is a doubling measure. Indeed, if ${\rm ind}\,(X,\rho)>1$ then  Lemma~\ref{GVa2} implies that $\dot{M}^1_{p,q}(X)$ and
$\dot{N}^1_{p,q}(X)$ contain nonconstant functions for any $q\in(0,\infty]$ which would contradict \cite[Theorem~4.1]{GKZ13}.
\end{remark}

\begin{proof}[Proof of Lemma \ref{GVa2}]
By Theorem~\ref{DST1},
we find that the function $(\varrho_{\#})^\alpha\colon X\times X\to[0,\infty)$ is a genuine metric on $X$.
From this and a straightforward computation, it follows that the function $\Phi_{r,R}\colon X\to\mathbb{R}$, defined by
setting, for any $y\in X$,
\begin{eqnarray*}
\Phi_{r,R}(y):=
\left\{
\begin{array}{ll}
 1\qquad\,\,&\mbox{if\,\, $y\in \overline{B}_{\varrho_{\#}}(x,r)$,}\\
\displaystyle\frac{R^\alpha-[\varrho_{\#}(x,y)]^\alpha}{R^\alpha-r^\alpha} &\mbox{if\,\, $y\in B_{\varrho_{\#}}(x,R)\setminus \overline{B}_{\rho_{\#}}(x,r)$,}\\
0 &\mbox{if\,\,  $y\in X\setminus B_{\varrho_{\#}}(x,R)$,}
\end{array}
\right.
\end{eqnarray*}
is a Lipschitz function in ${\rm Lip}(X,(\varrho_{\#})^\alpha)$ (hence, also, $\mu$-measurable
and in $\dot{\mathscr{C}}^\alpha(X,\varrho_{\#})$) that satisfies the properties listed in \eqref{PMab5} as well as the estimate
\begin{equation}
\label{PMab6}
0<\|\Phi_{r,R}\|_{\dot{\mathscr{C}}^\alpha(X,\varrho_{\#})}=
\|\Phi_{r,R}\|_{{\rm Lip}(X,(\varrho_{\#})^\alpha)}\leq\lf(R^\alpha-r^\alpha\r)^{-1}.
\end{equation}

Fix $s\in(0,\alpha]$, $p\in(0,\infty)$, and $q\in(0,\infty]$. We now turn to construct
a fractional  $s$-gradient of $\Phi_{r,R}$. To this end,
let $k_0\in\mathbb{Z}$ be the unique integer such that
\begin{equation}
\label{VN-2}
2^{k_0-1}\leq\|\Phi_{r,R}\|^{1/\alpha}_{\dot{\mathscr{C}}^\alpha(X,\varrho_{\#})}<2^{k_0}.
\end{equation}
Since  $\varrho_{\#}\approx\varrho$ (see Theorem~\ref{DST1})
and $\varrho\approx\rho$, it follows that there exists a constant $c\in(1,\infty)$ such that
\begin{eqnarray}
\label{VN-2-X}
c^{-1}\rho(z,y)\leq\varrho_{\#}(z,y)\leq c\,\rho(z,y),\quad\forall\,z,\,y\in X.
\end{eqnarray}
Now, for any  $k\in\mathbb{Z}$ and $y\in X$, define
\begin{eqnarray*}
g_k(y):=\left\{
\begin{array}{ll}
2^{-k(\alpha-s)}c^\alpha\,\|\Phi_{r,R}\|_{\dot{\mathscr{C}}^\alpha(X,\varrho_{\#})}\,{\bf 1}_{B_{\varrho_{\#}}(x,R)}(y)\quad&\mbox{if}\,\,\,k\geq k_0,
\\[6pt]
2^{(k+1)s+1}\,{\bf 1}_{B_{\varrho_{\#}}(x,R)}(y)&\mbox{if}\,\,\,k<k_0.
\end{array}
\right.
\end{eqnarray*}
Clearly, for any $k\in\zz$,   $g_k$ is $\mu$-measurable,
because $\mu$ is a Borel measure and each $\varrho_{\#}$-ball is an open set (cf.~Theorem~\ref{DST1}).

Now we prove that $\vec g:=\{g_k\}_{k\in\zz}$ is a
fractional $s$-gradient of $\Phi_{r,R}$ (with respect to $\rho$),
namely,  $\vec{g}\in\mathbb{D}^s_\rho(\Phi_{r,R})$. To this
end, fix $k\in\mathbb{Z}$ and take $y,\,z\in X$ satisfying
that $2^{-k-1}\leq \rho(y,z)< 2^{-k}$.
We first consider the case when $k\geq k_0$.
If $y,\,z\in X\setminus B_{\varrho_{\#}}(x,R)$, then
$\Phi_{r,R}(y)=\Phi_{r,R}(z)=0$ and there is nothing to show.
If, on the other hand, $y\in X$ and $z\in B_{\varrho_{\#}}(x,R)$,
then, using the H\"older continuity of $\Phi_{r,R}$, the second
inequality in \eqref{VN-2-X}, and the fact that $s\leq\alpha$, we conclude that
\begin{equation*}
\begin{split}
|\Phi_{r,R}(y)-\Phi_{r,R}(z)|
&\leq [\varrho_{\#}(y,z)]^{\alpha}\|\Phi_{r,R}\|_{\dot{\mathscr{C}}(X,\varrho_{\#})}	
\leq [\rho(y,z)]^{\alpha}c^\alpha \,\|\Phi_{r,R}\|_{\dot{\mathscr{C}}^\alpha(X,\varrho_{\#})}
\\
&= [\rho(y,z)]^{s}[\rho(y,z)]^{\alpha-s}c^\alpha\,\|\Phi_{r,R}\|_{\dot{\mathscr{C}}(X,\varrho_{\#})}\,{\bf 1}_{B_{\varrho_{\#}}(x,R)}(z)
\\
&\leq [\rho(y,z)]^{s}2^{-k(\alpha-s)}c^\alpha\,\|\Phi_{r,R}\|_{\dot{\mathscr{C}}(X,\varrho_{\#})}\,{\bf 1}_{B_{\varrho_{\#}}(x,R)}(z)
\\
&\leq [\rho(y,z)]^{s}\lf[g_k(y)+g_k(z)\r].
\end{split}
\end{equation*}	
Since the quasi-metric $\varrho_{\#}$ is symmetric, the above inequalities remain true
if we exchange the positions of $y$ and $z$.
This gives the desired inequality in the case when $k\geq k_0$.
Assume next that $k<k_0$. As in the case $k\geq k_0$, it suffices to
consider the scenario when $y\in X$ and $z\in B_{\varrho_{\#}}(x,R)$. Observe that,
by $\|\Phi_{r,R}\|_{L^\fz(X)}\leq1$, we have
\begin{equation*}
\begin{split}
|\Phi_{r,R}(y)-\Phi_{r,R}(z)|&\leq 2\, [\rho(y,z)]^{s}[\rho(y,z)]^{-s}\,{\bf 1}_{B_{\varrho_{\#}}(x,R)}(z)
\\
&\leq [\rho(y,z)^{s}]2^{s(k+1)+1}\,{\bf 1}_{B_{\varrho_{\#}}(x,R)}(z)
\\
&\leq [\rho(y,z)]^{s}\lf[g_k(y)+g_k(z)\r].
\end{split}
\end{equation*}	
This finishes the proof of the fact that $\vec{g}:=\{g_k\}_{k\in\mathbb{Z}}\in\mathbb{D}^s_\rho(\Phi_{r,R})$.
	
Next we use $\vec{g}:=\{g_k\}_{k\in\mathbb{Z}}$ to show
\eqref{gradest}. If $q<\infty$, then, by the definition of $g_k$,  \eqref{VN-2}, and \eqref{PMab6},
 we have (keeping in mind  the fact that $s<\alpha$ in this case)
\begin{align*}
\lf(\sum_{k\in\mathbb{Z}}|g_k|^q\r)^{1/q}
&\lesssim\lf(\sum_{k=-\infty}^{k_0-1}2^{(ks+s+1)q}\r)^{1/q}+c^\alpha
\,\|\Phi_{r,R}\|_{\dot{\mathscr{C}}^\alpha(X,\varrho_{\#})}
\lf(\sum_{k=k_0}^\infty2^{-kq(\alpha-s)}\r)^{1/q}
\nonumber\\
&\lesssim 2^{k_0s}+\|\Phi_{r,R}\|_{\dot{\mathscr{C}}^\alpha(X,\varrho_{\#})}\,2^{-k_0(\alpha-s)}
\lesssim \|\Phi_{r,R}\|_{\dot{\mathscr{C}}^\alpha(X,\varrho_{\#})}^{s/\alpha}\lesssim \big(R^\alpha-r^\alpha\big)^{-s/\alpha}.
\end{align*}
From this, $\vec{g}\in\mathbb{D}^s_\rho(\Phi_{r,R})$,  and
the fact that each $g_k$ is supported in $B_{\varrho_{\#}}(x,R)$, it follows that
$$
\Vert\Phi_{r,R}\Vert_{\dot{M}^s_{p,q}(X,\rho,\mu)}\leq\Vert\vec{g}\Vert_{L^p(X,\ell^q)}\lesssim \lf(R^\alpha-r^\alpha\r)^{-s/\alpha}\lf[\mu(B_{\varrho_{\#}}(x,R))\r]^{1/p}.
$$
Thus, the second estimate in \eqref{gradest} holds true. The case when $q=\infty$ is handled similarly.
	
Furthermore, observe that, if $q<\infty$, then
\begin{align*}
&\Vert\Phi_{r,R}\Vert_{\dot{N}^s_{p,q}(X,\rho,\mu)}\\
&\quad\leq\Vert\vec{g}\Vert_{\ell^q(L^p(X))}
\lesssim\lf(\sum_{k=-\infty}^{k_0-1}\Vert g_k\Vert_{L^p(X,\mu)}^{q}\r)^{1/q}
+\lf(\sum_{k=k_0}^\infty\Vert g_k\Vert_{L^p(X,\mu)}^{q}\r)^{1/q}
\\
&\quad\lesssim \lf[\mu\lf(B_{\varrho_{\#}}(x,R)\r)\r]^{1/p}\left[\lf(\sum_{k=-\infty}^{k_0-1}
2^{(ks+s+1)q}\r)^{1/q}+c^\alpha\,\|\Phi_{r,R}\|_{\dot{\mathscr{C}}^\alpha(X,\varrho_{\#})}
\lf\{\sum_{k=k_0}^\infty2^{-kq(\alpha-s)}\r\}^{1/q}\right]
\\
&\quad\lesssim\lf[\mu\lf(B_{\varrho_{\#}}(x,R)\r)\r]^{1/p}\left[ 2^{k_0s}+\|\Phi_{r,R}\|_{\dot{\mathscr{C}}^\alpha(X,\varrho_{\#})}\,2^{-k_0(\alpha-s)}\right]\\
&\quad\lesssim \|\Phi_{r,R}\|_{\dot{\mathscr{C}}^\alpha(X,\varrho_{\#})}^{s/\alpha}
\lf[\mu\lf(B_{\varrho_{\#}}(x,R)\r)\r]^{1/p}
\lesssim \lf(R^\alpha-r^\alpha\r)^{-s/\alpha}\lf[\mu\lf(B_{\varrho_{\#}}(x,R)\r)\r]^{1/p}.
\end{align*}
This proves  the second estimate in \eqref{gradest2} when $q<\fz$.
The estimates in the case $q=\infty$ follow along a similar line of reasoning.

Finally, notice that, by \eqref{PMab5},  $\Phi_{r,R}$ is
nonconstant in $X$. Therefore, according to Proposition~\ref{constant}, we find that $\Vert\Phi_{r,R}\Vert_{\dot{N}^s_{p,q}(X,\rho,\mu)}>0$ and $\Vert\Phi_{r,R}\Vert_{\dot{M}^s_{p,q}(X,\rho,\mu)}>0$.
This proves the first inequalities of \eqref{gradest} and
\eqref{gradest2}, and hence finishes the proof of
Lemma \ref{GVa2}.
\end{proof}

The next result constitutes the main result of this section, where we use Lemma~\ref{GVa2} to construct a family of maximally smooth H\"older continuous functions that approximate well the characteristic functions of balls.

\begin{lemma}\label{HolderBump}
Suppose $(X,\rho,\mu)$ is a quasi-metric measure space and $\varrho$ is any quasi-metric on $X$ such that $\varrho\approx\rho$. Let $C_\varrho\in[1,\infty)$ be as in \eqref{C-RHO.111} and fix
exponents $p\in(0,\infty)$ and $q\in(0,\infty]$, along with
a finite number $s\in(0,(\log_{2}C_\varrho)^{-1}]$,
where the value  $s=(\log_{2}C_\varrho)^{-1}$
is only permissible when $q=\infty$. Also suppose
that $\varrho_{\#}$ is the regularized quasi-metric
given by Theorem~\ref{DST1}. Then there exists a
number $\delta\in(0,1)$, depending only on $s$, $\varrho$, and the proportionality
constants in $\varrho\approx\rho$, such that, for any $x\in X$ and any finite
$r\in(0,{\rm diam}_\rho(X)]$,
there exist a sequence $\{r_j\}_{j\in\mathbb{N}}$ of radii
and a collection $\{u_j\}_{j\in\mathbb{N}}$ of functions  such that, for any $j\in\mathbb{N}$,
\begin{enumerate}[label=\rm{(\alph*)}]
\item $B_{\varrho_{\#}}(x,r_j)\subset B_{\rho}(x,r)$ \,\,\mbox{and }\,\,$\delta r<r_{j+1}<r_j<r$;
\item $0\leq u_j\leq 1$ pointwise on $X$;
\item $u_j\equiv 0$ on $X\setminus B_{\varrho_{\#}}(x,r_j)$;
\item $u_j\equiv 1$ on $\overline{B}_{\varrho_{\#}}(x,r_{j+1})$;
\item $u_j$ belongs to $M^s_{p,q}(X,\rho,\mu)$ and $N^s_{p,q}(X,\rho,\mu)$, and
$$0<\Vert u_j\Vert_{\dot{M}^s_{p,q}(X,\rho,\mu)}\leq C 2^{j}r^{-s}\lf[\mu(B_{\varrho_{\#}}(x,r_j))\r]^{1/p}$$
and
$$0<\Vert u_j\Vert_{\dot{N}^s_{p,q}(X,\rho,\mu)}\leq C 2^{j}r^{-s}\lf[\mu(B_{\varrho_{\#}}(x,r_j))\r]^{1/p}
$$
for some positive constant $C$ which is independent of $x$, $r$, and $j$.
\end{enumerate}
	
Furthermore, if there exists a constant $c_0\in(1,\infty)$
such that $r\leq c_0\varphi^x_{\rho}(r)$, where
$\varphi^x_{\rho}(r)$ is as in \eqref{radphi},
then the radii $\{r_j\}_{j\in\mathbb{N}}$ and the
functions $\{u_j\}_{j\in\mathbb{N}}$ can be chosen to
have the following property in addition
to  {\rm(a)}-{\rm(e)}  above:
\begin{enumerate}
\item[{\rm(f)}] for any fixed $j\in\mathbb{N}$ and $\gamma\in(0,\infty)$, there exists a $\mu$-measurable set $E^\gamma_j\subset B_\rho(x,r)$ such that $\mu(E^\gamma_j)\geq\mu(B_{\varrho_{\#}}(x,r_{j+1}))$ and $|u_j-\gamma\,|\geq\frac{1}{2}$ pointwise on $E^\gamma_j$.
\end{enumerate}
In this case, the number $\delta$   depends also on $c_0$.
\end{lemma}


\begin{remark}
We say that the sequence $\{r_j\}_{j\in\mathbb{N}}$ of radii  and the collection $\{u_j\}_{j\in\mathbb{N}}$ of functions as in Lemma~\ref{HolderBump}
are associated to the ball $B_\rho(x,r)$.
\end{remark}

\begin{proof}[Proof of Lemma~\ref{HolderBump}]
Fix a ball $B_\rho(x,r)$, where $x\in X$ and $r\in\big(0,{\rm diam}_\rho(X)\big]$ is finite. Since  $\varrho_{\#}\approx\varrho$ (see Theorem~\ref{DST1}) and $\varrho\approx\rho$,
 it follows that there exists a constant $c\in(1,\infty)$ such that
\begin{equation}
\label{fny-586}
c^{-1}\rho(z,y)\leq\varrho_{\#}(z,y)\leq c\,\rho(z,y),\quad\forall\,z,\,y\in X
\end{equation}
and hence,
\begin{equation}\label{ye.1}
B_{\varrho_{\#}}\big(x,r/c\big)\subset B_{\rho}(x,r).
\end{equation}
%
%
%
Fix a finite number $\alpha\in[s,(\log_{2}C_\varrho)^{-1}]$, where $\alpha\neq s$
unless $s=(\log_{2}C_\varrho)^{-1}$. We inductively define a
sequence $\{r_j\}_{j\in\mathbb{N}}$ of radii  by setting
\begin{equation}
\label{gg.e}
r_1:=r/c\quad\mbox{ and	}\quad r_{j+1}:=\lf[(r_j)^\alpha-2^{-(j+1)}(r_1)^\alpha\r]^{1/\alpha}
\quad\mbox{for any $j\in\mathbb{N}:=\{1,2,\ldots\}$}.
\end{equation}
Observe that the second condition
in \eqref{gg.e} is equivalent to $(r_j)^\alpha-(r_{j+1})^\alpha=2^{-(j+1)}(r_1)^\alpha$.
By \eqref{ye.1} and the definition of $\{r_j\}_{j\in\mathbb{N}}$,
we obtain $r_{j+1}<r_j<r$ and
$$B_{\varrho_{\#}}\big(x,r_j)\subset B_{\varrho_{\#}}\big(x,r_1)\subset B_{\rho}(x,r).$$
Moreover, for any $j\in\mathbb{N}$, we have
\begin{align}\label{lbrj}
r_{j+1}&=\lf[(r_1)^\alpha-\sum_{k=1}^{j}2^{-(k+1)}(r_1)^\alpha\r]^{1/\alpha}
>\lf[(r_1)^\alpha-\sum_{k=1}^{\infty}2^{-(k+1)}(r_1)^\alpha\r]^{1/\alpha}\nonumber\\
&=r_12^{-1/\alpha}=r/\lf(2^{1/\alpha}c\r).
\end{align}
Thus, the family $\{r_j\}_{j\in\mathbb{N}}$ of radii  satisfy (a) with  $\delta:=\big(2^{1/\alpha}c\big)^{-1}\in(0,1)$.	

Moving forward, for any $j\in\mathbb{N}$,
define $u_j:\,X\to[0,1]$ by setting
$u_j:=\Phi_{r_{j+1},r_j}\in\dot{\mathscr{C}}^\alpha(X,\varrho_{\#})$,
where $\Phi_{r_{j+1},r_j}$ is given by Lemma~\ref{GVa2}.
By the conclusion of Lemma~\ref{GVa2}, it is easy to show
that $\{u_j\}_{j\in\mathbb{N}}$ satisfies (b), (c), and  (d) of the present lemma
as well as the estimates
$$
0<\Vert u_j\Vert_{\dot{M}^s_{p,q}(X,\rho,\mu)}\leq C\lf[(r_j)^\alpha-(r_{j+1})^\alpha\r]^{-s/\alpha}
\lf[\mu(B_{\varrho_{\#}}(x,r_j))\r]^{1/p},\quad\forall\,j\in\mathbb{N},
$$
and
$$
0<\Vert u_j\Vert_{\dot{N}^s_{p,q}(X,\rho,\mu)}\leq C\lf[(r_j)^\alpha-(r_{j+1})^\alpha\r]^{-s/\alpha}\lf[\mu(B_{\varrho_{\#}}(x,r_j))\r]^{1/p},
\quad\forall\ j\in\mathbb{N},
$$
where $C$ is a positive
constant depending only on $s$, $q$, $\alpha$, $\varrho$, and $\rho$.
From these estimates, the observation that
$$\lf[(r_j)^\alpha-(r_{j+1})^\alpha\r]^{-s/\alpha}=c^{s} 2^{s(j+1)/\alpha}r^{-s},\quad\forall\ j\in\mathbb{N},$$
and the fact that $s/\alpha\leq 1$, we further deduce
the desired estimates in   (e).

Assume now that there exists a  constant
$c_0\in(1,\infty)$ such that  $r\leq c_0\varphi^x_{\rho}(r)$,  where $\varphi^x_{\rho}(r)$
is as in \eqref{radphi}. By this assumption and   Lemma~\ref{Gds.24}(ii), we have $0<\varphi^x_{\rho}(r)<r$. Now, take
\begin{equation*}
r_1:=\varphi^x_{\rho}(r)/c\quad\mbox{ and }\quad r_{j+1}:=\lf[(r_j)^\alpha-2^{-(j+1)}(r_1)^\alpha\r]^{1/\alpha}
\quad\mbox{for any $j\in\mathbb{N}$},
\end{equation*}
where $c\in(1,\infty)$ is as in \eqref{fny-586}. Then, similarly to
\eqref{lbrj} and \eqref{ye.1}, we have, for any $j\in\nn$,
$$r/\big(c\cdot c_02^{1/\alpha}\big)\leq\varphi^x_{\rho}(r)/\big(c2^{1/\alpha}\big)<r_{j+1}<r_j<\varphi^x_{\rho}(r)<r$$
and
\begin{equation}
\label{xzw-1}
B_{\varrho_{\#}}(x,r_j)\subset B_{\varrho_{\#}}(x,r_1)\subset B_\rho(x,\varphi^x_{\rho}(r))\subset B_\rho(x,r).
\end{equation}
Moreover, by $s/\alpha\leq1$ and $r\leq c_0\varphi^x_{\rho}(r)$, we also find that
\begin{equation}
\label{xzw-1-tn}
\lf[(r_j)^\alpha-(r_{j+1})^\alpha\r]^{-s/\alpha}=c^{s} 2^{s(j+1)/\alpha}\lf[\varphi^x_{\rho}(r)\r]^{-s}
\leq 2(c\cdot c_0)^{s}\cdot 2^{j}r^{-s},\quad \forall\ j\in\nn.
\end{equation}
Thus, by defining $u_j:\,X\to[0,1]$ as $u_j:=\Phi_{r_{j+1},r_j}$,
where $j\in\nn$ and  $\Phi_{r_{j+1},r_j}$ is given by Lemma~\ref{GVa2},
it follows from the above reasoning that $\{r_j\}_{j\in\mathbb{N}}$ and
$\{u_j\}_{j\in\mathbb{N}}$ satisfy  ({a})-({d})
of the present lemma with $\delta:=(c\cdot c_02^{1/\alpha})^{-1}$.
Moreover, by \eqref{xzw-1-tn}, we find that Lemma~\ref{GVa2} guarantees that
$u_j$ belongs to  $M^s_{p,q}(X,\rho,\mu)\bigcap   N^s_{p,q}(X,\rho,\mu)$
and satisfies the estimates in (e).
	
To complete the proof of the present lemma, we only need to prove ({\rm f}) now.
To this end, fix $\gamma\in(0,\infty)$. For any $j\in\nn$,
since $u_j\equiv1$ on $B_{\varrho_{\#}}(x,r_{j+1})$ and $u_j\equiv0$
on $B_\rho(x,r)\setminus B_{\varrho_{\#}}(x,r_{j})$, we deduce that
$|u_j-\gamma\,|\geq\frac{1}{2}$ on at least one of the
sets $B_{\varrho_{\#}}(x,r_{j+1})$ and $B_\rho(x,r)\setminus B_{\varrho_{\#}}(x,r_{j})$.
Let $E^\gamma_j\subset B_\rho(x,r)$ be one of these two sets in which
$|u_j-\gamma\,|\geq\frac{1}{2}$.
Note that $E^\gamma_j$ is
$\mu$-measurable because $\varrho_\#$-balls are open sets and $\mu$
is a Borel measure. We claim that $\mu(E^\gamma_j)\geq
\mu(B_{\varrho_{\#}}(x,r_{j+1}))$. To see this, we use
 \eqref{xzw-1}  and \eqref{LLk-3} to conclude that
$$
\mu\lf(B_{\varrho_{\#}}(x,r_{j+1})\r)\leq\mu\lf(B_\rho(x,\varphi^x_{\rho}(r))\r)\leq\frac{1}{2}\mu(B_\rho(x,r)),
$$
and
\begin{align*}
\mu\lf(B_\rho(x,r)\setminus B_{\varrho_{\#}}(x,r_{j})\r)
&=\mu(B_\rho(x,r))-\mu\lf(B_{\varrho_{\#}}(x,r_{j})\r)\\
&\geq\mu(B_\rho(x,r))-\mu\lf(B_\rho(x,\varphi^x_{\rho}(r))\r)\geq\frac{1}{2}\mu(B_\rho(x,r)).
\end{align*}
Therefore, $\mu(B_\rho(x,r)\setminus B_{\varrho_{\#}}(x,r_{j}))\geq\mu(B_{\varrho_{\#}}(x,r_{j+1}))$ and hence
$$
\mu(E^\gamma_j)\geq\min\lf\{\mu\lf(B_{\varrho_{\#}}(x,r_{j+1})\r),\
\mu\lf(B_\rho(x,r)\setminus B_{\varrho_{\#}}(x,r_{j})\r)\r\} =\mu\lf(B_{\varrho_{\#}}(x,r_{j+1})\r),
$$
as desired.	This finishes the proof of (f) and, in turn, the proof of Lemma~\ref{HolderBump}.
\end{proof}

The next result is an abstract iteration scheme
that will be applied many times in the proofs of
our main results. A version of this lemma in the
setting of metric measure spaces was proven in
\cite[Lemma~16]{agh20}, which itself is an
abstract version of an argument used in
\cite{korobenkomr}. Since its proof  is a
simple modification of that of \cite[Lemma~16]{agh20}, we omit the details.

\begin{lemma}
\label{iteration}
Let $(X,\rho,\mu)$ be a quasi-metric measure space.
Suppose that $a,\,b,\,p,\,t,\,\theta\in(0,\infty)$ satisfy $a<b$ and $p<t$.
Let  $x\in X$ and $\{r_j\}_{j\in\mathbb{N}}$  be  a sequence of radii  such that
\begin{equation*}
a\leq r_j\leq b
\quad
\text{and}
\quad
\lf[\mu(B_\rho(x,r_{j+1}))\r]^{1/t}\leq\theta 2^j \lf[\mu(B_\rho(x,r_j))\r]^{1/p},
\qquad
\forall\ j\in\mathbb{N}.
\end{equation*}
Then
$$\mu(B_\rho(x,r_{1}))\geq \theta^{-pt/(t-p)}\,2^{-pt^2/(t-p)^2}.$$
\end{lemma}

\subsection{The Non-Doubling Case}\label{non-db}

The first main result of this subsection, Theorem~\ref{LMeasINT} below,  establishes the
converse of Theorem~\ref{LBembedding}. Therefore, one has that
Theorems~\ref{LMeasINT} and  \ref{LBembedding} imply Theorem~\ref{LMeasINTCor}.

Let  $(X,\rho,\mu)$ be a quasi-metric measure space.
Recall that the measure $\mu$ is said to be  \emph{lower $Q$-Ahlfors-regular}
with  $Q\in(0,\infty)$, if there exists a positive constant $\kappa$ satisfying
\begin{eqnarray}\label{lowermeasure}
\kappa\,r^Q\leq\mu(B_\rho(x,r))\quad\mbox{for any}\ x\in X
\ \mbox{ and any finite }\ r\in(0,{\rm diam}_\rho(X)].
\end{eqnarray}
We also recall the following piece of notational convention.

\begin{convention}
\label{sindex}
Given a quasi-metric space $(X,\rho)$ and fixed numbers
$s\in(0,\infty)$ and $q\in(0,\infty]$, we will
understand by $s\preceq_q{\rm ind}\,(X,\rho)$
that $s\leq{\rm ind}\,(X,\rho)$ and that the
value $s={\rm ind}\,(X,\rho)$ is only
permissible when $q=\infty$ and the supremum
defining ${\rm ind}\,(X,\rho)$ in \eqref{index} is attained.
\end{convention}

\begin{theorem}\label{LMeasINT}
Let $(X,\rho,\mu)$ be a quasi-metric measure space and fix  $\sigma\in[1,\infty)$, $q\in(0,\infty]$, and $Q\in(0,\infty)$. Also, assume that
$s\in(0,\infty)$ satisfies $s\preceq_q{\rm ind}\,(X,\rho)$. Then the following statements are valid.
\begin{enumerate}[label=\rm{(\alph*)}]
\item Suppose that there exist a $p\in(0,Q/s)$ and a $C_S\in(0,\infty)$ such that, for any
$\rho$-ball $B:=B_\rho(x,r)$, where $x\in X$ and  $r\in(0,{\rm diam}_\rho(X)]$ is finite, and any
$u\in M^s_{p,q}(\sigma B,\rho,\mu)$,
\begin{equation}\label{hdx-1}
\left(\, \mvint_{B} |u|^{p^*}\, d\mu\right)^{1/p^*}\leq C_Sr^{-Q/p}
\left[r^s\|u\|_{\dot{M}^s_{p,q}(\sigma B)}+\Vert u\Vert_{L^p(\sigma B)}\right],
\end{equation}
where  $p^*:=Qp/(Q-sp)$. Then $\mu$ is lower $Q$-Ahlfors-regular.
\end{enumerate}	
	
\noindent If the space $(X,\rho)$ is uniformly perfect\footnote{Recall \eqref{U-perf}.}, then the following statements also hold true:
	
\begin{enumerate}[label=\rm{(\alph*)}]\addtocounter{enumi}{1}
\item Suppose that there exist a $p\in(0,Q/s)$ and a $C_P\in(0,\infty)$ such that, for
any $\rho$-ball $B:=B_\rho(x,r)$, where $x\in X$ and  $r\in(0,{\rm diam}_\rho(X)]$ is finite,
and any $u\in M^s_{p,q}(\sigma B,\rho,\mu)$,
\begin{equation}\label{HHs-175}
\inf_{\gamma\in\mathbb{R}}\left(\, \mvint_{B} |u-\gamma|^{p^*}\, d\mu\right)^{1/p^*}\leq
C_Pr^{s-Q/p}\|u\|_{\dot{M}^s_{p,q}(\sigma B)},
\end{equation}
where $p^*:=Qp/(Q-sp)$. Then $\mu$ is lower $Q$-Ahlfors-regular.
		
\item Suppose that  there exist positive constants $C_1,\,C_2$, and $\omega$ such that,
for any $\rho$-ball $B\subset X$  $($with radius at most ${\rm diam}_\rho(X)$$)$
and any $u\in M^s_{Q/s,\,q}(\sigma B,\rho,\mu)$
with $\|u\|_{\dot{M}^s_{Q/s,q}(\sigma B)}>0$,
\begin{eqnarray}\label{hdx-3}
\mvint_{B} {\rm exp}\lf(C_1\frac{|u-u_{B}|}{\|u\|_{\dot{M}^s_{Q/s,q}(\sigma B)}}\r)^{\omega}\,d\mu\leq C_2.
\end{eqnarray}
Then $\mu$ is lower $Q$-Ahlfors-regular.
		
\item Suppose that there exist a $p\in(Q/s,\infty)$
and a $C_H\in(0,\infty)$ such that
every   $u\in M^s_{p,q}(X,\rho,\mu)$ has
a H\"older continuous representative of order $s-Q/p$ on $X$,
denoted by $u$ again,  satisfying
\begin{eqnarray}\label{hdx-4}
|u(x)-u(y)|\leq C_H [\rho(x,y)]^{s-Q/p}\|u\|_{\dot{M}^s_{p,q}(X)},\qquad\forall\,x,\,y\in X.
\end{eqnarray}
Then $\mu$ is lower $Q$-Ahlfors-regular.
\end{enumerate}
In addition, all of the statements above are also valid with $M^s_{p,q}$ and $\dot{M}^s_{p,q}$ replaced by $N^s_{p,q}$ and $\dot{N}^s_{p,q}$, respectively.
\end{theorem}

\begin{remark}
If one assumes that the inequalities in \eqref{hdx-1}-\eqref{hdx-4} only hold true for balls having radii at most $r_\ast\in(0,\infty)$ then the proof of Theorem~\ref{LMeasINT} that is presented below can be modified to show that $\mu$ is lower $Q$-Ahlfors-regular on balls having radii at most $r_\ast\in(0,\infty)$. In this case, one only needs to assume that the uniformly perfect property holds true on balls with small radii.
\end{remark}

\begin{proof}[Proof of Theorem~\ref{LMeasINT}]
Since $s\preceq_q{\rm ind}\,(X,\rho)$, we can choose a
quasi-metric $\varrho$ on $X$ such that $\varrho\approx\rho$  and $s\leq(\log_{2}C_\varrho)^{-1}$, where  $C_\varrho\in[1,\infty)$ is as in \eqref{C-RHO.111}. Note that, by convention,
the value  $s=(\log_{2}C_\varrho)^{-1}$ is only
permissible when $q=\infty$ and $C_\varrho>1$.
In what follows, we let $\varrho_\#$ be the
regularized quasi-metric given by Theorem~\ref{DST1}.

We begin with the proof of the statement ({a}).
Fix a ball $B:=B_\rho(x,r)$ with $x\in X$ and a finite $r\in(0,{\rm diam}_\rho(X)]$.
Let $\{r_j\}_{j\in\mathbb{N}}$ and $\{u_j\}_{j\in\mathbb{N}}$ be as in
Lemma~\ref{HolderBump}  (associated to the ball $B$).
Then, for any $j\in\nn$,  the pointwise restriction of $u_j$ to the
ball $\sigma B$ belongs to $M^s_{p,q}(\sigma B,\rho,\mu)$
and therefore the function $u_j$ satisfies \eqref{hdx-1}.
Moreover, $B_{\varrho_{\#}}(x,r_j)\subset B$, due to Lemma~\ref{HolderBump}({a}). From these,
it follows from the properties listed in
({b})-({e})  of Lemma~\ref{HolderBump}
that, for any $j\in\mathbb{N}$,
\begin{equation}\label{HD-1-X}
\mbox{$\|u_j\|_{\dot{M}^s_{p,q}(\sigma B)}\lesssim r^{-s}2^{j}\lf[\mu(B_{\varrho_{\#}}(x,r_j))\r]^{1/p}$\qquad and\qquad $\|u_j\|_{L^{p}(\sigma B)}\lesssim \lf[\mu(B_{\varrho_{\#}}(x,r_{j}))\r]^{1/p}$.}
\end{equation}
Moreover, since $u_j\equiv1$ on $B_{\varrho_\#}(x,r_{j+1})$, we have
\begin{equation}\label{HD-3-X}
\left(\, \mvint_{B} |u_j|^{p^*}\, d\mu\right)^{1/p^*}\geq\left[\frac{\mu(B_{\varrho_{\#}}(x,r_{j+1}))}{\mu(B_\rho(x,r))}\right]^{1/p^*}.
\end{equation}
Combining \eqref{hdx-1}, \eqref{HD-1-X}, and \eqref{HD-3-X}, we find that
\begin{equation*}
\begin{split}
\left[\frac{\mu(B_{\varrho_{\#}}(x,r_{j+1}))}{\mu(B_\rho(x,r))}\right]^{1/p^*}
&\lesssim r^{-Q/p}\lf(2^{j}+1\r)\lf[\mu(B_{\varrho_\#}(x,r_j))\r]^{1/p}
\\
&\lesssim r^{-Q/p}\,2^{j}\,\lf[\mu(B_{\varrho_{\#}}(x,r_j))\r]^{1/p},\qquad\forall\,j\in\mathbb{N}.\nonumber
\end{split}
\end{equation*}
Therefore, there exists a positive constant  $C'$, independent of $x$, $r$, and $j$, such that
\begin{equation}
\label{hj-10}
\lf[\mu(B_{\varrho_{\#}}(x,r_{j+1}))\r]^{1/p^*}
\leq
\left(\frac{C'[\mu(B_{\rho}(x,r))]^{1/p^*}}{r^{Q/p}}\right) 2^{j}\lf[\mu(B_{\varrho_{\#}}(x,r_j))\r]^{1/p},
\qquad\forall\,j\in\mathbb{N}.
\end{equation}
Since $0<\delta r<r_j<r<\infty$ for every $j\in\mathbb{N}$ (where $\delta\in(0,1)$ is as in Lemma~\ref{HolderBump}),
we  invoke Lemma~\ref{iteration} with the quasi-metric $\varrho_{\#}$, and
$$
p:=p,\quad
t:=p^*,\quad
\mbox{and}\quad
\theta:=\frac{C'[\mu(B_\rho(x,r))]^{1/p^*}}{r^{Q/p}}
$$
to deduce that
\begin{equation}
\label{hj-1}
\mu(B_\rho(x,r))\geq\mu(B_{\varrho_{\#}}(x,r_1))\gtrsim 2^{-\frac{Q^2}{s^2p}}\lf[\mu(B_\rho(x,r))\r]^{1-\frac{Q}{sp}}\, r^\frac{Q^2}{sp},
\end{equation}
which further implies  \eqref{lowermeasure}. This finishes the proof of  ({a}).
	
To prove ({b}), ({c}), and ({d}) of the present theorem,
we now assume that the space $(X,\rho)$ is uniformly perfect and
let $\lambda\in(0,1)$ be as in \eqref{U-perf}.
For the remainder of the proof, without loss of generality,  we may
assume\footnote{Recall that we may make this assumption
without sacrificing the generality of our result.}
$\lambda<(C_\rho\widetilde{C}_\rho)^{-2}$,
where $C_\rho,\,\widetilde{C}_\rho\in[1,\infty)$
are as in \eqref{C-RHO.111} and \eqref{C-RHO.111XXX}, respectively.

In order to prove the statement (b), fix a point $x\in X$ and a finite radius $r\in(0,{\rm diam}_\rho(X)]$.
Let $B:=B_\rho(x,r)$.  Our goal is to establish an estimate like the one in \eqref{hj-10}.
To accomplish this, assume for the moment that $r\leq C_\rho\varphi^x_{\rho}(r)/\lambda^2$.
Then, by Lemma~\ref{HolderBump}  with $c_0:=C_\rho\lambda^{-2}\in(1,\infty)$,
we find that there exist a sequence $\{r_j\}_{j\in\mathbb{N}}\subset(0,\fz)$ and a collection $\{u_j\}_{j\in\mathbb{N}}$
of functions
 satisfying the properties (a)-(f)   of
Lemma~\ref{HolderBump}. In particular,    for any $j\in\nn$,  $u_j$ satisfies
\begin{equation}
\label{Gkw-924}
r^{s-Q/p}\,\|u_j\|_{\dot{M}^s_{p,q}(\sigma B)}\lesssim 2^{j}r^{-Q/p}\lf[\mu(B_{\varrho_\#}(x,r_j))\r]^{1/p}
\end{equation}
and, accordingly,  $u_j$ satisfies \eqref{HHs-175}, due to the assumption of (b).
Moreover,    Lemma~\ref{HolderBump}(f) implies that
\begin{align}
\label{Gkw-925}
\inf_{\gamma\in\mathbb{R}}\left(\, \mvint_{B} |u-\gamma|^{p^*}\, d\mu\right)^{1/p^*}
\ge \frac12 \inf_{\gamma\in\mathbb{R}} \left[\frac{\mu(E_{j}^\gamma)}{\mu(B_\rho(x,r))}\right]^{1/p^*}
\geq\frac{1}{2}\left[\frac{\mu(B_{\varrho_{\#}}(x,r_{j+1}))}{\mu(B_\rho(x,r))}\right]^{1/p^*},
\end{align}
where $E_j^\gamma$ is as in  Lemma~\ref{HolderBump}{(f)}.
Combining \eqref{HHs-175}, \eqref{Gkw-924}, and \eqref{Gkw-925}, we conclude that
\begin{equation*}
\frac{1}{2}\left[\frac{\mu(B_{\varrho_{\#}}(x,r_{j+1}))}{\mu(B_\rho(x,r))}\right]^{1/p^*}
\lesssim 2^{j}\,r^{-Q/p}\lf[\mu(B_{\varrho_{\#}}(x,r_j))\r]^{1/p}.
\end{equation*}
Hence,
\begin{equation*}
\lf[\mu(B_{\varrho_{\#}}(x,r_{j+1}))\r]^{1/p^*}
\leq
\left(\frac{C''[\mu(B_\rho(x,r))]^{1/p^*}}{r^{Q/p}}\right) 2^{j}\lf[\mu(B_{\varrho_{\#}}(x,r_j))\r]^{1/p},
\quad\forall\,j\in\mathbb{N},
\end{equation*}
where $C''$ is a positive
constant independent of $x$, $r$, and $j$.
At this stage, we have arrived at a step similar to \eqref{hj-10}.
Thus, arguing as in \eqref{hj-10} and \eqref{hj-1}, we conclude that
$\mu$ satisfies the desired estimate in \eqref{lowermeasure}
for any balls having radii $r\in (0,{\rm diam}_\rho(X)]$
for which $r\leq C_\rho\varphi^x_{\rho}(r)/\lambda^2$. Now, appealing to
Lemma~\ref{en2-4}(i) implies that
$$
\mu(B_\rho(x,r))\gtrsim r^Q\quad\mbox{whenever\ $x\in X$\ and\ $r\in (0,{\rm diam}_\rho(X)]$\ is\ finite.}
$$
This finishes the proof of the statement in ({b}).
	
We now turn our attention to proving ({c}) by fixing $x\in X$ and a finite number $r\in(0,{\rm diam}_\rho(X)]$,
and setting $B:=B_\rho(x,r)$. As in the above proof of ({b}), we
first consider the case when $r\leq C_\rho\varphi^x_{\rho}(r)/\lambda^2$.
Let $\{r_j\}_{j\in\mathbb{N}}$ and $\{u_j\}_{j\in\mathbb{N}}$
be as in Lemma~\ref{HolderBump} having the properties ({a})-({f}) therein.
Then
\begin{equation}
\label{yre-1}
0<\|u_j\|_{\dot{M}^s_{Q/s,q}(\sigma B)}\lesssim 2^{j}r^{-s}\lf[\mu(B_{\varrho_{\#}}(x,r_j))\r]^{s/Q},
\quad\forall\,j\in\mathbb{N}.
\end{equation}
Hence, $u_j$ satisfies \eqref{hdx-3}. Now,  combining \eqref{yre-1},  Lemma~\ref{HolderBump}(f),
and \eqref{hdx-3}, with $\gamma:=(u_j)_B$, we conclude that
\begin{align}
\label{Iue.5-2-X}
&\frac{\mu(B_{\varrho_{\#}}(x,r_{j+1}))}{\mu(B_\rho(x,r))}\,
{\rm exp}\left(\frac{Cr^s}{2^{j+1}\,[\mu(B_{\varrho_{\#}}(x,r_{j}))]^{s/Q}}\right)^\omega\nonumber\\
&\quad\le
\frac{1}{\mu(B_\rho(x,r))}\,\int_{E_j^\gamma}
{\rm exp}\left(\frac{Cr^s}{2^{j+1}\,[\mu(B_{\varrho_{\#}}(x,r_{j}))]^{s/Q}}\right)^\omega\,d\mu\nonumber\\
&\quad\le
\frac{1}{\mu(B_\rho(x,r))}\,\int_{E_j^\gamma}
{\rm exp}\left(\frac{C_1|u_j-(u_j)_B|}{\|u_j\|_{\dot{M}^s_{Q/s,q}(\sigma B)}}\right)^\omega\,d\mu\leq C_2
\end{align}
for some constant $C\in(0,\infty)$ independent of $x$, $r$, and $j$,
where $E_j^\gamma$ is as in Lemma~\ref{HolderBump}(f).
Without loss of generality, we may assume that $C_2>1$.
Then, using the estimate
$$\log(y)\leq 2Q(s\omega)^{-1}\,y^{s\omega/2Q},\quad\ \forall\ \  y\in(0,\infty),$$
a rewriting of \eqref{Iue.5-2-X} implies
\begin{align*}
\frac{Cr^s}{2^{j}\,[\mu(B_{\varrho_{\#}}(x,r_{j}))]^{s/Q}}
&\leq\lf[\log\lf(C_2\frac{\mu(B_\rho(x,r))}{\mu(B_{\varrho_{\#}}(x,r_{j+1}))}\r)\r]^{1/\omega}\\
&\leq \lf[2Q(s\omega)^{-1}\r]^{1/\omega}\,C_2^{s/(2Q)}\lf[\frac{\mu(B_\rho(x,r))}{\mu(B_{\varrho_{\#}}(x,r_{j+1}))}\r]
^{s/(2Q)}.
\end{align*}
Therefore,
\begin{equation*}
\lf[\mu(B_{\varrho_{\#}}(x,r_{j+1}))\r]^{s/(2Q)}\le \lf[2Q(s\omega)^{-1}\r]^{1/\omega}\,C_2^{s/(2Q)}
 \frac{[\mu(B_\rho(x,r))]^{s/(2Q)}}{Cr^s} \,2^{j}\,\lf[\mu(B_{\varrho_{\#}}(x,r_{j}))\r]^{s/Q}.
\end{equation*}
Now, applying Lemma~\ref{iteration} with the quasi-metric $\varrho_{\#}$, and
$$
p:=Q/s,\quad
t:=2Q/s,\quad
\text{and}
\quad
\theta:=\lf[2Q(s\omega)^{-1}\r]^{1/\omega}\,C_2^{s/(2Q)}\frac{\mu(B_\rho(x,r))^{s/(2Q)}}{Cr^s},
$$
we obtain
$$
\mu(B_\rho(x,r))\geq\mu(B_{\varrho_{\#}}(x,r_1))\gtrsim
\left[\frac{\mu(B_\rho(x,r))^{s/(2Q)}}{Cr^s}\right]^{-2Q/s}\,
2^{-4Q/s}.
$$
Therefore, we have $\mu(B_\rho(x,r))\gtrsim \,r^Q$
for all $\rho$-balls having radii $r\in\big(0,{\rm diam}_\rho(X)\big]$
for which $r\leq C_\rho\varphi^x_{\rho}(r)/\lambda^2$.
Now the desired condition in  \eqref{lowermeasure}
follows from  Lemma~\ref{en2-4}(i). The proof of (c) is then complete.

It remains to prove the statement in (d). Fix a point $x\in X$
and a finite radius $r\in (0,{\rm diam}_\rho(X)]$. If $B:=B_\rho(x,r)=X$,
then $r={\rm diam}_\rho(X)\in(0,\infty)$ and $\mu(X)\in(0,\infty)$, and  hence
\begin{equation*}
\mu(B_\rho(x,r))=\mu(X)=\mu(X)[{\rm diam}_\rho(X)]^{-Q}\,r^Q.
\end{equation*}
Thus, the measure condition in \eqref{lowermeasure} is
satisfied with $\kappa:=\mu(X)[{\rm diam}_\rho(X)]^{-Q}\in(0,\infty)$.
Therefore, we will assume that $X\setminus B_\rho(x,r)\neq\emptyset$
in what follows. Since $\varrho\approx\rho$ and $\varrho_{\#}\approx\varrho$
(see Theorem~\ref{DST1}), we have $\varrho_{\#}\approx\rho$.
Hence, there is a constant $c\in(1,\infty)$ such that
$$
c^{-1}\varrho_{\#}(z,y)\leq\rho(z,y)\leq c\,\varrho_{\#}(z,y),\quad\forall\,z,\, y\in X.
$$

If we let $r_0:=r/c$, then $B_{\varrho_{\#}}(x,r_0)\subset B_\rho(x,r)$,
and so $X\setminus B_{\varrho_{\#}}(x,r_0)\neq\emptyset$. Moreover,
since $(X,\rho)$ is uniformly perfect and $\varrho_{\#}\approx \rho$,
we find that $(X,\varrho_{\#})$ also satisfies the uniformly perfect property
with  constant $\lambda':=c^{-2}\lambda<\lambda$. Thus, we may select a
point $x_0\in B_{\varrho_{\#}}(x,r_0)\setminus B_{\varrho_{\#}}(x,\lambda' r_0)$.
Fix any finite number $\alpha\in[s,(\log_{2}C_\varrho)^{-1}]$,
where $\alpha\neq s$ unless $s=(\log_{2}C_\varrho)^{-1}$, and
consider the function $\Phi_{0,\lambda' r_0}\colon\, X\to[0,1]$
given by Lemma~\ref{GVa2}. Then $\Phi_{0,\lambda' r_0}\in M^s_{p,q}(X,\rho,\mu)$ and \eqref{gradest} implies that
$$
\lf\|\Phi_{0,\lambda' r_0}\r\|_{\dot{M}^s_{p,q}(X)}
\lesssim \lf[(\lambda' r_0)^\alpha-0^\alpha\r]^{-s/\alpha}\lf[\mu(B_{\varrho_\#}(x,\lambda' r_0))\r]^{1/p}\lesssim (\lambda' r_0)^{-s}\lf[\mu(B_{\varrho_\#}(x,\lambda' r_0))\r]^{1/p}.
$$
In particular,   $\Phi_{0,\lambda' r_0}$
satisfies \eqref{hdx-4}. Then, since
$x_0\in B_{\varrho_{\#}}(x,r_0)\setminus B_{\varrho_{\#}}(x,\lambda' r_0)$,  by
the inclusion $B_{\varrho_{\#}}(x,\lambda'r_0)\subset B_\rho(x,r)$,
inequality \eqref{hdx-4}, and the fact that $\Phi_{0,\lambda' r_0}$ is
supported in $B_{\varrho_{\#}}(x,\lambda' r_0)$, we conclude that
\begin{align*}
1&=\lf|\Phi_{0,\lambda' r_0}(x)-\Phi_{0,\lambda' r_0}(x_0)\r|
\leq C_H [\rho(x,x_0)]^{s-Q/p}\lf\|\Phi_{0,\lambda' r_0}\r\|_{\dot{M}^s_{p,q}(X)}\\
&\lesssim r_0^{s-Q/p}\lf(\lambda' r_0\r)^{-s}\lf[\mu(B_{\varrho_{\#}}(x,\lambda' r_0))\r]^{1/p}
\lesssim r^{-Q/p}\lf[\mu(B_\rho(x,r))\r]^{1/p},
\end{align*}
from which \eqref{lowermeasure} follows. This finishes the proof of (d).

The proof of statements (a)-(d) for Haj\l asz--Besov spaces
follows along the same line of reasoning as above.
This finishes the proof of Theorem~\ref{LMeasINT}.	
\end{proof}

According to Theorems~\ref{LMeasINT} and  \ref{LBembedding}, when $p>Q/s$,
the lower Ahlfors-regularity condition \eqref{lowermeasure}
is equivalent to   functions in $M^s_{p,q}$ and $N^s_{p,q}$
satisfying a global H\"older-type condition of order $s-Q/p$
(see \eqref{hdx-4}). In this vein,
we conclude this section by investigating the
relationship between \eqref{lowermeasure} and the
global Sobolev and Sobolev--Poincar\'e inequalities
for the spaces $M^s_{p,q}$ and $N^s_{p,q}$ when $p<Q/s$.

\begin{theorem}\label{GlobalEmbedd1}
Let $(X,\rho,\mu)$ be a quasi-metric measure space and
fix parameters $s,\,p,\,Q\in(0,\infty)$ and $q\in(0,\infty]$
satisfying $s\preceq_q{\rm ind}\,(X,\rho)$ and $sp<Q$.
Then, with $p^*:=Qp/(Q-sp)$, the following statements are valid.
\begin{enumerate}[label=\rm{(\alph*)}]
\item If there exists a positive constant $C_1$ satisfying
\begin{equation*}
\|u\|_{L^{p^\ast}(X)}\leq C_1\|u\|_{M^s_{p,q}(X)}\quad\ \forall\ u\in M^s_{p,q}(X,\rho,\mu),
\end{equation*}
then there exists a positive constant $\kappa$ such that
\begin{eqnarray*}
\mu(B_\rho(x,r))\geq \kappa\,r^Q \quad\mbox{for any}\ x\in X\ \mbox{ and any}\ r\in(0,1].
\end{eqnarray*}
\end{enumerate}

\noindent If the space $(X,\rho)$ is uniformly perfect,
then the following statement also holds true.

\begin{enumerate}[label=\rm{(\alph*)}]\addtocounter{enumi}{1}

\item If there exists a positive constant $C_2$ satisfying
\begin{equation}\label{yvx-1}
\inf_{\gamma\in\mathbb{R}}\|u-\gamma\|_{L^{p^\ast}(X)}\leq
C_2\|u\|_{\dot{M}^s_{p,q}(X)},\quad\ \forall\ u\in \dot{M}^s_{p,q}(X,\rho,\mu),
\end{equation}
then there exists a positive constant $\kappa$ such that
\begin{eqnarray*}
\mu(B_\rho(x,r))\geq \kappa\,r^Q\quad\  \mbox{for any }\ x\in X\ \mbox{ and any finite }\ r\in(0,{\rm diam}_\rho(X)].
\end{eqnarray*}
\end{enumerate}
In addition, all of the statements above  are also valid with
$M^s_{p,q}$ and $\dot{M}^s_{p,q}$ replaced, respectively,
by $N^s_{p,q}$ and $\dot{N}^s_{p,q}$.
\end{theorem}

The proofs of parts (a) and (b) in
Theorem~\ref{GlobalEmbedd1} are similar to those
of (a) and (b) in
Theorem~\ref{LMeasINT}, respectively. We omit the  details.

\begin{remark}
As a consequence of Theorems~\ref{GlobalEmbedd1}
and \ref{LBembedding},  on uniformly perfect quasi-metric measure spaces,
the global Sobolev--Poincar\'e inequality in \eqref{yvx-1},
as well as  its Haj\l asz--Besov space variant, implies all
of the local estimates in \eqref{eq18-LB}-\eqref{eq20-LB}
as well as the global H\"older condition in \eqref{eq30-LB}.
\end{remark}

As is illustrated by the following result, in certain settings,
the lower Ahlfors-regularity condition \eqref{lowermeasure}
is actually equivalent to the global Sobolev
and Sobolev--Poincar\'e inequalities for the
spaces $\dot{M}^s_{p,q}$ and $\dot{N}^s_{p,q}$   when $p<Q/s$.

\begin{theorem}\label{GlobalEmbeddCor}
Let $(X,\rho,\mu)$ be a quasi-metric measure space and fix parameters
$s,\,p,\,Q\in(0,\infty)$ and $q\in(0,\infty]$ satisfying
$s\preceq_q{\rm ind}\,(X,\rho)$ and $sp<Q$.
Also, suppose that either ${\rm diam}_\rho(X)<\infty$
or $\mu$ is $Q$-doubling $($see \eqref{Doub-2}$)$.	
Then, with $p^*:=Qp/(Q-sp)$, the following statements are equivalent.
\begin{enumerate}[label=\rm{(\alph*)}]
\item There exists a positive constant $\kappa$ such that
\begin{equation}
\label{gxz-2}
\kappa\,r^Q\leq\mu(B_\rho(x,r))\,\,\,\mbox{for any }\ x\in X\
\mbox{and any finite}\ r\in(0,{\rm diam}_\rho(X)].
\end{equation}
	
\item There exists a  positive constant $C_1$ satisfying
\begin{equation}\label{GBLS}
\|u\|_{L^{p^*}(X)}\leq C_1\|u\|_{\dot{M}^s_{p,q}(X)}+\frac{C_1}{[{\rm diam}_\rho(X)]^s}\,\|u\|_{L^p(X)}\quad\mbox{for any }\ u\in \dot{M}^s_{p,q}(X,\rho,\mu).
\end{equation}		
\end{enumerate}

\noindent If the space $(X,\rho)$ is uniformly perfect, then {\rm (a)} (hence,
also {\rm(b)}) is further equivalent to

\begin{enumerate}[label=\rm{(\alph*)}]\addtocounter{enumi}{2}	
\item There exists a positive constant $C_2$ satisfying
\begin{equation}\label{GBLP}
\inf_{\gamma\in\mathbb{R}}\|u-\gamma\|_{L^{p^\ast}(X)}\leq
C_2\|u\|_{\dot{M}^s_{p,q}(X)}\quad\mbox{for any}\ u\in \dot{M}^s_{p,q}(X,\rho,\mu).
\end{equation}
\end{enumerate}
In addition, if $q\leq p$ then all of the
statements above continue to be equivalent with $\dot{M}^s_{p,q}$ replaced by $\dot{N}^s_{p,q}$.
\end{theorem}

\begin{proof}
It follows from Theorem~\ref{GlobalEmbedd1}(b) that
Theorem~\ref{GlobalEmbeddCor}(c) implies Theorem~\ref{GlobalEmbeddCor}(a).

Now we prove  Theorem~\ref{GlobalEmbeddCor}(b)
implies Theorem~\ref{GlobalEmbeddCor}(a).
If  ${\rm diam}_\rho(X)<\infty$,
then, by Theorem~\ref{GlobalEmbedd1}(a),  \eqref{GBLS} implies that
there exists a positive constant $\kappa$ such that
\begin{equation*}
\mu(B_\rho(x,r))\geq \kappa\,r^Q \quad \text{for any $x\in X$ and  $r\in(0,1]$.}
\end{equation*}
If
$1<r\leq{\rm diam}_\rho(X)<\infty$, then $[{\rm diam}_\rho(X)]^{-1}r\leq1$, which implies
$$\mu(B_\rho(x,r))\geq \mu\lf(B_\rho\lf(x,[{\rm diam}_\rho(X)]^{-1}r\r)\r)\geq\kappa[{\rm diam}_\rho(X)]^{-Q}\,r^Q.$$
Hence, \eqref{GBLS} implies that \eqref{gxz-2}
in the case when ${\rm diam}_\rho(X)<\infty$.
If ${\rm diam}_\rho(X)=\infty$, then \eqref{GBLS} reduces to
\begin{equation*}
\|u\|_{L^{p^*}(X)}\ls \|u\|_{\dot{M}^s_{p,q}(X)},\quad\forall\,u\in \dot{M}^s_{p,q}(X,\rho,\mu).
\end{equation*}
At this stage, one can also prove \eqref{gxz-2}, by an argument similar to that used in  the proof of  Theorem~\ref{LMeasINT}(a). This finishes the proof of
showing that Theorem~\ref{GlobalEmbeddCor}(b) implies Theorem~\ref{GlobalEmbeddCor}(a).

It remains to show that Theorem~\ref{GlobalEmbeddCor}(a)
implies both (b)  and (c) of Theorem~\ref{GlobalEmbeddCor}.
At this time we need to use  the assumption that either ${\rm diam}_\rho(X)<\infty$
or $\mu$ is $Q$-doubling.

Assume that \eqref{gxz-2} holds true and suppose first that ${\rm diam}_\rho(X)<\infty$.
Observe that, if $x\in X$ and $r\in({\rm diam}_\rho(X),2{\rm diam}_\rho(X)]$,
then   $r/2\leq{\rm diam}_\rho(X)$, which, together with \eqref{gxz-2}, implies
$$\mu(B_\rho(x,r))\geq \mu(B_\rho\big(x,r/2\big))\geq\kappa2^{-Q}\,r^Q.$$
Hence, \eqref{gxz-2} holds true for any
$x\in X$ and $r\in(0,2{\rm diam}_\rho(X)]$, with $\kappa$ therein
replaced by $\kappa 2^{-Q}$.
Then, by taking any $\rho$-ball $B_0:=X$ having radius $R_0:=2{\rm diam}_\rho(X)$,
and applying  Remark~\ref{embedlocalLB} (with $r_\ast=2{\rm diam}_\rho(X)$), we conclude that
\eqref{GBLS} and \eqref{GBLP} are immediate consequences of \eqref{eq18-LB}
and \eqref{eq19-LB} in Theorem~\ref{LBembedding}. This proves that  Theorem~\ref{GlobalEmbeddCor}(a)
implies both (b)  and (c) of Theorem~\ref{GlobalEmbeddCor} in the case
when ${\rm diam}_\rho(X)<\infty$.

Now we assume that \eqref{gxz-2} holds true and $\mu$ is $Q$-doubling.
Fix
$u\in \dot{M}^s_{p,q}(X,\rho,\mu)$. Then, for any $\rho$-ball $B\subset X$,
$u\in \dot{M}^s_{p,q}(C_\rho B,\rho,\mu)$, where $C_\rho$ is as in \eqref{C-RHO.111}.
From this, applying  \eqref{eq18-DOUB} in Theorem~\ref{DOUBembedding}
with $\sigma:=C_\rho$ and \eqref{gxz-2}, we deduce that,
for any $\rho$-ball $B:=B_\rho(x,r)\subset X$ with $r\in(0,{\rm diam}_\rho(X))$,
\begin{align}
\label{eq18-DOUB-X}
\Vert u\Vert_{L^{p^\ast}(B)}
&\lesssim \lf[\mu(\sigma B)\r]^{-s/Q}\left[
r^s\Vert u\Vert_{\dot{M}^s_{p,q}(\sigma B)}
+\Vert u\Vert_{L^{p}(\sigma B)}\right]\nonumber\\
&\lesssim r^{-s}\left[
r^s\Vert u\Vert_{\dot{M}^s_{p,q}(\sigma B)}
+\Vert u\Vert_{L^{p}(\sigma B)}\right]\nonumber
\\
&\lesssim \Vert u\Vert_{\dot{M}^s_{p,q}(X)}
+r^{-s}\Vert u\Vert_{L^{p}(X)}.
\end{align}
Letting $r$ tend to  ${\rm diam}_\rho(X)$ in \eqref{eq18-DOUB-X},  we obtain
the desired estimate \eqref{GBLS}. Thus, (b) is proven.

We now turn our attention to proving (c).
By what we have already shown above,
it suffices to consider the case when
${\rm diam}_\rho(X)=\infty$. With this in mind,
fix a point $x_0\in X$ and let $B_k:=B(x_0,k)$ for any $k\in\nn$.
Appealing to inequality \eqref{eq19-DOUB} in Theorem~\ref{DOUBembedding}
(used here with $\sigma:=C_\rho$), the lower measure
bound in \eqref{gxz-2} gives
\begin{equation*}
\inf_{\gamma\in\mathbb{R}}\Vert u-\gamma\Vert_{L^{p^\ast}(B_k)}\lesssim
\lf[\mu(\sigma B_k)\r]^{-s/Q}\,k^{s}\Vert u\Vert_{\dot{M}^s_{p,q}(\sigma B_k)}
\lesssim\Vert u\Vert_{\dot{M}^s_{p,q}(X)},\quad\forall\ k\in\nn.
\end{equation*}
Thus, for any $k\in\nn$, there exists a $\gamma_k\in\mathbb{R}$ satisfying
\begin{equation}
\label{eq19-DOUB-XX}
\Vert u-\gamma_k\Vert_{L^{p^\ast}(B_k)}
\lesssim\Vert u\Vert_{\dot{M}^s_{p,q}(X)}.
\end{equation}
On the other hand, by \eqref{eq18-DOUB-X} with $B$ therein replaced by $B_k$, we have
\begin{equation}
\label{ayrtb-3}
\Vert u\Vert_{L^{p^\ast}(B_1)}\leq\Vert u\Vert_{L^{p^\ast}(B_k)}\lesssim \Vert u\Vert_{\dot{M}^s_{p,q}(X)}
+k^{-s}\Vert u\Vert_{L^{p}(X)},\quad\forall\ k\in\mathbb{N},
\end{equation}
which further implies that
$\Vert u\Vert_{L^{p^\ast}(B_1)}\lesssim
\Vert u\Vert_{\dot{M}^s_{p,q}(X)}$, via letting $k\to\fz$ in
\eqref{ayrtb-3}. Combining this, \eqref{eq19-DOUB-XX}, and
\eqref{gxz-2}, we find that, for any $k\in\nn$,
\begin{equation*}
\begin{split}
|\gamma_k|&\lesssim \lf[\mu(B_1)\r]^{-1/p^*}\left[\Vert u-\gamma_k\Vert_{L^{p^\ast}(B_1)}+\Vert u\Vert_{L^{p^\ast}(B_1)}\right]\\
&\lesssim\Vert u-\gamma_k\Vert_{L^{p^\ast}(B_k)}+\Vert u\Vert_{L^{p^\ast}(B_1)}\lesssim\Vert u\Vert_{\dot{M}^s_{p,q}(X)}<\infty,
\end{split}
\end{equation*}
which implies that $\{\gamma_{k}\}_{k\in\nn}$ is a bounded sequence of
real numbers. From this, we deduce that there exist a
subsequence $\{\gamma_{k_j}\}_{j\in\nn}$ and a constant $\gamma_\ast\in\rr$
such that $\gamma_\ast=\lim_{j\to\infty}\gamma_{k_j}$. As such, by \eqref{eq19-DOUB-XX} and Fatou's lemma, we further conclude that
\begin{equation*}
\begin{split}
\inf_{\gamma\in\mathbb{R}}\|u-\gamma\|_{L^{p^\ast}(X)}
&\leq\|u-\gamma_\ast\|_{L^{p^\ast}(X)}
=\left[\int_{X}\lim_{j\to\infty}\big|(u(x)-\gamma_{k_j})
{\bf 1}_{B_{k_j}}(x)\big|^{p^\ast}d\mu(x)\right]^{1/{p^\ast}}\\
&\leq\liminf_{j\to\infty}\left[\int_{B_{k_j}}\big|u(x)-\gamma_{k_j}\big|^{p^\ast}d\mu(x)\right]^{1/p^\ast}
\lesssim\Vert u\Vert_{\dot{M}^s_{p,q}(X)}.
\end{split}
\end{equation*}
This finishes the proof of  (c) and hence, the proof of Theorem~\ref{GlobalEmbeddCor}.
\end{proof}

\begin{remark}
\label{GlobalEmbeddCor-rmk}
It follows from Theorem~\ref{GlobalEmbeddCor} that, in the
context of uniformly perfect quasi-metric measure spaces
that are either bounded or equipped with a doubling
measure, the global estimates in \eqref{GBLS} and
\eqref{GBLP} are equivalent to the local estimates
in \eqref{hdx-1}-\eqref{hdx-3} as well as the
global H\"older condition in \eqref{eq30-LB}.
Additionally, it is worth noting that the
uniform perfect property is only used in
showing that \eqref{GBLP} implies \eqref{gxz-2}.
Moreover, the limitation
$s\preceq_q{\rm ind}\,(X,\rho)$  is only used in
proving that (\textit{a}) follows from   (\textit{b}) or (\emph{c}).
\end{remark}

\begin{remark}
\label{BesovUB}
The embeddings for the Haj\l asz--Besov spaces $\dot{N}^s_{p,q}$ in
Theorem~\ref{GlobalEmbeddCor} are restricted to the case
when $q\leq p$, which is a limitation that is inherited
from the embeddings in Theorems~\ref{LBembedding} and \ref{DOUBembedding}.
An upper bound on the exponent $q$ is natural because, on the one hand, \eqref{GBLS} implies
that $N^s_{p,q}(\mathbb{R}^n)\hookrightarrow L^{p^\ast}(\mathbb{R}^n)$,
while on the other hand, it follows from  \cite[Theorem~1.73]{T06} that  $N^s_{p,q}(\mathbb{R}^n)=B^s_{p,q}(\mathbb{R}^n)\hookrightarrow
L^{p^\ast}(\mathbb{R}^n)$ if and only if $q\leq p^\ast$.
In the general context of quasi-metric measure spaces,
it is not clear if the embedding properties of the spaces $N^s_{p,q}$ in
Theorems~\ref{LBembedding} and \ref{DOUBembedding}
hold true for all $q\leq p^\ast$ (when $p<Q/s$).
\end{remark}

\subsection{The Doubling Case}\label{dbc}

The main goal of this subsection is to
present an analogue of Theorem~\ref{LMeasINT} for
doubling measures in Theorem~\ref{DoubMeasINT}.
Then, Theorem~\ref{DoubMeasINTCor} will follow as a consequence of
Theorems~\ref{DOUBembedding} and \ref{DoubMeasINT}.
Recall that a measure $\mu$ is said to be \emph{$Q$-doubling}
for some  $Q\in(0,\infty)$, provided that there
exists a positive constant $\kappa$ satisfying
\begin{equation}\label{Doub-2-eh}
\kappa\,\bigg(\frac{r}{R}\bigg)^{Q}\leq\frac{\mu\big(B_\rho(x,r)\big)}{\mu\big(B_\rho(y,R)\big)},
\end{equation}
whenever $x,\,y\in X$ and $0<r\leq R<\infty$ satisfy $B_\rho(x,r)\subset B_\rho(y,R)$.

\begin{theorem}\label{DoubMeasINT}
Let $(X,\rho,\mu)$ be a quasi-metric measure space and fix  $\sigma\in[1,\infty)$,
$q\in(0,\infty]$, and $Q\in(0,\infty)$. Also,
assume that $s\in(0,\infty)$ satisfies $s\preceq_q{\rm ind}\,(X,\rho)$.
Then the following statements are valid.
\begin{enumerate}[label=\rm{(\alph*)}]
		
\item Suppose that there exist a $p\in(0,Q/s)$ and a $C_S\in(0,\infty)$ with the property that
\begin{equation}\label{doubsob}
\Vert u\Vert_{L^{p^\ast}(B_0)}\leq
\frac{C_S}{[\mu(\sigma B_0)]^{s/Q}}\left[
R_0^s\|u\|_{\dot{M}^s_{p,q}(\sigma B_0)}
+\Vert u\Vert_{L^{p}(\sigma B_0)}\right],
\end{equation}
whenever $B_0:=B_\rho(x_0,R_0)$ is a $\rho$-ball with $x_0\in X$ and $R_0\in(0,\infty)$, and $u\in M^s_{p,q}(\sigma B_0,\rho,\mu)$. Here, $p^*:=Qp/(Q-sp)$. Then $\mu$ is a $Q$-doubling measure.
		
\end{enumerate}

\noindent If the space $(X,\rho)$ is uniformly perfect, then the following statements also hold true.
		
\begin{enumerate}[label=\rm{(\alph*)}]\addtocounter{enumi}{1}
\item Suppose that there exist a $p\in(0,Q/s)$ and a $C_P\in(0,\infty)$ with the property that

\begin{equation}\label{doubpoin}
\inf_{\gamma\in\mathbb{R}}\Vert u-\gamma\Vert_{L^{p^\ast}(B_0)}\leq
\frac{C_P}{[\mu(\sigma B_0)]^{s/Q}}\,R_0^{s}\|u\|_{\dot{M}^s_{p,q}(\sigma B_0)},
\end{equation}
whenever $B_0:=B_\rho(x_0,R_0)$ is a $\rho$-ball with $x_0\in X$ and $R_0\in(0,\infty)$, and $u\in M^s_{p,q}(\sigma B_0,\rho,\mu)$.  Here, $p^*:=Qp/(Q-sp)$. Then $\mu$ is a $Q$-doubling measure.

\item Suppose that there exist positive constants $c_1$, $c_2$, and $\omega$ such that
\begin{equation}\label{amv-1}
\mvint_{B_0} {\rm exp}\left(c_1\frac{[\mu(\sigma B_0)]^{s/Q}|u-u_{B_0}|}{R_0^s\|u\|_{\dot{M}^s_{Q/s,q}(\sigma B_0)}}\right)^\omega\,d\mu\leq c_2,
\end{equation}
whenever $B_0\subset X$ is a $\rho$-ball of radius $R_0\in(0,\infty)$ and $u\in M^{s}_{p,q}(\sigma B_0,\rho,\mu)$ satisfies $\|u\|_{\dot{M}^s_{Q/s,q}(\sigma B_0)}>0$. Then $\mu$ is a $(Q+\varepsilon)$-doubling measure for any fixed $\varepsilon\in(0,\infty)$.

\item Suppose that there exist a $p\in(Q/s,\infty)$ and a $C_H\in(0,\infty)$
with the property that, for any fixed ball
$B_0:=B_\rho(x_0,R_0)$, where $x_0\in X$ and $R_0\in(0,\infty)$,
one has that every function
$u\in M^s_{p,q}(\sigma B_0,\rho,\mu)$ has a H\"older continuous
representative of order $s-Q/p$ on $B_0$, denoted by $u$ again,  satisfying
\begin{eqnarray}
\label{doubHold}
|u(x)-u(y)|\leq C_H\,\lf[\rho(x,y)\r]^{s-Q/p}\frac{R_0^{Q/p}}{[\mu(\sigma B_0)]^{1/p}}\,\|u\|_{\dot{M}^s_{p,q}(\sigma B_0)}
\quad\mbox{for any  $x,\,y\in B_0$.}
\end{eqnarray}
Then $\mu$ is a $Q$-doubling measure.
\end{enumerate}
In addition, all of the statements  above  are also valid with $M^s_{p,q}$ and $\dot{M}^s_{p,q}$ replaced by $N^s_{p,q}$ and $\dot{N}^s_{p,q}$, respectively.
\end{theorem}

\begin{remark}
If one assumes that the inequalities in \eqref{doubsob}-\eqref{doubHold} only hold true for balls having radii at most $r_\ast\in(0,\infty)$ then the proof of Theorem~\ref{DoubMeasINTCor}  that is presented below can be modified to show that $\mu$ is doubling on balls having radii at most $r_\ast\in(0,\infty)$. In this case, one only needs to assume that the uniformly perfect property holds true on balls with small radii.
\end{remark}

\begin{proof}[Proof of Theorem~\ref{DoubMeasINT}]
Since $s\preceq_q{\rm ind}\,(X,\rho)$, we   choose a quasi-metric $\varrho$ on $X$ such that $\varrho\approx\rho$  and $s\leq(\log_{2}C_\varrho)^{-1}$, where  $C_\varrho\in[1,\infty)$ is as in \eqref{C-RHO.111}. Note that, by convention, we have that the value  $s=(\log_{2}C_\varrho)^{-1}$ is only permissible when $q=\infty$ and $C_\varrho>1$. In what follows, let $\varrho_\#$ be the regularized quasi-metric given by Theorem~\ref{DST1}.

Focusing first on proving Theorem~\ref{DoubMeasINT}(a),
fix a $\rho$-ball $B_0:=B_\rho(y,R)$, where $y\in X$ and $R\in(0,\infty)$,
and suppose that $B:=B_\rho(x,r)\subset B_0$ is a $\rho$-ball
with $r\in(0,R]$. If $r>{\rm diam}_\rho(X)$ then \eqref{Doub-2-eh} is trivially
satisfied for any $\kappa\in(0,1]$. Thus, we can assume that
$r\leq{\rm diam}_\rho(X)$. Let $\{r_j\}_{j\in\mathbb{N}}$ and $\{u_j\}_{j\in\mathbb{N}}$
be as in Lemma~\ref{HolderBump} associated to the ball $B$.
Then, for any $j\in\mathbb{N}$,  $u_j\in M^s_{p,q}(\sigma B_0,\rho,\mu)$
which, in turn, implies that $u_j$ satisfies \eqref{doubsob}
with $B_0=B_\rho(y,R)$. Moreover, since $B_{\varrho_{\#}}(x,r_j)\subset B\subset B_0$,
it follows from the properties listed in (b)-(e) of Lemma~\ref{HolderBump}
that, for any $j\in\mathbb{N}$,
\begin{equation}\label{xu-12}
R^s\|u_j\|_{\dot{M}^s_{p,q}(\sigma B_0)}
\lesssim\frac{R^s2^j}{r^s}\lf[\mu(B_{\varrho_\#}(x,r_j))\r]^{1/p}
\quad\mbox{and}\quad\Vert u_j\Vert_{L^{p}(\sigma B_0)}\lesssim
\lf[\mu(B_{\varrho_\#}(x,r_j))\r]^{1/p}.
\end{equation}
Moreover, since $u_j\equiv1$ on $B_{\varrho_\#}(x,r_{j+1})$, we have
\begin{equation}\label{xu-14}
\Vert u_j\Vert_{L^{p^\ast}(B_0)}\geq \lf[\mu(B_{\varrho_{\#}}(x,r_{j+1}))\r]^{1/p^*}.
\end{equation}
Combining \eqref{doubsob}, \eqref{xu-12}, and \eqref{xu-14}, we conclude that
\begin{equation*}
\begin{split}
\lf[\mu(B_{\varrho_{\#}}(x,r_{j+1}))\r]^{1/p^*}
&\lesssim\frac{C_S}{[\mu(\sigma B_0)]^{s/Q}}\left(\frac{2^jR^s}{r^s}+1\right)\lf[\mu(B_{\varrho_\#}(x,r_j))\r]^{1/p}
\\
&\lesssim \frac{2^jR^s}{r^s}\lf\{\frac{[\mu(B_{\varrho_\#}(x,r_j))]^{1/p}}{[\mu( B_0)]^{s/Q}}\r\},\quad\forall\,j\in\mathbb{N},\nonumber
\end{split}
\end{equation*}
where the last inequality follows from the fact that $r\leq R$ and $2^j,\,\sigma\geq1$.
Therefore, there exists a positive constant  $C'$, independent of $x$, $r$, and $j$, such that
\begin{equation}
\label{xu-15}
\lf[\mu(B_{\varrho_{\#}}(x,r_{j+1}))\r]^{1/p^*}
\leq
\left\{\frac{C'R^s}{r^{s}[\mu(B_0)]^{s/Q}}\right\} 2^{j}\lf[\mu(B_{\varrho_{\#}}(x,r_j))\r]^{1/p},
\quad\forall\,j\in\mathbb{N}.
\end{equation}
Since $0<\delta r<r_j<r<\infty$ for any $j\in\mathbb{N}$, we
invoke Lemma~\ref{iteration} with the quasi-metric $\varrho_{\#}$, and
$$
p:=p,\quad
t:=p^*,\quad
\text{and}
\quad
\theta:=\frac{C'R^s}{r^{s}[\mu(B_0)]^{s/Q}}
$$
to conclude that
\begin{equation}
\label{xu-16}
\mu(B_\rho(x,r))\geq \mu(B_{\varrho_{\#}}(x,r_1))\gtrsim 2^{-\frac{Q^2}{s^2p}}
\left\{\frac{C'R^s}{r^{s}[\mu(B_0)]^{s/Q}}\right\}^{-Q/s}.
\end{equation}
This implies \eqref{Doub-2-eh}, and then finishes the proof of Theorem~\ref{DoubMeasINT}(a).

Moving on, we now establish  (b), (c), and (d) of the present theorem,
under the additional assumption that the space $(X,\rho)$ is uniformly perfect. In this
context, let $\lambda\in(0,1)$ be as in \eqref{U-perf} and recall that there is no loss in
generality by assuming that $\lambda<(C_\rho\widetilde{C}_\rho)^{-2}$.
	
To prove the statement in (b), fix $x,\,y\in X$
and $r,\, R\in(0,\infty)$ such that $B:=B_\rho(x,r)\subset B_\rho(y,R)$ and $r\leq R$. In
light of Lemma~\ref{en2-4}(ii), we may assume that $r\leq C_\rho\varphi^x_{\rho}(r)/\lambda^2$. Moreover,
as the proof of Theorem~\ref{DoubMeasINT}(a), we can also assume that $r\leq{\rm diam}_\rho(X)$.
Then, by Lemma~\ref{HolderBump}, there exist $\{r_j\}_{j\in\mathbb{N}}\subset (0,\fz)$ and  a collection
$\{u_j\}_{j\in\mathbb{N}}$  of
functions having the properties ({a})-({f}) listed in Lemma~\ref{HolderBump}.
In particular,  $u_j\in M^s_{p,q}(\sigma B_0,\rho,\mu)$ for any $j\in\mathbb{N}$,
where $B_0:=B_\rho(y,R)$. Hence, by the assumption of (b) in this theorem, for any $j\in\nn$,
$u_j$ satisfies \eqref{doubpoin}.
Observe that, by Lemma~\ref{HolderBump}(e),  for any $j\in\mathbb{N}$, we have (keeping in mind that $\sigma\geq1$)
\begin{equation}
\label{Gkw-924-2}
\frac{C_P}{[\mu(\sigma B_0)]^{s/Q}}\,R^{s}\|u_j\|_{\dot{M}^s_{p,q}(\sigma B_0)}
\lesssim \frac{2^{j}R^s}{r^{s}[\mu(B_0)]^{s/Q}}\,\lf[\mu(B_{\varrho_\#}(x,r_j))\r]^{1/p}.
\end{equation}
Moreover,  Lemma~\ref{HolderBump}(f) implies
\begin{equation}
\label{Gkw-925-2}
\inf_{\gamma\in\mathbb{R}}\Vert u_j-\gamma\Vert_{L^{p^\ast}(B_0)}\ge \inf_{\gamma\in\mathbb{R}}\Vert u_j-\gamma\Vert_{L^{p^\ast}(E_j^\gamma)} \geq\frac{1}{2}
\lf[\mu(B_{\varrho_{\#}}(x,r_{j+1}))\r]^{1/p^*},
\end{equation}
where $E_j^\gamma$ is as in Lemma~\ref{HolderBump}(f).
In concert, by \eqref{Gkw-925-2}, \eqref{doubpoin}, and \eqref{Gkw-924-2}, we find that
\begin{equation*}
\frac{1}{2}\lf[\mu(B_{\varrho_{\#}}(x,r_{j+1}))\r]^{1/p^*}
\lesssim
\left\{\frac{R^s}{r^{s}[\mu(B_0)]^{s/Q}}\right\} 2^{j}\,\lf[\mu(B_{\varrho_\#}(x,r_j))\r]^{1/p},
\quad\forall\ j\in\mathbb{N}.
\end{equation*}
Hence, \eqref{xu-15} holds true. At this stage, we can argue as in the proofs of
\eqref{xu-15} and \eqref{xu-16} to conclude that $\mu$ satisfies the desired
estimate   \eqref{Doub-2-eh}. This finishes the proof of the statement in Theorem~\ref{DoubMeasINT}(b).

In order to prove Theorem~\ref{DoubMeasINT}(c), we need to show that, for any fixed
$\varepsilon\in(0,\infty)$, there exists a positive constant $\kappa$ such that
\begin{equation}\label{measboundthm-52-X}
\kappa\lf(\frac{r}{R}\r)^{Q+\varepsilon}\leq\frac{\mu(B_\rho(x,r))}{\mu(B_\rho(y,R))},
\end{equation}
whenever $x,\,y\in X$ and $0<r\leq R<\infty$ satisfy $B_\rho(x,r)\subset B_\rho(y,R)$.
Note that the estimate in \eqref{measboundthm-52-X} will follow once
we prove that, for any  $\beta\in(1,\infty)$, there exists a
$\kappa_\beta\in(0,\infty)$ satisfying
\begin{equation}\label{ME13-52}
\frac{\mu(B_\rho(x,r))}{\mu(B_\rho(y,R))}
\geq \kappa_\beta\lf(\frac{r}{R}\r)^{\beta s/(\beta-1)},
\end{equation}
whenever $x,\,y\in X$ and $0<r\leq R<\infty$ satisfy $B_\rho(x,r)\subset B_\rho(y,R)$.
To this end, fix $\beta\in(1,\infty)$ and suppose  $B:=B_\rho(x,r)\subset
B_\rho(y,R)=:B_0$ for some $x,\, y\in X$ and $0<r\leq R<\infty$.
As in the proof of (b), it suffices to consider the case
when $r\leq C_\rho\varphi^x_{\rho}(r)/\lambda^2$ and $r\leq{\rm diam}_\rho(X)$, given Lemma~\ref{en2-4}.
Now let $\{r_j\}_{j\in\mathbb{N}}$ and $\{u_j\}_{j\in\mathbb{N}}$ be as in
Lemma~\ref{HolderBump} (applied here to the ball $B_\rho(x,r)$). Then
then we still have $u_j\in M^s_{p,q}(\sigma B_0,\rho,\mu)$
and $\|u_j\|_{\dot{M}^s_{Q/s,q}(\sigma B_0)}>0$ for any $j\in\mathbb{N}$.
As such, for any $j\in\nn$,  $u_j$ satisfies \eqref{amv-1} with $B_0$. That is,
\begin{equation}
\label{amv-1-32}
\mvint_{B_0} {\rm exp}\left(c_1\frac{[\mu(\sigma B_0)]^{s/Q}|u_j-(u_j)_{B_0}|}{R^s\|u_j\|_{\dot{M}^s_{Q/s,q}(\sigma B_0)}}\right)^\omega\,d\mu\leq c_2,
\quad\forall\,j\in\mathbb{N}.
\end{equation}
Observe that   Lemma~\ref{HolderBump}(e) implies
\begin{equation*}
0<\|u_j\|_{\dot{M}^s_{Q/s,q}(\sigma B_0)}\lesssim 2^{j}r^{-s}\lf[\mu(B_{\varrho_{\#}}(x,r_j))\r]^{s/Q},\quad\forall\ j\in\mathbb{N}.
\end{equation*}
Thus, by $\sigma\geq1$, we have
$$
\frac{[\mu(\sigma B_0)]^{s/Q}|u_j-(u_j)_{B_0}|}{R^s\|u_j\|_{\dot{M}^s_{Q/s,q}(\sigma B_0)}}
\geq\frac{C\mu(B_0)^{s/Q}r^s|u_j-(u_j)_{B_0}|}{R^s2^{j}\,[\mu(B_{\varrho_{\#}}(x,r_{j}))]^{s/Q}},\quad\forall\,j\in\mathbb{N},
$$
for some constant $C\in(0,\infty)$ independent of $x$, $y$, $r$, $R$, and $j$.
Combining this with Lemma~\ref{HolderBump}(f) and \eqref{amv-1-32}, we find that	
\begin{align}
\label{Iue.5-2-X-52}
\frac{\mu\big(B_{\varrho_{\#}}(x,r_{j+1})\big)}{\mu(B_\rho(y,R) )}\,
{\rm exp}\left(\frac{C\mu(  B_0)^{s/Q}r^s}{R^s2^{j+1}\,\mu(B_{\varrho_{\#}}(x,r_{j})\big)^{s/Q}}\right)^\omega\leq C_2.
\end{align}
By increasing the constant $C_2$, we may assume that $C_2>1$.
As such, using the elementary estimate
$$\log(z)\leq \beta
Q(s\omega)^{-1}\,z^{s\omega/\beta Q},\quad \forall\ z\in(0,\infty),$$
a rewriting of \eqref{Iue.5-2-X-52} implies
\begin{align*}
\frac{C[\mu(B_\rho(y,R))]^{s/Q}r^s}{R^s2^{j+1}\,[\mu(B_{\varrho_{\#}}(x,r_{j}))]^{s/Q}}
&\leq\lf\{\log\lf(C_2\frac{\mu(B_\rho(y,R))}{\mu(B_{\varrho_{\#}}(x,r_{j+1}))}\r)\r\}^{1/\omega}\\
&\leq \lf[\beta Q(s\omega)^{-1}\r]^{1/\omega}\,C_2^{s/(\beta Q)}\left[\frac{\mu(B_\rho(y,R))}{\mu(B_{\varrho_{\#}}(x,r_{j+1}))}\right]^{s/(\beta Q)}.
\end{align*}
Therefore,
\begin{equation*}
\lf[\mu(B_{\varrho_{\#}}(x,r_{j+1}))\r]^{s/(\beta Q)}\lesssim
\left\{\frac{R^s [\mu(B_\rho(y,R))]^{\frac{s(1-\beta)}{Q\beta}}}{r^s}\right\}\,2^{j}
\,\lf[\mu(B_{\varrho_{\#}}(x,r_{j})\r]^{s/Q},\quad \forall\ j\in\nn.
\end{equation*}
From  this, applying Lemma~\ref{iteration} with  the quasi-metric $\varrho_{\#}$, and
$$
p:=Q/s,\quad
t:=\beta Q/s,\quad
\text{and}
\quad
\theta:=\frac{R^s[\mu(B_\rho(y,R))]^{\frac{s(1-\beta)}{Q\beta}}}{Cr^s},
$$
we deduce that
$$
\mu(B_\rho(x,r))\geq\mu(B_{\varrho_{\#}}(x,r_1))\gtrsim
\left\{\frac{R^s[\mu(B_\rho(y,R))]^{\frac{s(1-\beta)}{Q\beta}}}{Cr^s}\right\}^{\frac{-\beta Q}{s(\beta-1)}}\,
2^{\frac{-\beta^2 Q}{s(\beta-1)^2}},
$$
which implies
$$
\frac{\mu(B_\rho(x,r))}{\mu(B_\rho(y,R))}\gtrsim \,\left(\frac{r}{R}\right)^{Q\beta/(\beta-1)}.
$$
This finishes the proof of \eqref{ME13-52} and hence, the proof of Theorem~\ref{DoubMeasINT}(c).

Concerning the statement in Theorem~\ref{DoubMeasINT}(d), fix $x,\,y\in X$ and
$0<r\le R<\fz$  such that $B:=B_\rho(x,r)\subset B_\rho(y,R)$.
If $B_\rho(x,r)=X$, then $B_\rho(y,R)= X$ because
$B_\rho(x,r)\subset B_\rho(y,R)$, and \eqref{Doub-2-eh} easily
follows with any choice of $\kappa\in(0,1]$. Thus, in what follows,
 we assume that $X\setminus B_\rho(x,r)\neq\emptyset$.
 Since $\varrho\approx\rho$ and $\varrho_{\#}\approx\varrho$ (see Theorem~\ref{DST1}),
 we have $\varrho_{\#}\approx\rho$ and hence there exists a positive constant $c_0$ such that
$$
c_0^{-1}\varrho_{\#}(z,w)\leq\rho(z,w)\leq c_0\,\varrho_{\#}(z,w),\quad\forall\, z,\,w\in X.
$$
Let $r_0:=r/c_0$. Then $B_{\varrho_{\#}}(x,r_0)\subset B_\rho(x,r)$
and $X\setminus B_{\varrho_{\#}}(x,r_0)\neq\emptyset$. Moreover, since $(X,\rho)$
is uniformly perfect and $\varrho_{\#}\approx \rho$, we find that $(X,\varrho_{\#})$
also satisfies the uniformly perfect property with  constant $\lambda':=c_0^{-2}\lambda<\lambda$.
Thus, we may select a point $x_0\in B_{\varrho_{\#}}(x,r_0)\setminus B_{\varrho_{\#}}(x,\lambda' r_0)$.

Next, fix any finite number $\alpha\in [s,(\log_{2}C_\varrho)^{-1}]$
such that $\alpha\neq s$ unless $s=(\log_{2}C_\varrho)^{-1}$, and consider
the function $\Phi_{0,\lambda' r_0}\colon X\to[0,1]$ given by Lemma~\ref{GVa2}.
Then,  by Lemma~\ref{GVa2}, $\Phi_{0,\lambda' r_0}$ belongs to $M^s_{p,q}(X,\rho,\mu)$ and
\begin{equation}\label{estt}
\lf\|\Phi_{0,\lambda' r_0}\r\|_{\dot{M}^s_{p,q}(X)}
\lesssim \lf[(\lambda' r_0)^\alpha-0^\alpha\r]^{-s/\alpha}\lf[\mu(B_{\varrho_\#}(x,\lambda' r_0))\r]^{1/p}
\lesssim (\lambda' r_0)^{-s}\lf[\mu(B_{\varrho_\#}(x,\lambda' r_0))\r]^{1/p}.
\end{equation}
Consider the ball $B_0:=B_\rho(y,R)$. Clearly,	 $\Phi_{0,\lambda' r_0}\in M^s_{p,q}(\sigma B_0,\rho,\mu)$.
Then, by the assumption of (d) of this theorem,  $\Phi_{0,\lambda' r_0}$ satisfies the inequality in \eqref{doubHold}
with the ball $B_0$. Note also that the choice of $x_0$ ensures that
$x_0\in B_{\varrho_{\#}}(x,r_0)\subset B\subset B_0$.
	
Now, by the choice  $x_0\in B_{\varrho_{\#}}(x,r_0)\setminus B_{\varrho_{\#}}(x,\lambda' r_0)$
and the fact that $\Phi_{0,\lambda' r_0}$ is supported in $B_{\varrho_{\#}}(x,\lambda' r_0)$ (due to Lemma~\ref{GVa2}),
it follows from \eqref{doubHold} (used here with the ball $B_0$) and \eqref{estt} that
\begin{align*}
1&=\lf|\Phi_{0,\lambda' r_0}(x)-\Phi_{0,\lambda' r_0}(x_0)\r|\\
&\ls   [\rho(x,x_0)]^{s-Q/p}\frac{R^{Q/p}}{[\mu(\sigma B_0)]^{1/p}}\,\|\Phi_{0,\lambda' r_0}\|_{\dot{M}^s_{p,q}(X)}\\
&\lesssim [\rho(x,x_0)]^{s-Q/p}\frac{R^{Q/p}}{[\mu(B_0)]^{1/p}}\,(\lambda' r_0)^{-s}\lf[\mu(B_{\varrho_\#}(x,\lambda' r_0))\r]^{1/p}\\
&\lesssim r_0^{-Q/p}\frac{R^{Q/p}}{[\mu(B_\rho(y,R))]^{1/p}}\,\lf[\mu(B_{\varrho_\#}(x,\lambda' r_0))\r]^{1/p}\\
&\lesssim \left(\frac{R}{r}\right)^{Q/p}\left[\frac{\mu(B_\rho(x, r))}{\mu(B_\rho(y,R))}\right]^{1/p},
\end{align*}
where the second inequality follows from the fact that $\sigma\geq1$
and the last inequality is deduced from the definition of $r_0$ as well
as the fact that $\lambda'<1$ and $B_{\varrho_{\#}}(x,r_0)\subset B_\rho(x,r)$.
The desired estimate in \eqref{Doub-2-eh} now follows. This finishes the proof of (d) and,
in turn, the proof of Theorem~\ref{DoubMeasINT}.	
\end{proof}

Finally, in the case of doubling measures, we have the following characterization which is a consequence of Theorem~\ref{DOUBembedding}(b) and Theorem~\ref{DoubMeasINT}(c).

\begin{theorem}\label{AS-U-BDD}
Let $(X,\rho,\mu)$ be a uniformly perfect  quasi-metric measure space, $\sigma\in[C_\rho,\infty)$,
and $q\in(0,\infty]$. Also, assume $s\in(0,\infty)$
satisfies $s\preceq_q{\rm ind}\,(X,\rho)$. Then the following two statements are equivalent.
\begin{enumerate}[label=\rm{(\alph*)}]
\item The measure $\mu$ is doubling.
\vskip.08in
		
\item There exist positive constants $c_1$, $c_2$, $p$, and $\omega$ such that
\begin{equation}
\label{eq53}
\mvint_{B_0} {\rm exp}\left(c_1\frac{[\mu(\sigma B_0)]^{1/p}|u-u_{B_0}|}{R_0^s\|u\|_{\dot{M}^s_{p,q}(\sigma B_0)}}\right)^\omega\,d\mu\leq c_2,
\end{equation}
whenever $B_0\subset X$ is a $\rho$-ball of radius $R_0\in(0,\infty)$ and $u\in \dot{M}^{s}_{p,q}(\sigma B_0,\rho,\mu)$ satisfies $\|u\|_{\dot{M}^s_{p,q}(\sigma B_0)}>0$.
\end{enumerate}
\end{theorem}
\begin{proof}
In light of Theorem~\ref{DoubMeasINT}(c), we immediately have
that the measure $\mu$ is doubling whenever \eqref{eq53} holds true. To show
the implication (a) $\Rightarrow$  (b) in Theorem~\ref{AS-U-BDD},
we take $p:=s^{-1}\log_2C_{\rm doub}$, where
\begin{equation*}
C_{\rm doub}:=\sup_{x\in X,\, r\in(0,\infty)}\frac{\mu(B_\rho(x,2r))}{\mu(B_\rho(x,r))}\in(1,\infty)
\end{equation*}
is the doubling constant for $\mu$. Then $\mu$ is $Q$-doubling with $Q:=\log_2C_{\rm doub}$
(in the sense of \eqref{Doub-2-eh}). By this and Theorem~\ref{DOUBembedding}(b), we obtain
\eqref{eq53}. This finishes the proof of Theorem \ref{AS-U-BDD}.
\end{proof}

\section{Triviality of Triebel--Lizorkin and Besov  Spaces}\label{sec-triv}

Given a quasi-metric measure space $(X,\rho,\mu)$, it follows
from Lemma~\ref{GVa2} that $\dot{M}^s_{p,q}(X,\rho,\mu)$ and
$\dot{N}^s_{p,q}(X,\rho,\mu)$ contain plenty of nonconstant
functions whenever $p\in(0,\infty)$, $q\in(0,\infty]$, and
$s\in(0,\infty)$ satisfies $s\preceq_q{\rm ind}\,(X,\rho)$.
In this section, we will highlight the fact that, if the
smoothness parameter $s$ is too large, then the spaces $\dot{M}^s_{p,q}(X,\rho,\mu)$
and $\dot{N}^s_{p,q}(X,\rho,\mu)$ may be trivial, in the sense that they only
contain constant functions. As it turns out, the range of $s$ for which these
spaces are trivial is directly related to the geometry of the quasi-metric space
$(X,\rho)$. The latter concept is quantified by the notion of  {H\"older} {index} from
\cite{MMMM13}. Recall that the {\it H\"older} {\it index} of the quasi-metric space $(X,\rho)$
is defined as
\begin{align}\label{VVV-AGBVdef}
{\rm ind}_H(X,\rho)&:=
\inf\Bigg\{\alpha\in(0,\infty):\,\forall\ x,\,y\in X
\,\,\mbox{ and }\,\,
\forall\ \varepsilon\in(0,\infty),\,\,\,\,\exists\ N\in\nn\ \,
\mbox{and} \nonumber
\\[-2pt]
&\quad \lf. \xi_1,...,\xi_{N+1}\in X\
\mbox{ such that $\xi_1=x$, $\xi_{N+1}=y$ and }
\sum_{i=1}^N[\rho(\xi_i,\xi_{i+1})]^\alpha<\varepsilon\r\}
\end{align}
with the agreement that $\inf\emptyset:=\infty$. This index is
purely geometrical and the terminology of ``H\"older index"
used for \eqref{VVV-AGBVdef} is justified by the fact that
\begin{equation}\label{VVV-AGBV}
{\rm ind}_H(X,\rho)=\sup\,\lf\{\alpha\in(0,\infty):\,
\dot{\mathscr{C}}^\alpha(X,\rho)\not={\mathbb{R}}\r\}\in(0,\infty],
\end{equation}
which follows from \cite[p.\,215, Theorem~4.59]{MMMM13}.

The main result of this section is as follows.

\begin{theorem}
\label{trivial}
Let $(X,\rho,\mu)$ be a quasi-metric measure space and
suppose that there exist positive constants $\kappa$ and $Q$ satisfying
\begin{equation}
\label{rtw-578}
\kappa\,r^Q\leq\mu(B_\rho(x,r))\quad\,\,\mbox{for any }\ x\in X
\ \mbox{ and any finite }\ r\in(0,{\rm diam}_\rho(X)].
\end{equation}
Then both of the spaces $\dot{M}^s_{p,q}(X,\rho,\mu)$ and
$\dot{N}^s_{p,q}(X,\rho,\mu)$ are trivial whenever
$s\in({\rm ind}_H(X,\rho),\infty)$, $p\in (Q/[s-{\rm ind}_H(X,\rho)],\infty)$, and $q\in(0,\infty]$.
\end{theorem}

\begin{proof}
Let $s\in ({\rm ind}_H(X,\rho),\infty)$,
$p\in (Q/[s-{\rm ind}_H(X,\rho)],\infty)$, and $q\in(0,\infty]$.
Then $s-Q/p>{\rm ind}_H(X,\rho)$ and hence, by \eqref{VVV-AGBV},
we have $\dot{\mathscr{C}}^{s-Q/p}(X,\rho)=\mathbb{R}$. Moreover,
since $p>Q/s$, Theorem~\ref{LBembedding} implies that $\dot{M}^s_{p,q}(X,\rho,\mu)\hookrightarrow\dot{\mathscr{C}}^{s-Q/p}(X,\rho)=\rr$, that is,
 $\dot{M}^s_{p,q}(X,\rho,\mu)$ is trivial.

To prove that $\dot{N}^s_{p,q}(X,\rho,\mu)$ is also trivial, take
$\varepsilon\in ({\rm ind}_H(X,\rho),s)$ so that
$p>Q/[\varepsilon-{\rm ind}_H(X,\rho)]$. Then $p>Q/\varepsilon$.
Since $\mu$ satisfies \eqref{rtw-578}, we can
argue as in the proof of \eqref{eq30-LB} (using Theorem~\ref{mainembedding-epsilon}
in place of Theorem~\ref{embedding}) to show
that $\dot{N}^s_{p,q}(X,\rho,\mu)\hookrightarrow\dot{\mathscr{C}}^{s-Q/p}(X,\rho)=\rr$ is valid.
This finishes the proof of Theorem \ref{trivial}.
\end{proof}

\begin{example} It follows from  \cite[Corollary 4.36]{MMMM13} that, if $n\in\mathbb{N}$, then
\begin{equation*}
{\rm ind}\,(\mathbb{R}^n,|\,\cdot-\cdot\,|)={\rm ind}_H(\mathbb{R}^n,|\,\cdot-\cdot\,|)=1,\quad
{\rm ind}\,([0,1]^n,|\,\cdot-\cdot\,|)={\rm ind}_H([0,1],|\,\cdot-\cdot\,|)=1,
\end{equation*}
where $|\,\cdot-\cdot\,|$ is the standard Euclidean distance.
In fact, more generally,   if $(X,\|\,\cdot\,\|)$ is a
nontrivial normed vector space, then ${\rm ind}\,(Y,\|\,\cdot\,\|)={\rm ind}_H(Y,\|\,\cdot\,\|)=1$
for any convex subset $Y\subset X$ of cardinality $\geq2$. In this context, if $\mu$ is a $Q$-lower
Ahlfors-regular measure on $X$ for some $Q\in(0,\infty)$,
then Theorem~\ref{trivial} implies that
$\dot{M}^s_{p,q}(Y)$ and $\dot{N}^s_{p,q}(Y)$ are trivial whenever $s\in(1,\infty)$, $p\in(Q/(s-1),\infty)$,
and $q\in(0,\infty]$, and $\dot{M}^1_{p,\infty}(Y)$, $\dot{N}^1_{p,\infty}(Y)$, $\dot{M}^s_{p,q}(Y)$,
and $\dot{N}^s_{p,q}(Y)$ are all nontrivial whenever $s\in(0,1)$, $p\in(0,\infty)$, and $q\in(0,\infty]$.
\end{example}

\begin{example}
Recall that a  quasi-metric space $(X,\rho)$ is said to be {\it pathwise connected}
if, for any pair of points $x,\,y\in X$, there exists a continuous path
$f\colon [0,1]\to(X,\tau_\rho)$ with $f(0)=x$ and $f(1)=y$, where $\tau_\rho$ represents
the canonical topology induced by the quasi-metric	$\rho$ on $X$.
If   $(X,\rho,\mu)$ is a pathwise connected quasi-metric measure space  equipped with a
nonnegative measure $\mu$ on $X$ satisfying the following upper Ahlfors-regularity
condition for some $Q\in(0,\infty)$: there exists a positive constant $c$ such that
\begin{equation*}
\mu(B_\rho(x,r))\leq c r^Q \quad \mbox{for any }\ x\in X\
\mbox{ and any finite }\ r\in (0,{\rm diam}_\rho\,(X)],
\end{equation*}
then ${\rm ind}\,(X,\rho)\leq{\rm ind}_H(X,\rho)\leq Q$. In this case,
Theorem~\ref{trivial} implies that  $\dot{M}^s_{p,q}(X)$ and
$\dot{N}^s_{p,q}(X)$ are trivial whenever $s\in(Q,\infty)$, $p\in (Q/[s-Q],\infty)$,
and $q\in(0,\infty]$.

A particular case of the above setting which is worth mentioning
 is the graph   $\Sigma$ of a real-valued Lipschitz function defined
in $\mathbb{R}^{n-1}$. In this
case, one has that $(\Sigma, |\cdot-\cdot|, \mathcal{H}^{n-1}\lfloor_\Sigma)$
is a pathwise connected Ahlfors-regular space of dimension $(n-1)$,
where $\mathcal{H}^{n-1}\lfloor_\Sigma$ denotes the $(n-1)$-dimensional Hausdorff
measure restricted to $\Sigma$. Hence, in this context, ${\rm ind}_H(\Sigma,\rho)\leq(n-1)$
and Theorem~\ref{trivial} implies that $\dot{M}^s_{p,q}(\Sigma)$ and $\dot{N}^s_{p,q}(\Sigma)$
are trivial whenever $s\in(n-1,\infty)$, $p\in(\frac{n-1}{s-(n-1)},\infty)$, and $q\in(0,\infty]$.
\end{example}

The triviality of the spaces $\dot{M}^s_{p,q}(X)$ and $\dot{N}^s_{p,q}(X)$ cannot,
in general, be expected in the absence connectivity. Indeed, if $(X,\rho,\mu)$
is a quasi-metric measure space containing an ``island", in the sense that there exist
a point $x_0\in X$ and a number $r_0\in(0,\infty)$ satisfying
$$X\setminus B_\rho(x_0,r_0)\neq\emptyset\quad\mbox{and}\quad{\rm dist}_\rho\lf(B_\rho(x_0,r_0),X\setminus B_\rho(x_0,r_0)\r)>0,
$$
then it is easy to check that $u:={\bf 1}_{B_\rho(x_0,r_0)}\in \dot{M}^s_{p,q}(X)
\cap\dot{N}^s_{p,q}(X)$ for all $s,\,p\in(0,\infty)$ and $q\in(0,\infty]$.

\bigskip

\noindent Ryan Alvarado  (Corresponding author)

\medskip

\noindent Department of Mathematics and Statistics, Amherst College, Amherst, MA, USA

\smallskip

\noindent{\it E-mail:} \texttt{rjalvarado@amherst.edu}
\bigskip

\noindent Dachun Yang  and Wen Yuan

\medskip

\noindent Laboratory of Mathematics and Complex Systems (Ministry of Education of China),
School of Mathematical Sciences, Beijing Normal University, Beijing 100875, People's Republic of China

\smallskip

\noindent{\it E-mails:} \texttt{dcyang@bnu.edu.cn} (D. Yang)

\noindent\phantom{{\it E-mails:} }\texttt{wenyuan\@@bnu.edu.cn} (W. Yuan)

\addcontentsline{toc}{section}{References}
\begin{thebibliography}{99}

\bibitem{ad}
R. A. Adams, Sobolev Spaces,
Pure and Applied Mathematics  65, Academic Press, New York-London, 1975.


\vspace{-0.3cm}

\bibitem{agh20} R. Alvarado, H. G\'orka and P. Haj\l asz,
Sobolev embedding for $M^{1,p}$ spaces is equivalent
to a lower bound of the measure, J. Funct.
Anal. 279 (2020), 108628, 39 pp.


\vspace{-0.3cm}

\bibitem{AM15}  R.~Alvarado and M.~Mitrea, Hardy Spaces on
Ahlfors-Regular Quasi Metric Spaces, Lecture Notes in Mathematics 2142, Springer, New York, 2015.


\vspace{-0.3cm}

\bibitem{AMM13} R. Alvarado, I. Mitrea and M. Mitrea,
Whitney-type extensions in quasi-metric spaces, Commun. Pure Appl. Anal. 12 (2013),   59-88.


\vspace{-0.3cm}

\bibitem{AWYY21}  R.~Alvarado, F.~Wang, D.~Yang and W.~Yuan,
Pointwise characterization of Besov and Triebel--Lizorkin spaces on spaces of homogeneous type, Submitted or arXiv: 2201.10196.


\vspace{-0.3cm}

\bibitem{AYY21}  R.~Alvarado, D.~Yang and W.~Yuan, A measure characterization
of embedding and extension domains for Sobolev, Triebel--Lizorkin, and Besov spaces
on spaces of homogeneous type, Submitted.


\vspace{-0.3cm}


\bibitem{cheeger}
J.~Cheeger, Differentiability of Lipschitz functions on metric measure spaces,
 Geom.\ Funct.\ Anal. 9 (1999), 428-517.

\vspace{-0.3cm}
\bibitem{CoWe71} R.R.\,Coifman and G.\,Weiss,
Analyse Harmonique Non-Commutative sur Certains Espaces Homogenes,
Lecture Notes in Mathematics 242, Springer-Verlag, 1971.

\vspace{-0.3cm}

\bibitem{CoWe77} R.R.\,Coifman and G.\,Weiss,
Extensions of Hardy spaces and their use in analysis,
Bull. Amer. Math. Soc., 83 (1977),  569-645.

\vspace{-0.3cm}

\bibitem{DaSe97}
G.\,David and S.\,Semmes,
Fractured Fractals and Broken Dreams:\ Self-similar Geometry through Metric and Measure,
Oxford Lecture Series in Mathematics and its
Applications 7, The Clarendon Press, Oxford University Press, New York, 1997.


\vspace{-0.3cm}

\bibitem{EG}
L.C.~Evans and R.F.~Gariepy,
Measure Theory and Fine Properties of Functions,
Studies in Advanced Mathematics, CRC Press, Boca Raton, FL, 1992.

\vspace{-0.3cm}

\bibitem{gag}
E.~Gagliardo,
Propriet\`a di alcune classi di funzioni in pi\`u variabili,
Ricerche Mat. 7 (1958), 102-137.


\vspace{-0.3cm}

\bibitem{GKZ13} A.~Gogatishvili, P.~Koskela and Y.~Zhou,
Characterizations of Besov and Triebel--Lizorkin
spaces on metric measure spaces,
Forum Math. 25 (2013),   787-819.

\vspace{-0.3cm}

\bibitem{gorka}
P.~G\'orka, In metric-measure spaces Sobolev embedding
is equivalent to a lower bound for the measure, Potential Anal. 47 (2017), 13-19.

\vspace{-0.3cm}

\bibitem{hajlasz}
P.~Haj\l{}asz, Sobolev spaces on metric-measure spaces,
In: Heat Kernels and Analysis on Manifolds, Graphs, and Metric Spaces (Paris, 2002),  173-218,
Contemp. Math.  338, Amer. Math. Soc., Providence, RI, 2003.

\vspace{-0.3cm}

\bibitem{hajlasz2}
P.~Haj\l{}asz, Sobolev spaces on an arbitrary metric space,
Potential Anal. 5 (1996), 403-415.

\vspace{-0.3cm}

\bibitem{SMP}
P.~Haj\l{}asz and P.~Koskela,
Sobolev met Poincar\'e. Mem.\ Amer.\ Math.\ Soc. 145 (2000), no. 688.

\vspace{-0.3cm}

\bibitem{SMP2}
P.~Haj\l{}asz and P.~Koskela,
Sobolev Meets Poincar\'e,
C. R. Acad.\ Sci.\ Paris S\'er.\ I Math. 320 (1995), 1211--1215.


\vspace{-0.3cm}

\bibitem{hajlaszkt1}
P. Haj\l{}asz, P. Koskela and H. Tuominen,
Sobolev embeddings, extensions and measure density condition,
J. Funct. Anal. 254 (2008), 1217-1234.

\vspace{-0.3cm}

\bibitem{hajlaszkt2}
P.~Haj\l{}asz, P.~Koskela, H.~Tuominen,
Measure density and extendability of Sobolev functions,
Rev. Mat. Iberoam 24 (2008), 645-669.


\vspace{-0.3cm}

\bibitem{hhhpl21} Y. Han, Y. Han, Z. He, C. Pereyra and J. Li,
Geometric characterization of embedding theorems -- for Sobolev,
Besov, and Triebel--Lizorkin spaces on spaces
of homogeneous type -- via orthonormal wavelets, J. Geom. Anal.
31 (2021), 8947-8978.

\vspace{-0.3cm}

\bibitem{HaLuYa99i} Y.\,Han, S.\,Lu  and D.\,Yang,
Inhomogeneous Besov and Triebel-Lizorkin spaces on spaces of homogeneous type,
Approx. Theory Appl. (N.S.), 15 (1999),  37-65.

\vspace{-0.3cm}

\bibitem{HaLuYa99ii}
Y.\,Han, S.\,Lu  and D.\,Yang,
Inhomogeneous Triebel-Lizorkin spaces on spaces of homogeneous type,
Math. Sci. Res. Hot-Line, 3 (1999),   1-29.

\vspace{-0.3cm}

\bibitem{HaMuYa08} Y.\,Han, D.\,M\"{u}ller and D.\,Yang,
A theory of Besov and Triebel-Lizorkin spaces on metric measure spaces modeled on Carnot-Carath\'eodory spaces,
Abstr. Appl. Anal. 2008, Art. ID 893409, 250 pp.

\vspace{-0.3cm}

\bibitem{HaYa03} Y.\,Han and D.\,Yang,
Some new spaces of Besov and Triebel-Lizorkin type
on homogeneous spaces, Studia Math. 156 (2003), 67-97.

\vspace{-0.3cm}

\bibitem{HaYa02} Y.-S.\,Han and D.\,Yang,
New characterizations and applications of
inhomogeneous Besov and Triebel-Lizorkin spaces on homogeneous type spaces and fractals,
Dissertationes Math. (Rozprawy Mat.)  403 (2002), 102 pp.



\vspace{-0.3cm}

\bibitem{hebey}
E. Hebey,
 Sobolev Spaces on Riemannian Manifolds,
Lecture Notes in Mathematics 1635, Springer-Verlag, Berlin, 1996.

\vspace{-0.3cm}

\bibitem{HIT16} T.~Heikkinen, L.~Ihnatsyeva
and H.~Tuominen, Measure density and extension
of Besov and Triebel--Lizorkin functions,
J. Fourier Anal. Appl. 22 (2016), 334-382.



\vspace{-0.3cm}

\bibitem{HK21} T. Heikkinen and N. Karak,
Orlicz--Sobolev embeddings, extensions and Orlicz--Poincar\'e inequalities,
J. Funct. Anal. 282 (2022), Paper No. 109292.

\vspace{-0.3cm}

\bibitem{HKT17}
T. Heikkinen, P. Koskela and H.   Tuominen,
Approximation and quasicontinuity of Besov and Triebel--Lizorkin functions,
Trans. Amer. Math. Soc. 369 (2017),  3547-3573.

\vspace{-0.3cm}

\bibitem{Hu03}
J. Hu,  A note on Haj\l asz--Sobolev spaces on fractals,
J. Math. Anal. Appl. 280 (2003), 91-101.

\vspace{-0.3cm}

\bibitem{Jon81} P.W.~Jones,
Quasiconformal mappings and extendability of functions in Sobolev spaces,
Acta Math., 47 (1981), 71-88.


\vspace{-0.3cm}

\bibitem{Karak1}
N.~Karak, Lower bound of measure and embeddings
of Sobolev, Besov and Triebel--Lizorkin spaces, Math. Nachr. 293 (2020),  120-128.

\vspace{-0.3cm}

\bibitem{Karak2}
N.~Karak, Measure density and embeddings of Haj{\l}asz--Besov
and Haj{\l}asz--Triebel--Lizorkin spaces,
J. Math. Anal. Appl. 475 (2019), 966-984.

\vspace{-0.3cm}

\bibitem{korobenko}
L.~Korobenko,
Orlicz--Sobolev inequalities and the doubling condition,
Ann. Fenn. Math. 46 (2021), 153-161.


\vspace{-0.3cm}
\bibitem{korobenkomr}
L.~Korobenko, D.~Maldonado and C.~Rios,
From Sobolev inequality to doubling,
Proc. Amer. Math. Soc. 143 (2015), 4017-4028.


\vspace{-0.3cm}

\bibitem{KYZ11} P.~Koskela, D.~Yang and Y.~Zhou, Pointwise characterizations of Besov and Triebel--Lizorkin spaces and quasiconformal mappings, Adv. Math. 226 (2011), 3579-3621.


\vspace{-0.3cm}

\bibitem{MS}
O.~Martio and J.~Sarvas,
Injectivity theorems in plane and space,
Ann. Acad. Sci. Fenn. Ser. A I Math. 4 (1979),   383-401.


\vspace{-0.3cm}

\bibitem{MMMM13} D. Mitrea, I. Mitrea, M. Mitrea and S. Monniaux,
Groupoid Metrization Theory. With Applications to Analysis on Quasi-Metric Spaces
and Functional Analysis, Applied and Numerical Harmonic Analysis, Birkh\"auser/Springer,
New York, 2013.



\vspace{-0.3cm}

\bibitem{Niren}
L.~Nirenberg,
On elliptic differential equations, Ann. Scuola Norm. Pisa (III) 13 (1959), 1-48.


\vspace{-0.3cm}

\bibitem{shanmugalingam}
N.~Shanmugalingam,
Newtonian spaces: an extension of Sobolev spaces to metric measure spaces.
Rev.\ Mat.\ Iberoamericana 16 (2000), 243-279.


\vspace{-0.3cm}

\bibitem{ST95}
W. Sickel and H. Triebel, H\"older inequalities and sharp embeddings in
function spaces of $B^s_{p,q}$ and $F^s_{p,q}$ type,
Z. Anal. Anwendungen 14 (1995), 105-140.


\vspace{-0.3cm}

\bibitem{sob36}
S.\,L.~Sobolev,
On the estimates relating to families of functions having derivatives that are square integrable,
Dokl. Akad. Nauk SSSR 1 (1936), 267-270 (Russian).

\vspace{-0.3cm}

\bibitem{sob38}
S.\,L.~Sobolev,
On a theorem of functional analysis,
Mat. Sbornik 46 (1938), 471-497 (Russian).

\vspace{-0.3cm}

\bibitem{T92} H.\,Triebel,
Theory of Function Spaces. II,
Monographs in Mathematics 84, Birkh\"{a}user Verlag, Basel, 1992.


\vspace{-0.3cm}

\bibitem{T06} H.~Triebel, Theory of Function Spaces. III,
Monographs in Mathematics  100, Birkh\"auser Verlag, Basel, 2006.

\vspace{-0.3cm}

\bibitem{trud}
N.~Trudinger,
On imbeddings into Orlicz spaces and some applications,
J. Math. Mech. 17 (1967), 473-483.


\vspace{-0.3cm}

\bibitem{Y03} D.~Yang, New characterizations of Haj\l{}asz--Sobolev spaces on metric spaces, Sci. China Ser. A 46 (2003), 675-689.


\vspace{-0.3cm}

\bibitem{y17} W. Yuan,  Besov and Triebel--Lizorkin spaces on metric spaces: embeddings and pointwise multipliers, J. Math. Anal. Appl. 453 (2017),  434-457.

\vspace{-0.3cm}

\bibitem{z15}
 Y. Zhou, Fractional Sobolev extension and imbedding,
Trans. Amer. Math Soc. 367 (2015), 959-979.


\end{thebibliography}
\end{document}